\documentclass[11pt]{amsart}

\usepackage[T1]{fontenc}
\usepackage[utf8]{inputenc}
\usepackage{textcomp}
\usepackage{lmodern}
\usepackage{microtype}

\usepackage{amssymb}
\usepackage{amsthm}
\usepackage{amscd}

\usepackage{mathscinet}
\usepackage[breaklinks=true]{hyperref}
\usepackage{url}
\usepackage[shortlabels]{enumitem}

\numberwithin{equation}{section}

\def\Ind#1#2{#1\setbox0=\hbox{$#1x$}\kern\wd0\hbox to 0pt{\hss$#1\mid$\hss}
\lower.9\ht0\hbox to 0pt{\hss$#1\smile$\hss}\kern\wd0}
\def\Notind#1#2{#1\setbox0=\hbox{$#1x$}\kern\wd0\hbox to 0pt{\mathchardef
\nn="3236\hss$#1\nn$\kern1.4\wd0\hss}\hbox to 0pt{\hss$#1\mid$\hss}\lower.9\ht0
\hbox to 0pt{\hss$#1\smile$\hss}\kern\wd0}

\makeatletter
\newtheorem*{rep@theorem}{\rep@title}
\newcommand{\newreptheorem}[2]{%
\newenvironment{rep#1}[1]{%
 \def\rep@title{#2 \ref{##1}}%
 \begin{rep@theorem}}%
 {\end{rep@theorem}}}
\makeatother
\newtheorem{theorem}{Theorem}[section]

\newtheorem{lemma}[theorem]{Lemma}
\newtheorem{prop}[theorem]{Proposition}
\newreptheorem{prop}{Proposition}

\newtheorem{cor}[theorem]{Corollary}
\newtheorem{fac}[theorem]{Fact}

\newtheorem{problem}[theorem]{Problem}
\newtheorem{thmx}{Theorem}

\theoremstyle{definition}
\newtheorem{definition}[theorem]{Definition}

\newtheorem{remark}[theorem]{Remark}

\newtheorem{example}[theorem]{Example}

\newtheorem{claim}[theorem]{Claim}
\newtheorem*{claim*}{Claim}

 \newenvironment{claimproof}{\begin{proof}}{\end{proof}}

  \def \B{\operatorname{\mathcal{B}}}

   \def \Br{\operatorname{Br}}

  \def \VC{\operatorname{VC}}
    \def \Perf{\operatorname{Perf}}
    \def \ima{\operatorname{Im}}

        \def \Tree{\operatorname{Tree}}
                \def \supp{\operatorname{supp}}

        \def \lex{\operatorname{lex}}
        
                \def \Lad{\operatorname{Lad}}

        \def \vdisc{\operatorname{vdisc}}
  
    \def \cU{\operatorname{\mathcal{U}}}

\allowdisplaybreaks[2]

\begin{document}

\title[Perfect stable regularity and slice-wise stable hypergraphs]{Perfect stable regularity lemma and slice-wise stable hypergraphs}
\author{Artem Chernikov and Henry Towsner}
\date{\today}

\begin{abstract}
We investigate various forms of (model-theoretic) stability for hypergraphs and their corresponding strengthenings of the hypergraph regularity lemma with respect to partitions of vertices. On the one hand, we provide a complete classification of the various possibilities in the ternary case.
On the other hand, we provide an example of a family of slice-wise stable 3-hypergraphs so that for no partition of the vertices, any triple of parts has density close to 0 or 1.
In particular, this addresses some questions and conjectures of Terry and Wolf. We work in the general measure theoretic context of graded probability spaces, so all our results apply both to measures in ultraproducts of finite graphs, leading to the aforementioned combinatorial applications, and to commuting definable Keisler measures, leading to applications in model theory.
 \end{abstract}

\maketitle

\tableofcontents

\section{Introduction}
\subsection{Variants of stable regularity for hypergraphs}
This paper contributes to the emerging subject of tame hypergraph regularity and its connections to model theoretic classification.
Several recent results demonstrate that better bounds and stronger forms of regular partitions 
can be obtained for families of hypergraphs satisfying certain combinatorial restrictions.

Recall that a graph $G = (V,E)$ is $d$-stable if it omits an induced copy of the bipartite half-graph with both sides of size $d$ (i.e.~there are no $a_i, b_i \in V, 1 \leq i \leq d$ so that $(a_i,b_j) \in E \iff i \leq j$). In the case of stable graphs, it is demonstrated in \cite{malliaris2014regularity} that for every $d\in\mathbb{N}$ and $\varepsilon>0$, every finite $d$-stable graph $G=(V,E)$ satisfies \emph{stable regularity}: there is a regular partition $\mathcal{P}$ of $V$ of size $N = N(d,\varepsilon)$ without any irregular pairs, i.e.~the density between any pair of sets in $\mathcal{P}$ is in $[0,\varepsilon) \cup (1 - \varepsilon, 1]$ (this corresponds to Definition \ref{def: Intro stable reg versions}(3) with $k=2$ below). Some superficially stronger variations on this result have since appeared in the literature. One variant, which we call \emph{strong stable regularity}, analogous to the strong regularity lemma for general graphs \cite{alon2007efficient}, was considered recently in \cite{terry2021irregular, chavarria2021continuous}. In this version, the value $\varepsilon$ is allowed to depend on the size of the partition---that is, instead of fixing a value $\varepsilon$ in advance, we fix a function $f: \mathbb{N} \to (0,1]$, and the density between any pair of sets in $\mathcal{P}$ is restricted to $[0,f(N))\cup (1-f(N),1]$, where $N$ is the size of $\mathcal{P}$. (The precise definition of this variant appears as Definition \ref{def: Intro stable reg versions}(2) with $k=2$ below.) Also a measure-theoretic variant in the context of finitely additive measures in the ultraproduct of a family of finite graphs was considered in \cite{malliaris2016stable, chernikov2016definable, coregliano2022countable}. We consider this here in a general countably additive setting, and provide a finitary characterization as Definition \ref{def: Intro stable reg versions}(1) with $k=2$, which we call \emph{approximately perfect stable regularity}. No distinction was made in the literature between these three versions of stable regularity for graphs, which is a posteriori explained by the formal equivalence of all three versions for hereditarily closed families of graphs (see Theorem \ref{thm: Intro stab reg equiv for graphs}) with each other and (when stated in a suitable uniform way) with stability.

It turns out that, when generalized to $3$-hypergraphs, the third notion is strictly stronger, as we will explain below. (Whether the first two remain equivalent for $k$-hypergraphs remains open.) The main goal of this paper is to clarify this distinction and examine the connection between tame hypergraph regularity properties and notions of stable $3$-hypergraphs.

As noted in \cite{chernikov2020hypergraph} and \cite{terry2021irregular}, there are two natural ways to lift a binary property like stability to a $3$-hypergraph. Given a $3$-hypergraph $(V,E)$ with $E \subseteq V^3$, we can say that it is \emph{partition-wise} stable\footnote{This is the most commonly used definition, and is often referred to as simply ``stable''.} if $E$ is stable when viewed as a binary relation on $V^2\times V$, and that $(V,E)$ is \emph{slice-wise} stable if each of the slice graphs $E_v=\{(x,y)\mid (v,x,y)\in E\}$ is stable\footnote{More precisely, the $E_v$ should be uniformly stable; see Definition \ref{def:slice-wise}. This notion is called \emph{weakly stable} in \cite{terry2021irregular}.}.

Applying the arguments for graphs directly to $3$-hypergraphs shows that partition-wise stability implies any of the notions of tame hypergraph regularity described above, and that any of these notions implies slice-wise stability. An example from \cite{terry2021irregular}, which we analyze in detail in Section \ref{sec:example}, satisfies strong stable regularity but fails to be partition-wise stable. In Section \ref{sec: counterexample}, we give a new example which is slice-wise stable but fails to satisfy stable regularity. Putting these results together, we get the following string of implications for hypergraphs:
\begin{gather*}
  \text{slice-wise stable} \Leftarrow\text{Stable regularity}\leftarrow\text{Strong stable regularity}\\
   \Leftarrow\text{partition-wise Stable}=\text{Approximately perfect stable regularity,}
\end{gather*}
where $\Leftarrow$ indicates a strict implication while the implication denoted $\leftarrow$ is not known to be strict (see Section \ref{sec: summary of results} for details).

This leaves open the question of whether there is a combinatorial characterization of stable regularity or strong stable regularity. To make progress towards this question, we consider a slightly more refined setting: we consider tripartite $3$-hypergraphs, $(X,Y,Z,E)$ where $E\subseteq X\times Y\times Z$. Such a $3$-hypergraph might be partition-wise or slice-wise stable in some directions, but not in others. The bulk of this paper analyzes the various combinations of slice-wise and partition-wise stability and identifies corresponding tame regularity properties. In particular, we show that slice-wise stability (in all directions) together with partition-wise stability in \emph{one} direction suffices to give strong stable regularity. (Such an asymmetric property cannot be necessary, however, so a precise combinatorial characterization, if there is one, remains open.)

One current approach to tame hypergraph regularity is to use guesses at the right notion of tame hypergraph regularity as a guide to identifying new combinatorial notions. The results in this paper suggest that (approximately) perfect stable regularity, rather than ordinary stable regularity, may be the notion that leads to more natural combinatorial properties for hypergraphs.

\subsection{History of the problem}
Several recent results demonstrate that better bounds and stronger forms of regular partitions 
can be obtained for families of hypergraphs satisfying certain combinatorial restrictions.
The first results of this kind were obtained for graphs of finite VC-dimension by Alon, Fischer, Newman \cite{alon2007efficient}, Lov\'asz, Szegedy \cite{MR2815610}, and qualitatively in the context of model theory of NIP structures (i.e.~structures in which all definable families of sets have finite VC-dimension) by Hrushovski, Pillay, Simon \cite[Lemma 1.6]{hrushovski2013generically}, demonstrating that for every $d \in \mathbb{N}$ and $\varepsilon > 0$, for every finite graph $G=(V,E)$ with the edge relation of VC-dimension at most $d$, there exists a regular partition $\mathcal{P}$ of $V$ with $|\mathcal{P}| = \left( 1/\varepsilon \right)^{O(d^2)}$ and the density of the edges in $[0, \varepsilon) \cup (1-\varepsilon, 1]$ for all but $\varepsilon$ fraction of the pairs of sets in $\mathcal{P}$.

The two extremal special cases of NIP structures correspond to stable (discussed above) and to distal structures. 
In \cite{fox2012overlap,fox2016polynomial} it is demonstrated that when the edge relation
is semi-algebraic, of bounded description complexity, then additionally $\mathcal{P}$ can be chosen so that all good pairs in $\mathcal{P}$ are  homogeneous, i.e.~ the density is actually \emph{equal} to $0$ or $1$, and the sets in $\mathcal{P}$ can be chosen to be semi-algebraic, of bounded
complexity. In \cite{chernikov2018regularity} these results 
were generalized to graphs definable in arbitrary distal structures, showing moreover that this strong form of regularity characterizes the model theoretic notion  of distality (see also \cite{simon2016note} for a  simplified proof). This applies to graphs definable in the reals, $p$-adics,  or graphs of bounded expansion/twin-width \cite{przybyszewski2023distal}. An analytic version of distal regularity  is established in \cite{anderson2023generically}.
In an orthogonal direction (i.e.~outside of the finite VC-dimension context), polynomial bounds on the size of the partition were obtained by Tao \cite{tao2015expanding} for algebraic
hypergraphs of bounded description complexity in large finite fields of growing characteristic, generalized to fixed characteristic in \cite{pillay2013remarks} and to hypergraphs in \cite{chevalier2022algebraic}.

In the current paper we are focused on generalizations of stability and  corresponding strong regular decompositions of vertices to hypergraphs. 
A systematic study was initiated in \cite{chernikov2016definable} (which was published as \cite{chernikov2021definable}), providing the first generalization from graphs of finite VC-dimension (equivalently, NIP) to hypergraphs. 
It is demonstrated in \cite{chernikov2016definable} that if a $k$-ary (uniform) hypergraph $H = (E;V)$ is \emph{partition-wise NIP}, meaning that viewed as a binary relation between $V$ and $V^{k-1}$ it is NIP (equivalently, the family of subsets of $V$ given by the corresponding fibers with $k-1$ vertices fixed has finite VC-dimension), then there exists a partition $\mathcal{P}$ of $V$ with $|\mathcal{P}| = (1/\varepsilon)^{O((k-1)d^2)}$ and the density of the edges in $[0, \varepsilon) \cup (1-\varepsilon, 1]$ for all but $\varepsilon$ fraction of the $k$-tuples of sets in $\mathcal{P}$. Stronger bounds on the size of such partitions were obtained in \cite{fox2019erdHos}. In \cite{chernikov2020hypergraph} it is demonstrated that in fact such partitions can be obtained under a weaker assumption of \emph{slice-wise NIP} on the hypergraph, namely that for every $(a_1, \ldots, a_{k-2}) \in V^{k-2}$, the corresponding fiber $E_{(a_1, \ldots, a_{k-2})} \subseteq V^2$ has VC-dimension $\leq d$ (and that for a hereditarily closed family of hypergraphs, the converse holds as well).  Related questions on which combinatorial conditions 
 allow to reduce to binary cylinder intersection sets (i.e.~partitions of $V\times V$, rather than unary sets, i.e.~partitions of the vertices $V$) have also attracted attention recently. In \cite{chernikov2020hypergraph} we establish the first instance of tame hypergraph regularity taking into account partitions not only of vertices, but also of hyperedges of various arities. Specifically, we demonstrate that given a $(k+1)$-hypergraph of finite slice-wise $\VC_{k}$-dimension (a generalization of the usual $\VC$-dimension implicit in Shelah, and explicitly considered in \cite{chernikov2019n}), equivalently omitting a fixed finite
$(k + 1)$-partite $(k + 1)$-uniform hypergraph, can be approximated by $k$-ary
cylinder intersection sets.  And in \cite{terry2021irregular}, a generalization of stability to $3$-hypergraphs was proposed and characterized by the existence of certain regular partitions. Improved bounds in the partition-wise VC$_2$ case were given in \cite{MR4662634}. 

\subsection{Summary of the results}\label{sec: summary of results}

Similarly to finite VC-dimension/NIP, stability is a property of binary relations, and there are multiple natural ways to generalize it to hypergraphs. Being focused on partitions of vertices, we will consider the following here (see Section \ref{sec: mu-stability} for a detailed discussion of these and related notions in a measure-theoretic context):
\begin{definition} Given $d \in \mathbb{N}$ and a $k$-partite $k$-hypergraphs of the form $H = (E;X_1, \ldots, X_k)$ with $E \subseteq \prod_{i = 1}^{k} X_i$, we will say:
	\begin{enumerate}
		\item $H$ is \emph{partition-wise $d$-stable} if for every $1 \leq \ell \leq k$, $E$ viewed as a binary relation on $X_{\ell} \times \left(\prod_{i \in [k] \setminus \{\ell\}} X_i \right)$ is $d$-stable.
		\item $H$ is \emph{slice-wise $d$-stable} if for every $1 \leq \ell < m \leq k$ and for every $b = \left(b_i : i \in [k] \setminus \{\ell, m \} \right) \in \prod_{i \in [k] \setminus \{\ell, m\}} X_i$, the binary relation $E_{b} = \left\{ (b_\ell, b_m) \in X_{\ell} \times X_{m} : \left(b_i : i \in [k] \right) \in E \right\} \subseteq X_{\ell} \times X_{m}$ is $d$-stable.
	\end{enumerate}
\end{definition}
These represent the strongest and the weakest symmetric notions with respect to interactions with the sets of vertices, respectively.  Partition-wise stable hypergraphs were considered in \cite{chernikov2016definable, ackerman2017stable}, and slice-wise stable hypergraphs in \cite{terry2021irregular} 
 (under the name of \emph{weak stability}). We will also consider intermediate situations, e.g.~simultaneous partition-wise stability for some partition of the variables, but only slice-wise stability for the other ones.
 
 In the context of partite hypergraphs, the aforementioned three variants of the stable regularity lemma can be formalized  as follows (where (1) implies (2) implies (3), see Section \ref{sec: transfer for stable regularities} for a detailed discussion):
\begin{definition}\label{def: Intro stable reg versions}
Let $\mathcal{H}$ be a  family of finite $k$-partite $k$-hypergraphs of the form $H = (E;X_1, \ldots, X_k)$ with $E \subseteq \prod_{i = 1}^{k} X_i$ and $X_i$ finite. We say that $\mathcal{H}$ satisfies:
	\begin{enumerate}
		
	\item  \emph{approximately perfect stable regularity} if for every $\varepsilon \in \mathbb{R}_{>0}$ and every function $f: \mathbb{N} \to (0,1)$ there exists some $N = N(f,\varepsilon)$ so that: for any $H \in \mathcal{H}$ there exists $N' \leq N$ and partitions $X_i = \bigsqcup_{0 \leq t < N'} A_{i,t}$  so that:
	\begin{enumerate}
	\item $\left \lvert A_{i,0} \right \rvert \leq \varepsilon \left \lvert X_{i} \right \rvert$ for all $1 \leq i \leq k$;
		\item for any $1 \leq i \leq k$ and any $ 1 \leq t < N'$, we have $\frac{\lvert A_{i,t} \cap E_{b} \rvert}{\lvert A_{i,t} \rvert} \in [0, f(N')) \cup (1- f(N'), 1]$ for all but at most an $f(N')$ fraction of tuples $b \in \prod_{j \neq i} X_j$.
	\end{enumerate}

		\item \emph{strong stable regularity} if for every $\varepsilon \in \mathbb{R}_{>0}$ and every function $f: \mathbb{N} \to (0,1]$ there exists some $N = N(f,\varepsilon)$ so that: for any $H \in \mathcal{H}$ there exists $N' \leq N$ and partitions $X_i = \bigsqcup_{0 \leq t < N'} A_{i,t}$  so that:
	\begin{enumerate}
		\item $\left \lvert A_{i,0} \right \rvert \leq \varepsilon \left \lvert A_{i,1} \right \rvert$ for all $1 \leq i \leq k$;
		\item for any $ 1 \leq t_1, \ldots, t_k < N'$ we have 
		\begin{gather*}
			\frac{\left \lvert E \cap \left( A_{1,t_1} \times \ldots \times  A_{k,t_k} \right) \right \rvert }{\left \lvert A_{1,t_1} \times \ldots \times  A_{k,t_k}  \right \rvert } \in [0, f(N')) \cup (1 - f(N'), 1].
		\end{gather*}
	\end{enumerate}
	
	\item \emph{stable regularity} if for every $\varepsilon \in \mathbb{R}_{>0}$ there exists some $N = N(\varepsilon)$ such that: for any $H = (E;X_1, \ldots, X_k) \in \mathcal{H}$ there exists $N' \leq N$ and partitions $X_i = \bigsqcup_{1 \leq t < N'} A_{i,t}$  so that 
	for any $ 1 \leq t_1, \ldots, t_k \leq N'$ we have 
		\begin{gather*}
			\frac{\left \lvert E \cap \left( A_{1,t_1} \times \ldots \times  A_{k,t_k} \right) \right \rvert }{\left \lvert A_{1,t_1} \times \ldots \times  A_{k,t_k}  \right \rvert } \in [0, \varepsilon) \cup (1 - \varepsilon, 1].
		\end{gather*}
\end{enumerate}
\end{definition}

The first generalization of stable regularity to hypergraphs was established in \cite{chernikov2016definable}. Namely, in our terminology here, it is demonstrated that if $\mathcal{H}$ is the family of  partition-wise $d$-stable $k$-hypergraphs for some fixed $d$, then it satisfies stable regularity. Similar result was obtained  in 
 the finitary case in  \cite{ackerman2017stable}. 

As our first result here, we refine and complement this for partition-wise stable hypergraphs, demonstrating that partition-wise stability gives a precise characterization of the existence of perfect $0-1$ density countable partitions in the limit (see Theorem \ref{thm: metastable general equivalences} for a more precise version):

\begin{thmx}\label{thm: Intro stab reg equiv for graphs}
	Let $\mathcal{H}$ be a hereditarily closed family of finite $k$-partite $k$-hypergraphs. Then the following are equivalent:
	\begin{enumerate}
		\item $\mathcal{H}$ satisfies approximately perfect stable regularity;
		\item there exists some $d \in \mathbb{N}$ so that every $H \in \mathcal{H}$ is partition-wise $d$-stable;
		\item  any ultraproduct of $\widetilde{H} = (\widetilde{E}, \widetilde{X}_1, \ldots, \widetilde{X}_k)$ of hypergraphs from $\mathcal{H}$ satisfies \emph{perfect stable regularity}, i.e.~there exist countable partitions $A_{i,t}, t \in \mathbb{N}$ of $X_i$, for each $1 \leq i \leq k$, so that $A_{i,t}$ is in the algebra generated by the internal sets and for any $t_1, \ldots, t_k \in \mathbb{N}$ the density of $\widetilde{E}$ on $A_{1,t_1} \times \ldots \times A_{k,t_k}$ is precisely $0$ or $1$.
	\end{enumerate}
\end{thmx}

As mentioned  above, we show that in the case of graphs (i.e.~when $k=2$) all three versions of stable regularity coincide (see Theorem \ref{thm: all versions of reg coincide for stable graphs} for a more precise version):

	\begin{thmx}\label{thm: Intro stable reg equiv for graphs}	Let $\mathcal{H}$ be a hereditarily closed family of finite bipartite graphs. The following are equivalent:
	\begin{enumerate}
		\item any ultraproduct of the graphs in $\mathcal{H}$   satisfies perfect stable regularity;
		\item the family $\mathcal{H}$ satisfies  approximately perfect stable regularity;

		\item the family $\mathcal{H}$ satisfies  strong stable regularity;

		\item the family $\mathcal{H}$ satisfies stable regularity;

		\item there is some $d \in \mathbb{N}$ so that all graphs in $\mathcal{H}$ are $d$-stable.
	\end{enumerate}
\end{thmx}

  On the other hand, for 3-hypergraphs we demonstrate that stable regularity (this notion is called $\vdisc_3$ error in the terminology of \cite{terry2021irregular}) is strictly weaker than the approximately perfect version:
 \begin{thmx}[see Example \ref{ex: x=y<z stability props}]\label{thm: Intro meta neq strong for 3-hyper}
 	There exists a hereditarily closed family of 3-hypergraphs satisfying strong stable regularity, but not approximately perfect stable regularity.
 \end{thmx} 
 
 Then we consider the case of partition-wise stability in one direction and slice-wise stability in the opposite direction. We show that this assumption is still sufficient to obtain strong stable regularity (but not approximately perfect regularity, since the  example in 
 Theorem \ref{thm: Intro meta neq strong for 3-hyper} satisfies these assumptions).  See Corollary \ref{cor: finitary one dir stable other slice-wise} for a more precise version:
\begin{thmx}\label{thm: Intro partition-wise and slice-wise}
Given $d \in \mathbb{N}$, let $\mathcal{H}$ be the family of all finite $3$-partite $3$-hypergraphs such that for every $H = (E; X,Y,Z) \in \mathcal{H}$, $E$ viewed as a binary relation on $X\times Y$ and $Z$ is $d$-stable, and for every $z \in Z$, the binary slice $E_z \subseteq X \times Y$ is $d$-stable.
	Then $\mathcal{H}$ satisfies strong stable regularity.
\end{thmx}
\noindent  These results provide a partial answer to \cite[Problem 2.14(1)]{terry2021irregular}. 
 
Moving on to slice-wise stable hypergraphs, we obtain the following countable partition result with actual 0-1 densities and no bad triads (this is a special case of Theorem \ref{thm: all three parts perfect rects}):
\begin{thmx}\label{thm: Intro slice-wise infinitary}
Let $d \in \mathbb{N}$ and let $\mathcal{H}$ be a family of $3$-partite $3$-hypergraphs all of which are slice-wise $d$-stable. Let a $H = (E; X, Y, Z)$ be an ultraproduct of hypergraphs from $\mathcal{H}$.
 Then there exist countable partitions $X \times Y = \bigsqcup_{i \in \omega} A^i$, $X \times Z = \bigsqcup_{j \in \omega}B^j$ and $Y \times Z = \bigsqcup_{k \in \omega} D^k$ with  $A^i, B^j, D^k$ in the algebra generated by internal sets, so that:
\begin{enumerate}
	\item each $A^i = A^{i,X} \times A^{i,Y}, B^j = B^{j,X} \times B^{j,Z}, D^k = D^{k,Y} \times D^{k,Z}$ is a rectangle with $A^{i,X},B^{j,X}, A^{i,Y}, D^{k,Y}, B^{j,Z}, D^{k,Z}$ all in the algebra generated by internal sets;
	\item \begin{enumerate}
 	\item for each $i,j \in \omega$,  the density of $E$ on  $A^i \land B^j$ is either $0$ or $1$;
 	\item for each $i,k \in \omega$, the density of $E$ on $A^i \land D^k$ is either $0$ or $1$;
 	\item for each $j,k \in \omega$,  the density of $E$ on $B^j \land D^k$ is either $0$ or $1$.

 \end{enumerate} 

\end{enumerate}
 \end{thmx}

Next we demonstrate that existence of such perfect partitions in the ultraproduct implies in particular a stronger form of  regular partitions with ``binary vdisc$_3$ error'' established for slice-wise stable hypergraphs in \cite{terry2021irregular}, and in fact both are equivalent to slice-wise stability for a hereditarily closed family (this is a special case of Theorem \ref{thm: finitary slice-wise reg}):

\begin{thmx}\label{thm: Intro finitary slice-wise reg}
	Let $\mathcal{H}$ be a hereditarily closed family of $3$-partite $3$-hypergraphs. Then the following are equivalent.
	\begin{enumerate}
		\item There exists $d \in \mathbb{N}$ so that every $H \in \mathcal{H}$ is slice-wise $d$-stable.
		\item Any ultraproduct  $\widetilde{H} = \left(\widetilde{E}; \widetilde{X}, \widetilde{Y}, \widetilde{Z} \right)$ of hypergraphs from $\mathcal{H}$ admits a partition as in Theorem \ref{thm: Intro slice-wise infinitary}.
\item (Strong slice-wise stable regularity) For every $\varepsilon > 0$ and $f: \mathbb{N} \to (0,1]$, there exists $N = N(\varepsilon,f) \in \mathbb{N}$ satisfying the following. For any $H = (E; X,Y,Z) \in \mathcal{H}$ there exists $N' \leq N$, partitions $X = \bigsqcup_{i \in [N']} A^i$, $Y = \bigsqcup_{j \in [N']} B^j$  and $Z = \bigsqcup_{k \in [N']} C^k$, and $\Sigma_{X,Y}, \Sigma_{X,Z}, \Sigma_{Y,Z} \subseteq [N']^2$ so that:
\begin{enumerate}
\item
 $\left \lvert \bigsqcup_{(i,j) \notin \Sigma_{X,Y}} A^i \times B^j \right \rvert < \varepsilon 	\left \lvert X \times Y \right \rvert$,\\ $ \left \lvert \bigsqcup_{(i,j) \notin \Sigma_{X,Z}} A^i \times C^k \right \rvert < \varepsilon \left \lvert X \times Z\right \rvert$ and \\
	$\left \lvert \bigsqcup_{(j,k) \notin \Sigma_{Y,Z}} B^j \times C^k \right \rvert < \varepsilon 	\left \lvert Y \times Z \right \rvert$;
	\item for all  
$(i,j,k) \in \left( \Sigma_{X,Y} \land \Sigma_{X,Z} \right) \cup \left(\Sigma_{X,Y} \land \Sigma_{Y,Z}  \right) \cup \left( \Sigma_{X,Z} \land \Sigma_{Y,Z} \right) $,
$\frac{\left \lvert E \cap \left( A^i \times B^j \times C^k \right) \right \rvert }{\left \lvert  A^i \times B^j \times C^k \right \rvert } \in \left[0, f(N') \right) \cup \left( 1 - f(N'), 1\right]$.
\end{enumerate}
\item (Slice-wise stable regularity) For every $\varepsilon >0$  there exists some $N = N(\varepsilon) \in \mathbb{N}$ satisfying the following. For any $H = (E; X,Y,Z) \in \mathcal{H}$ there exists $N' \leq N$ and partitions 
$X \times Y = \bigsqcup_{0 \leq i < N'} A^i$, $X \times Z = \bigsqcup_{0 \leq j< N'}B^j$ and $Y \times Z = \bigsqcup_{0 \leq k < N'} D^k$   so that:
	\begin{enumerate}
	\item $\left \lvert A^0 \right \rvert < \varepsilon \left \lvert X \times Y\right \rvert, \mu_{X \times Z} \left \lvert B^0 \right \rvert < \varepsilon  \left \lvert X \times Z \right \rvert,  \left \lvert D^0 \right \rvert < \varepsilon \left \lvert Y \times Z  \right \rvert$;
			\item for each $1 \leq i,j, k < N'$, each of 
	\begin{gather*}
		\frac{ \left \lvert E \land A^i \land B^j \right \rvert }{ \left \lvert  A^i \land B^j \right \rvert },
		\frac{\left \lvert E \land A^i \land D^k \right \rvert}{\left \lvert A^i \land D^k \right \rvert}, \frac{\left \lvert E \land B^j \land D^k \right \rvert}{\left \lvert B^j \land D^k \right \rvert }
	\end{gather*}
is in  $\left[0, \varepsilon\right) \cup \left(1 - \varepsilon, 1 \right]$.
	\end{enumerate}

	\end{enumerate}
\end{thmx}
\noindent In particular this answers \cite[Problem 2.14(2)]{terry2021irregular} (see Corollary \ref{cor: slice-wise stab binary vdisc3}).

Finally, we demonstrate that slice-wise stability does not imply stable regularity for $3$-hypergraphs (see Theorem \ref{thm: slice-wise counterexample}): 

\begin{thmx}\label{thm: Intro slice-wise counterex}
	There exists a hereditarily closed family of finite $3$-hypergraphs so that every $H \in \mathcal{H}$ is slice-wise $7$-stable, but $\mathcal{H}$ does not satisfy stable regularity.
\end{thmx}
\noindent In particular, this provides a counterexample to  \cite[Conjecture 2.42]{terry2021irregular}. 

\subsection{Choice of a setting and structure of the paper}
  
There are several different (but ultimately equivalent) formalisms in the literature for discussing these questions. In particular, the direction stating that graphs and hypergraphs satisfying tame regularity have the corresponding combinatorial property has to be formulated with the right amount of uniformity.

For instance, the sorts of theorems we prove here can be stated for hereditarily closed families of hypergraphs. To get equivalences, rather than just implications, they need to be stated with some uniformity: for instance, a hereditarily closed family of graphs $\mathcal{H}$ is $d$-stable for some $d$ (that is, every graph in $\mathcal{H}$ omits the ladder of size $d$) if and only if, for every $\varepsilon>0$, there is an $N$ so that every graph in $\mathcal{H}$ has a partition into at most $N$ pieces witnessing stable regularity.

This can be restated using ultraproducts: every ultraproduct of graphs from $\mathcal{H}$ is stable if and only if every ultraproduct of graphs from $\mathcal{H}$ admits partitions witnessing stable regularity. This is an equivalent statement; the ultraproduct machinery takes care of the uniformity for us.\footnote{Actually, in general both of these settings are slightly more complicated---one must either, as in \cite{terry2021irregular}, consider hereditarily closed families only up to a ``close to'' relation, or consider families additionally closed under blow-ups, or allow arbitrarily probability measures on them in addition to the counting measure. This paper can avoid this complication because we deal only with ladders, which have enough self-similarity for the equivalence to hold anyway.}

A third approach is to focus only on individual infinite measurable hypergraphs. We cannot hope to have an equivalence between regularity and stability in measurable graphs, because a graph which is unstable only on a part of measure zero should behave just like a stable measurable graph. The correct formulation here is to weaken stability to ``stability up to measure $0$'', which we call $\mu$-stability. (This is similar to the notion of robust stability introduced in \cite{martin2022stability}.)

This third setting seems particularly well-suited to some of the notions we consider in this paper: in this setting, the strongest version of stable regularity we consider  (``perfect stable regularity'') can be naturally stated in terms of a countable partition, with no approximations or error sets needed.

Since these formalisms are all fundamentally equivalent, our approach in this paper will be to work in the third setting and then use that to prove statements in the first two formalisms.

More precisely, we work here (as well as in \cite{chernikov2020hypergraph}) in the context of arbitrary graded probability spaces.  
 While this sometimes creates additional complications, as a benefit  all of the partitions that we obtain are (uniformly) measurable with respect to the algebra of the fibers of the edge relation $E$ of the hypergraph that we are decomposing. On the one hand, applied to ultraproducts of finite hypergraphs with Loeb measure, this results in the finitary partition results discussed above. On the other hand, our results apply to tuples of commuting definable Keisler measures in an arbitrary first-order theory, and the partitions that we obtain will be definable in the same theory --- so we get model-theoretic implications for higher amalgamation and uniqueness questions with respect to Keisler measures. We review graded probability spaces in Section \ref{sec: graded prob spaces}, and recall how to view them as first-order structures and  to take ultraproducts of graded probability spaces in order to carry out various transfer and compactness arguments in Section \ref{sec: ultraprods of graded prob spaces}.

As remarked above, stability is a combinatorial property, while regularity decompositions are of measure theoretic nature, and are not sensitive to adding or removing sets of measure $0$. So it is natural to relax stability from forbidding any induced half-graphs to allowing a measure zero set of half-graph. This results in the notion that we call \emph{$\mu$-stability}. We generalize some basic results about stability to $\mu$-stability in Section \ref{sec: mu-stability}\footnote{While this paper was in preparation, a similar notion was considered in \cite{martin2022stability} in the context of ultraproducts of finite sets under the name of \emph{rough stability}.}. A reader only interested in finite combinatorial applications and familiar with the definition of stability can skip this section.

 In Section \ref{sec: perfect stable reg for graphs} we present a streamlined  measure-theoretic account of qualitative stable regularity for graphs from \cite{malliaris2014regularity} and for partition-wise stable hypergraphs from \cite{chernikov2016definable}. We consider \emph{perfect} sets, i.e.~sets so that every fiber of the graph either almost contains it or is almost disjoint from it,  and their approximate version \emph{$\varepsilon$-decent} sets (see Definition \ref{def: perfect and decent}). The key point here is the existence of perfect sets of positive measure for stable graphs. In fact, in 
 Lemma \ref{lem: pos meas of perf sets for stable E} we show that for every $\mu$-stable graph there is a fixed Boolean combination of its fibers which defines a perfect set of positive measure, and moreover this happens for a positive measure set of the parameter choices (which will be needed in the study of slice-wise stable hypergraphs). Combined with the observation that the density of edges between any two perfect sets is either $0$ or $1$ (Lemma \ref{lem: 01 dens on perf sets}), this implies perfect stable regularity (Definition \ref{def: perf stab reg}) for partition-wise stable hypergraphs (Proposition \ref{prop: stable hypergraph reg}).

  In Section \ref{sec: transfer for stable regularities} we study various versions of stable regularity for hypergraphs, and prove Theorems \ref{thm: Intro stab reg equiv for graphs}, \ref{thm: Intro stable reg equiv for graphs} and \ref{thm: Intro meta neq strong for 3-hyper} discussed above. We review the basic relationship between the three versions of stable hypergraph regularity in Section \ref{sec: defs transf}.
To illustrate the difference between strong stable regularity and approximately perfect stable regularity, we consider a single example in some detail in Section \ref{sec:example}.
   The key tool is a transfer principle between perfect regularity with countable partitions in ultraproducts and our new notion of approximately perfect regularity for families of finite hypergraphs, utilizing type-definability of perfection and related properties in continuous logic developed in Section \ref{sec: subsec tranfer} (see Proposition \ref{prop: proof of transfer perf part}). Theorems \ref{thm: Intro stab reg equiv for graphs}, \ref{thm: Intro stable reg equiv for graphs} and \ref{thm: Intro meta neq strong for 3-hyper} are proved in Section \ref{sec: stab regs rels}.
   We prove Theorem \ref{thm: Intro stab reg equiv for graphs} (see Theorem  \ref{thm: metastable general equivalences}) and Theorem \ref{thm: Intro stable reg equiv for graphs} (see Theorem \ref{thm: all versions of reg coincide for stable graphs}) using this transfer principle for one implication and an adaptation of  the well-known observation that any bounded size partition of large half-graphs contains irregular pairs to show that hypergraphs which are not partition-wise stable do not admit partitions into perfect sets.
 We also prove an analogous transfer principle for strong stable regularity (Proposition \ref{prop: one way transfer for strong stable reg}), and use it (combined with Corollary \ref{cor: finitary one dir stable other slice-wise}) to prove Theorem  \ref{thm: Intro meta neq strong for 3-hyper}.
 Namely, in Example \ref{ex: x=y<z stability props} we show that the hereditarily closed family of $3$-partite $3$-hypergraphs on $[n]\times [n] \times [n], n \in \mathbb{N}$ defined by $x \leq y = z$ (pointed out in \cite{terry2021irregular} as an example satisfying stable regularity without being partition-wise stable) satisfies strong stable regularity, but does not satisfy approximately perfect stable regularity.

   In Section \ref{sec: one dir or slicewise stable} we begin investigating  which partitions could be obtained for $3$-hypergraphs that only satisfy slice-wise stability for some of the partitions of the variables, working towards Theorem \ref{thm: Intro slice-wise infinitary} in arbitrary graded probability spaces. Assume that $E \subseteq X \times Y \times Z$ satisfies: for almost every $x \in X$, the binary fiber $E_{x} \subseteq Y \times Z$ is stable. Of course, stable graph regularity (Proposition \ref{prop: stable hypergraph reg})  could be applied to every fiber separately to get perfect partitions of $Y$ and $Z$, however a priori these partitions might also vary in an unpredictable manner as $x$ varies. 
   As a first step, we demonstrate in Theorem \ref{lem: uniform perf decomp} that perfect partitions can be chosen in a uniformly measurable manner, i.e.~there exists a partition of $X \times Y$ into countably many measurable sets $A^i, i \in \omega$, so that for almost every $x \in X$, $\left(A^i_x : i \in  \omega \right)$ is a partition of $Y$ into countably many sets perfect for $E_x$. 
   
    In Section \ref{sec: two dirs of slicewise stab} we consider hypergraphs that are slice-wise stable in two directions.
First  we establish a  general ``maximal local symmetrization'' lemma, roughly saying that for an arbitrary binary relation $R$ we can choose countable partitions of the sides so that $R$ is almost contained in the diagonal, and this cannot be further refined without breaking the symmetry (see Lemma  \ref{lem: loc sym lemma} and Corollary \ref{cor: part of binary into symm sets} for the precise statement). Combining it with Theorem \ref{lem: uniform perf decomp}, we obtain an optimal partition for ternary relations that are slice-wise stable in two directions in Proposition \ref{prop: part XY into perf sets}: if the binary relation $E_{x}$ is $\mu$-stable for almost all $x \in X$, and $E_y$ is $\mu$-stable for almost all $y \in Y$, then there is a partition of $X \times Y$ into measurable sets perfect for $E$, viewed as a binary relation on $\left( X \times Y \right) \times Z$. In Theorem \ref{thm: two dirs slicewise stab finitary version} we extract a finitary equivalence for hereditarily closed families of finite hypergraphs using the transfer principle.

 In Section \ref{sec: all three dirs of slice-wise stab} we consider hypergraphs that are slice-wise stable in all three directions and prove Theorems \ref{thm: Intro slice-wise infinitary}, \ref{thm: Intro finitary slice-wise reg} in Section \ref{sec: slicewise stab reg proof} and Theorem  \ref{thm: Intro slice-wise counterex} in Section \ref{sec: counterexample}. 
Some further  analysis of infinite (and infinitely branching) trees of partitions, with infinite branches tackled by $\mu$-stability on various repartitions of coordinates and slices, allows us to improve the partition from Section \ref{sec: two dirs of slicewise stab} in Proposition \ref{prop: part perf sets and rects}: under the same assumption,  there exists a countable partition $X \times Y = \bigsqcup_{i \in \omega} A^i$ with each $A^i$ perfect for the relation $E \subseteq (X \times Y) \times Z$, and a countable partition $Y \times Z = \bigsqcup_{j \in \omega} B^j$ into rectangles $B^j = B^{j,Y} \times B^{j,Z}$ for some measurable $B^{j,Y} \subseteq Y,  B^{j,Z} \subseteq Z$, so that for each $i,j \in \omega$,  $A^i \land B^j$  is almost contained in $E$,  or is almost disjoint from it.

 \noindent Combining all of the above, after some further repartitions we obtain in Proposition \ref{prop: part perf rects} that if $E$ is slice-wise $\mu$-stable (in all three directions), then there exist a countable partition $X \times Y = \bigsqcup_{i \in \omega} A^i$ so that each $A^i$ is both perfect for the relation $E \subseteq (X \times Y) \times Z$, and  is a rectangle (along with some additional coherence conditions). Using this, Theorem  \ref{thm: Intro slice-wise infinitary} follows quickly (see Theorem \ref{thm: all three parts perfect rects}). Theorem \ref{thm: Intro finitary slice-wise reg} is then deduced using transfer principle, some additional repartitions and an adaptation of the necessary existence of irregular pairs in half-graphs (see Theorem \ref{thm: finitary slice-wise reg}).

In Section \ref{sec: counterexample} we prove Theorem \ref{thm: Intro slice-wise counterex}. We consider the ternary relation on (finite or infinite) words in the alphabet $\{0,1,2\}$ defined as follows: $(x,y,z) \in E$ if for the first  $n$ so that $\left \lvert \left\{ x(n), y(n), z(n) \right\} \right\rvert > 1$, we have $\left \lvert \left\{ x(n), y(n), z(n) \right\} \right\rvert  = 3$, i.e.~at the first coordinate where not all three of the sequences have the same entry, all three have different entries. Equivalently, this can be thought of as a ternary relation on branches of a complete ternary tree which holds when all three branches meet at the same vertex. We demonstrate that this relation is slice-wise stable (Lemma \ref{lem: the counterex is slice-wise stable}), but does not satisfy stable regularity (Theorem \ref{thm: slice-wise counterexample}). It is easy to see that such a partition cannot consist of sets determined by all words with a fixed prefix, and the bulk of the proof is a reduction from an arbitrary partition to one of this special form. Both example and its analysis are related to the analysis of $\textrm{GS}_{3}(n)$ hypergraphs in \cite[Section 7]{terry2021irregular}, however enjoy a quite different combination of properties.

  Finally, in Section  \ref{sec: one dir partition-wise one slice-wise} we establish a corresponding perfect measure-theoretic version of Theorem \ref{thm: Intro partition-wise and slice-wise}, and deduce Theorem \ref{thm: Intro partition-wise and slice-wise} from it using a transfer principle from Section \ref{sec: transfer for stable regularities}.

\subsection{Acknowledgements}
Chernikov was partially supported by the NSF CAREER grant
DMS-1651321 and by the NSF Research Grant DMS-2246598. Towsner was
partially supported by NSF Research Grant DMS-2054379. 
 
\section{Stability, partition-wise and slice-wise stability, measure stability}\label{sec: prelims on slice-wise stab}

\subsection{Some notation}
We have $\mathbb{N} = \{ 0,1, \ldots \}$ and $\omega$ is the first infinite ordinal.
Given a set $X$ and a family of its subsets $\mathcal{F} \subseteq \mathcal{P}(X)$, we let $\sigma_0 \left( \mathcal{F} \right)$ denote the Boolean algebra generated by $\mathcal{F}$, and $\sigma \left(\mathcal{F} \right)$ the $\sigma$-algebra generated by $\mathcal{F}$. For a set $A \subseteq X$, $\chi_A : X \to \{0,1\}$ denotes its characteristic function. For a $\sigma$-algebra $\mathcal{B} \subseteq 2^{X}$ on $X$, $\mu$ a measure on $\mathcal{B}$ and $A,B \in \mathcal{B}$, we will write $A \subseteq^0 B$ if $\mu \left(A \setminus B \right) = 0$ and $A =^0 B$ if $\mu \left( A \triangle B \right) = 0$, where $\triangle$ denotes the symmetric difference.

We use standard set-theoretic notation for $0$-$1$ sequences. Namely, for $n \in \omega$, we have $n = \{0,1,2, \ldots, n-1 \}$ (so $0 = \emptyset$), and $\omega = \{0,1,2, \ldots \}$. For $n \in \omega$, we write $2^n$ for the set of functions from $n$ to $2 = \{0,1\}$.  Equivalently, $\sigma \in 2^{n}$ is  naturally viewed as a $0$-$1$ sequence $\sigma = \left(\sigma(0), \sigma(1), \ldots, \sigma(n-1) \right)$ of length $n$. In particular, $2^{0} = \{ \emptyset \}$, where $\emptyset$ is the empty string. For $n \in \omega$, $2^{<n} := \bigcup_{i=0}^{n-1} 2^{i}$ is the set of $0$-$1$ sequences of length smaller than $n$, and $2^{\leq n} := \bigcup_{i=0}^{n} 2^{i}$ is the set of $0$-$1$ sequences of length at most $n$. In particular $2^{<0} = \emptyset$. Of course $|2^{n}| = 2^{n}$, $\left \lvert 2^{<n} \right \rvert = 2^{n}-1$ and $\left \lvert 2^{\leq n} \right \rvert = 2^{n+1}-1$. We also write $2^{< \omega} := \bigcup_{i \in \omega} 2^{i}$ for the set of all finite sequences, and $2^{\omega}$ for the set of all infinite sequences indexed by the ordinal $\omega$. For $\sigma \in 2^{n}$ and $i \leq m$, $\sigma|_{i} := \left( \sigma(0), \sigma(1), \ldots, \sigma(i -1) \right) \in 2^{i}$ is the initial segment of the sequence $\sigma$ of length $i$, and $\sigma|_{0} = \emptyset \in 2^{0}$. For $\sigma \in 2^{n}$, $|\sigma| := n$ denotes the length of the sequence $\sigma$. And $\triangleleft$ denotes the (strict) lexicographic ordering on the sequences in $2^{< \omega}$.

For $E \subseteq \prod_{i \in [k]} X_i$, $I \subseteq [k]$ and $\bar{b} = \left(b_i : i \in [k] \setminus I \right) \in \prod_{i \in [k] \setminus I} X_i$, we write $E_{\bar{b}}$ (or $E(x,\bar{b})$) to denote the slice (or fiber) of $E$ at $\bar{b}$, that is 
$$E_{\bar{b}} = \left\{(b_i : i \in I) \in \prod_{i \in I} X_i : (b_i : i \in [k]) \in E \right\}.$$
 For any $A \subseteq X$, we let $A^{\emptyset} := X, A^1 := A$ and $A^{0} := X \setminus A$.

\subsection{Graded probability spaces}\label{sec: graded prob spaces}
Assume $k \in \mathbb{N}$ is fixed and we are given \emph{base sets}  $V_1, \ldots, V_k$. By a \emph{sort} we will mean an arbitrary cartesian product $X$ of the sets $V_1, \ldots, V_k$, in an arbitrary order and possibly with repetitions. 
\begin{definition}
	
A \emph{$k$-partite graded probability space} on the base sets $V_1, \ldots, V_k$ is given by specifying, for each sort $X$, a $\sigma$-algebra $\mathcal{B}_{X} \subseteq \mathcal{P}(X)$ of subsets of $X$ and a countably additive probability measure $\mu_{X}$ on $\mathcal{B}_{X}$ satisfying the following compatibility conditions between different sorts:
\begin{itemize}
	\item (closure under products) for any sorts $X,Y$, $\mathcal{B}_{X \times Y} \supseteq \mathcal{B}_{X} \times \mathcal{B}_{ Y} $, i.e. the $\sigma$-algebra on the sort $X \times Y$ contains (typically  properly) the product of the $\sigma$-algebras $\mathcal{B}_{X} $ and $\mathcal{B}_{ Y} $, and the measure $\mu_{X \times Y}$ extends the product measure $\mu_{X} \times \mu_{Y}$;
	\item (symmetry) For any set $A \in \mathcal{B}_{X \times Y}$, $\mu_{X \times Y} (A) = \mu_{Y \times X}(A')$, where $A' \in \mathcal{B}_{Y \times X}$ is obtained from $A$ by permuting the coordinates;
	\item (Fubini) Given $A \in \mathcal{B}_{X \times Y}$, we have  $A_y = \{x \in X : (x,y) \in A \} \in \mathcal{B}_{X}$ for $\mu_{Y}$-almost every $y \in Y$; the function $y \in Y \mapsto \mu_{X}(A_y) \in [0,1]$ is $\mu_{Y}$-measurable; and $\mu_{X \times Y}(A) = \int_{Y} \mu_{X}(A_y) d \mu_{Y}$.
\end{itemize}
\end{definition}

 We refer to \cite[Section 2.2]{chernikov2020hypergraph} for a further discussion of graded probability spaces and their basic properties. We will need the following couple of observations.

\begin{fac}\label{lem: meas of av fib}\cite[Lemma 4.8]{chernikov2020hypergraph}
	Let $\B \subseteq \B_{X}$ be an arbitrary $\sigma$-algebra, and assume that $g: X \times Y \to \mathbb{R}$ is a $\B_{X \times Y}$-measurable function such that
	the set of  $y \in Y$ for which the function $g_y : x \mapsto  g \left( x,y \right)$ is $\B$-measurable has measure $1$.	 Then the ``average fiber'' function $g'(x) := \int_Y g\left(x,y \right) d\mu \left( y \right)$ is also $\B$-measurable.
\end{fac}

Given a measurable relation, we will consider various sub-algebras of sets that can be defined using the fibers of this relation:
\begin{definition}\label{def: algebras of fibers}
Assume that $\mathfrak{P}$ is a (partite) graded probability space on base sets $V_1, \ldots, V_k$. Assume that $X = \prod_{1 \leq t \leq K} V_{i_t}$ is a sort of $\mathfrak{P}$, for some $K \in \mathbb{N}$ and $i_t \in [k]$ for $t \in [K]$, and let $E \in \mathcal{B}_{X}$.
\begin{enumerate}
\item For $I \subseteq [K]$ and a sort $X_{I} := \prod_{t \in I} V_{i_t}$ of $\mathfrak{P}$, we let 
 $$\mathcal{F}^{0,E}_{X_{I}} := \left\{  E^{\sigma}_{\bar{b}} : \bar{b} = (b_t : t \in [K] \setminus I) \in \prod_{i \in [K] \setminus I} V_{i_t}, \sigma \in \{0,1\} \right\} \subseteq \mathcal{B}_{X_I},$$
 i.e.~the collection of all fibers of $E$ and their complements (in particular $\mathcal{F}^{0,E}_{X_{[K]}} = \left\{E, X_{[K]} \setminus E \right\}$).
 \item Let now $Y = \prod_{t \in [L]} V_{j_t}$ for $L \in \mathbb{N}, j_t \in [k]$, be an arbitrary sort in $\mathfrak{P}$. For $N \in \mathbb{N}$, we let $\mathcal{F}^{E, N}_{Y}$ be the collection of all subsets of $Y$ of the form
 \begin{gather*}
 	 \left\{ (b_i : i \in [L]) \in \prod_{t \in [L]} V_{j_t} : \bigwedge_{\ell \in [N]} (b_i : i \in J_{\ell}) \in S_{\ell}    \right\}
 \end{gather*}
 for some $J_{\ell} \subseteq [L]$ so that $\left( j_t : t \in J_{\ell} \right) = \left( i_t : t \in I_{\ell} \right)$  for some $I_{\ell} \subseteq [K]$ and some  $S_{\ell} \in \mathcal{F}^{0,E}_{X_{I_{\ell}}}$. I.e., the collection of all subsets of $Y$ that can be defined using at most $N$ conditions on some of the coordinates given by fibers of $E$ (or their complements) of arbitrary arities.
	\item 	We let $\mathcal{B}^{E}_{Y}$ be the $\sigma$-subalgebra of $\mathcal{B}_{Y}$ generated by $\bigcup_{N \in \mathbb{N}}\mathcal{F}^{E,N}_{Y}$ (in particular, if the sets $\left\{i_t : t \in [K] \right\}$ and $\left\{j_t : t \in [L] \right\}$ are disjoint, $\mathcal{B}^{E}_{Y} = \left\{ \emptyset, Y \right\}$).
\end{enumerate}
\end{definition}

\begin{lemma}\label{lem: alg of fibers closed under meas}
 For any sorts $X,Y$,  if $R \in  \mathcal{B}^E_{X \times Y}$ and $r \in \mathbb{R}$, then 
 $$\left\{ x \in X : \mu_{Y}(R_x) \square  r \right\} \in  \mathcal{B}^E_{X} \textrm{ for } \square \in \{ >, <, \geq, \leq, = \}.$$
\end{lemma}
\begin{proof}
	Indeed, $(x,y) \mapsto \chi_{R}(x,y)$ is $\mathcal{B}_{X\times Y}$-measurable and $x \mapsto \chi_{R_y}(x)$ is $\mathcal{B}^{E}_{X}$-measurable for almost every $y \in Y$ by the assumption on $R$, hence the function $x \mapsto \mu_Y(R_x) = \int_{Y} \chi_{R}(x,y) d\mu_{Y}(y) $ is also $\mathcal{B}^{E}_{X}$-measurable by Fact \ref{lem: meas of av fib}.
\end{proof}
\begin{remark}\label{rem: graded prob space of fibers}
	If follows that if $\mathfrak{P}$ is a graded probability space and $E \in \mathcal{B}_{X}$ for some sort $X$ of $\mathfrak{P}$, then replacing $\mathcal{B}_{Y}$ by $\mathcal{B}^{E}_{Y}$ and $\mu_{Y}$ by $\mu_{Y} \restriction_{\mathcal{B}^{E}_{Y}}$ for every sort $Y$ of $\mathfrak{P}$, we obtain a graded probability space $\mathfrak{P}^{E}$ on the same sorts.
\end{remark}

\subsection{Graded probability spaces as first-order structures and their ultraproducts}\label{sec: ultraprods of graded prob spaces}
We will use ultraproducts to connect asymptotic behavior in a family of finite hypergraphs with the qualitative behavior of infinite measurable hypergraphs, as well as to carry out certain compactness arguments. We will consider ultraproducts of graded probability spaces with measurable functions in the setting of approximate logic for measures \cite{goldbring2014approximate}  similarly to \cite[Section 9.3]{chernikov2020hypergraph}.

\begin{definition}
Given a graded probability space $\mathfrak{P}$, and for a countable set $I$,  $X_i$ a sort in $\mathfrak{P}$ and $E_i \in \mathcal{B}_{X_i}$, we let $\bar{E} := \left(E_i : i \in I \right)$ and associate to it a multi-sorted first-order structure 
$\mathcal{M}_{\mathfrak{P},\bar{E}}$ in the language $\mathcal{L}_{\infty,I}$ (or just $\mathcal{L}_{\infty}$ when $I$ is clear from the context) with sorts $\mathbf{X}_i$ corresponding to $X_i$ and $P_{\mathbf{X}_i}$ corresponding to $2^{\mathcal{B}_{X_i}}$, so its elements correspond to the sets in $\mathcal{B}_{X_i}$, as follows.
For every $i \in I$, $\mathcal{L}_0$ contains a relation symbol $E_{i}$ interpreted as $E_i$, and a binary relation symbol $S_{\mathbf{X}_i}(x,y)$ on $\mathbf{X}_i \times P_{\mathbf{X}_i}$ interpreted as the membership relation between the elements of $X_i$ and the sets in $\mathcal{B}_{X_i}$.
By induction on $t \in \mathbb{N}$, we define a \emph{countable} language $\mathcal{L}_t$ as follows. In addition to all the symbols in $\mathcal{L}_t$, for every \emph{quantifier-free} $\mathcal{L}
_t$-formula $\varphi(\bar{x}, \bar{y})$ such that the tuple of variables $\bar{x}$ corresponds to the sort $\mathbf{X}_i$ for some $i$ and $r \in \mathbb{Q}$, we add to $\mathcal{L}_{t+1}$ a new relational symbol $m_{\bar{x}} < r. \varphi(\bar{x}, \bar{y})$ with free variables $\bar{y}$, interpreted by: for every tuple $\bar{b}$ corresponding to $\bar{y}$,
	$$\mathcal{M}_{\mathfrak{P},\bar{E}} \models m_{\bar{x}} < r. \varphi(\bar{x}, \bar{b}) ~ :\iff \mu_{X_i} \left( \varphi(\bar{x}, \bar{b}) \right) < r,$$
	where as usual $\varphi(\bar{x},\bar{b}) = \left\{ \bar{a} \in X_i \mid  \mathcal{M}_{\mathfrak{P},\bar{E}} \models \varphi(\bar{a}, \bar{b}) \right\}$ is the set defined by the corresponding instance of $\varphi$
	(note that this set is in $\mathcal{B}_{X_i}$ by Fubini property in $\mathfrak{P}$ and induction).  
	Let $\mathcal{L}_{\infty} := \bigcup_{t \in \mathbb{N}} \mathcal{L}_t$.
	We will write $m_{\bar{x}} \geq r$ as an abbreviation for $\neg m_{\bar{x}} < r$.
\end{definition}

\begin{definition}
	Assume that for each $j \in \mathbb{N}$, $\mathfrak{P}_j$ is a graded probability space and for $i \in I$, $E_{i}^j \in \mathcal{B}_{X_i^j}$ for a sort $X_i^j$ of $\mathfrak{P}_j$. Let $\mathcal{U}$ be an ultrafilter on $\mathbb{N}$.
  We let $\widetilde{\mathcal{M}} := \prod_{j \in \mathbb{N}} \mathcal{M}_{\mathfrak{P}_j, \bar{E}^j} / \cU$ (i.e., the usual ultraproduct of $\mathcal{L}_{\infty}$-structures). We let $\widetilde{X}_i$ denote the ultraproduct of the corresponding sorts $X_i^j$, and $\widetilde{E}_i$ the ultraproduct of the $E_i^j$'s.
We let $\tilde{\B}_{\widetilde{X}_i}^0$ be the Boolean algebra of all subsets of $\widetilde{X}_i$ of the form $A = \prod_{j \in \mathbb{N}} A_j / \cU$ for some $A_j \in \B_{X_i^j}$. We let $\mu^0_{\widetilde{X}_i} := \lim_{j \to \mathcal{U}} \mu_{X_i^j}$ be the finitely additive probability measure on $\mathcal{B}^0_{\widetilde{X}_i}$, let $\mathcal{B}_{\widetilde{X}_i}$
 be the $\sigma$-algebra generated by $\mathcal{B}^0_{\widetilde{X}_i}$ and $\mu_{\widetilde{X}_i}$ the unique countably additive probability measure on $\mathcal{B}_{\widetilde{X}_i}$ extending $\mu^0_{\widetilde{X}_i}$. Then $\widetilde{\mathfrak{P}} = \left(\widetilde{X}_i, \mathcal{B}_{\widetilde{X}_i}, \mu_{\widetilde{X}_i}  \right)_{i \in I}$ is a graded probability space, and $\widetilde{E}_i \in \mathcal{B}_{\widetilde{X}_i}^0$.
 
 (We refer to \cite[Fact 9.12]{chernikov2020hypergraph} for more details.)
\end{definition}

\begin{remark}
	By \L os' theorem and definition of $\mathcal{M}_{\mathfrak{P}_j,\bar{E}^j}$, the sets in $\mathcal{B}^0_{\widetilde{X}_i}$ are precisely the sets of the form $\left\{ \bar{x} \in \mathbf{X}_i^{\widetilde{M}} : \widetilde{M} \models S_{\mathbf{X}_i}(\bar{x}, b) \right\}$ for $b \in P_{\mathbf{X}_i}^{\widetilde{M}}$ (hence they form a definable family of sets in $\widetilde{M}$).
\end{remark}

The interpretations of the ``$m_{\bar{x}} < r$'' predicates may differ in $\mathcal{M}_{\widetilde{\mathfrak{P}}, \bar{\widetilde{E}}}$ and $\widetilde{\mathcal{M}}$, but not by much:
\begin{definition}
	Let $\mathcal{M}, \mathcal{M}'$ be two $\mathcal{L}_{\infty}$-structures. We write $\mathcal{M} \propto \mathcal{M}'$ if the $\mathcal{L}_0$-reducts of the structures $\mathcal{M}, \mathcal{M}'$ are equal, and 
	for every $i \in I$, $r \in \mathbb{Q}^{[0,1]}$ and $\varepsilon \in \mathbb{Q}_{>0}$ so that $r+\varepsilon \leq 1$ we have
\begin{gather*}
		\mathcal{M}' \models m_{\bar{x}} < r . \varphi(\bar{x}, \bar{b}) \Rightarrow \mathcal{M} \models m_{\bar{x}} < r . \varphi(\bar{x}, \bar{b}) \Rightarrow
		\mathcal{M}' \models m_{\bar{x}} < \left(r + \varepsilon\right) . \varphi(\bar{x}, \bar{b})
\end{gather*}
for every  quantifier-free $\mathcal{L}_{\infty}$-formula $\varphi(\bar{x}, \bar{y})$ and a tuple $\bar{b}$ from $\mathcal{M}$ of appropriate sorts.
\end{definition}
\begin{remark}\label{rem: measure preds almost agree}\cite[Remark 9.13]{chernikov2020hypergraph}
	We have  $\widetilde{\mathcal{M}} \propto \mathcal{M}_{\widetilde{\mathfrak{P}}, \bar{\widetilde{E}}}$.
\end{remark}
 
\subsection{Stability, trees and their measurable variants}\label{sec: mu-stability}
We recall some notation around model-theoretic stability, and define their natural counterparts in a measure-theoretic context. Throughout we fix a graded probability space $\mathfrak{P}$ with sorts $X, Y, Z$, etc.
\begin{definition}
	\begin{enumerate}
		\item If $A,B \subseteq X$, we say that \emph{$B$ splits $A$} if both $A \cap B \neq \emptyset$ and $A \setminus B \neq \emptyset$.
		\item  If $A, B \in \mathcal{B}_{X}$, we say that \emph{$B$ $\mu$-splits $A$} if $\mu \left( A \cap B \right) >0$ and $\mu \left(A \setminus B \right) > 0$.
		\end{enumerate}
\end{definition}

\begin{definition}
For $E \subseteq X \times Y$, $n \in \omega$ and $\sigma \in 2^{n}$, we define:
\begin{itemize}
	\item $E^{\emptyset} := X \times Y$;
	\item $E^1 := E$, $E^0 := \left( X \times Y \right) \setminus E$;
	\item for $n \geq 1$, 
\end{itemize}
	$$E^{\sigma} := \left \{ (x, y_0, \ldots, y_{n-1}) \in X \times Y^{n} : \bigwedge_{i = 0}^{n-1} \left( (x,y_i) \in E^{\sigma(i)} \right) \right\}.$$
\end{definition}
\noindent Note that if $E \in \mathcal{B}_{X \times Y}$, then $E^{\sigma} \in \mathcal{B}_{X \times Y^{n}}$. As usual, for  $\bar{b} \in Y^{n}$, $E^{\sigma}_{\bar{b}} \subseteq X$ denotes the slice of $E^{\sigma}$ at $\bar{b}$.

\begin{definition}
	\begin{enumerate}
		\item If $E \subseteq X \times Y$, $U \subseteq X$ and $1 \leq d \in \omega$, by a  \emph{tree of height $d$ for $E$ relative to $U$} we mean a tuple $\bar{a}^{\frown}\bar{b} = \left(a_{\sigma} : \sigma \in 2^{d} \right)^{\frown} \left(b_{\sigma} : \sigma \in 2^{<d} \right)$ with $a_{\sigma} \in U, b_{\sigma} \in Y$ such that for all $\sigma \in 2^{d}$ and $i < d$,
		$$\left(a_{\sigma},b_{\sigma|_i} \right) \in E \iff \sigma(i) = 1.$$
		\item We let $\Tree^{E,d}(U) \subseteq X^{2^{d}} \times Y^{2^d - 1}$ be the set of all tuples $\bar{a}^{\frown}\bar{b} = \left(a_{\sigma} : \sigma \in 2^{d} \right)^{\frown} \left(b_{\sigma} : \sigma \in 2^{<d} \right)$ so that $\bar{a}^{\frown}\bar{b}$ is a tree of height $d$ for $E$ relative to $U$. Note that if $E \in \mathcal{B}_{X \times Y}, U \in \mathcal{B}_{X}$, then $\Tree^{E,d}(U) \in \mathcal{B}_{X^{2^d} \times Y^{2^d - 1}}$.
		\item If $E \in \mathcal{B}_{X\times Y}$, $U \in \mathcal{B}_{X}$ and $1 \leq d \in \omega$, by a \emph{$\mu$-tree of height $d$ for $E$ relative to $U$} we mean a tuple $\bar{b} = \left(b_{\sigma} : \sigma \in 2^{<d} \right)$ with $b_{\sigma} \in Y$ such that $E_{b_{\emptyset}}$ $\mu$-splits $U$, and for every $\sigma \in 2^{<d}$ with $|\sigma| \geq 1$, 
		$$E_{b_{\sigma}} \ \mu \textrm{-splits } U \cap E^{\sigma}_{\left(b_{\sigma|_0}, \ldots, b_{\sigma|_{|\sigma|-1}} \right)}.$$
			\item For $E \in \mathcal{B}_{X \times Y}$ and $U \in \mathcal{B}_{X}$, we let $\Tree^{\mu, E, d}(U) \in \mathcal{B}_{Y^{2^d - 1}}$ be the set of all $\bar{b} = \left( b_{\sigma} : \sigma \in 2^{<d}\right)$ so that $\bar{b}$ is a $\mu$-tree of height $d$ for $E$ relative to $U$.
			\item If $U = X$, we simply say ``a ($\mu$-)tree of height $d$ for $E$'', and write $\Tree^{E,d}$  (and $\Tree^{\mu,E,d}$) instead of $\Tree^{E,d}(U)$ (respectively, $\Tree^{\mu,E,d}(U)$).
	\end{enumerate}
\end{definition}
\begin{remark}\label{rem: basic props trees}
\begin{enumerate}
	\item For any $E \in \mathcal{B}_{X \times Y}$ and $U, U' \in \mathcal{B}_{X}$ with $U \subseteq U'$ we have $\Tree^{\mu, E, d}(U) \subseteq \Tree^{\mu, E,d}(U')$ and  $\Tree^{E, d}(U) \subseteq \Tree^{E,d}(U')$.
	\item Note that if $\left( b_{\sigma} : \sigma \in 2^{<d}\right) \in \Tree^{\mu, E, d}(U)$ and $\tau \in 2^{d}$, then 
	$$\mu \left( U \cap E^{\tau}_{\left( b_{\tau|_0}, \ldots,  b_{\tau|_{ d - 1}} \right)} \right) > 0.$$
	\item For any $1 \leq d \leq d'$, $\mu \left ( \Tree^{\mu, E, d'}(U) \right) \leq \mu \left ( \Tree^{\mu, E, d}(U) \right) $ (using Fubini).
\end{enumerate}
\end{remark}

\begin{lemma}\label{lem: two defs of mu stab}
	For $E \in \mathcal{B}_{X \times Y}$, we have $\mu \left( \Tree^{E,d}\right) > 0$ if and only if $\mu \left( \Tree^{\mu, E,d} \right) > 0$.
\end{lemma}
\begin{proof}
Let $B := \left\{\bar{b} \in Y^{2^d - 1} :  \mu \left( \Tree^{E,d}_{\bar{b}} \right) > 0\right\}$. By Fubini $\mu \left( \Tree^{E,d} \right) >0$ if and only if $\mu(B) > 0$.
	Note that for every fixed $\bar{b} \in Y^{2^d - 1}$ we have 
	\begin{gather*}
		\Tree^{E,d}_{\bar{b}} = \prod_{\sigma \in 2^{d}} \left\{a_{\sigma} \in X :  \bigwedge_{i < d} \left(a_{\sigma},b_{\sigma|_i} \right) \in E \iff \sigma(i) = 1\right\}.
	\end{gather*}
	Hence $\bar{b} \in B$ if and only if for every $\sigma \in 2^{d}$ we have 
	$$\mu \left( \left\{a_{\sigma} \in X :  \bigwedge_{i < d} \left(a_{\sigma},b_{\sigma|_i} \right) \in E \iff \sigma(i) = 1\right\} \right) >0.$$
	Unwinding the definitions, this holds if and only if $\bar{b} \in \Tree^{\mu, E, d}$ (using Remark \ref{rem: basic props trees}(2)). 
\end{proof}

We use one of the equivalent definitions of stability in terms of the tree property (equivalently, in terms of Shelah's $2$-rank).
\begin{definition}
	\begin{enumerate}
		\item We say that the relation $E \subseteq X \times Y$ is \emph{$d$-stable}, for $d \geq 1$, if there is no tree of height $d$ for $E$, that is $\Tree^{E,d} = \emptyset$.
		\item And $E \subseteq X \times Y$ is stable if it is $d$-stable for some $d \in \omega$.
		\item We say that $E \in \mathcal{B}_{X \times Y}$ is $d$-$\mu$-stable if the measure of the set of $\mu$-trees of height $d$ for $E$ is $0$, that is $\mu \left( \Tree^{\mu, E, d} \right) = 0$ (equivalently $\mu(\Tree^{E,d}) = 0$, by Lemma \ref{lem: two defs of mu stab}).
		\item We say that $E \subseteq X \times Y$ is $\mu$-stable if $E$ is $d$-$\mu$-stable for some $d \in \omega$.
	\end{enumerate}
\end{definition}

\begin{remark}
	Of course, if $E \in \mathcal{B}_{X \times Y}$ is $d$-stable, then it is also $d$-$\mu$-stable (but not vice versa in general).
\end{remark}

\begin{remark}
	A notion analogous to $\mu$-stability is considered independently in \cite{martin2022stability}, where it is called \emph{robust stability}.
\end{remark}

Next we will show that $\mu$-stability of $E$ is preserved under exchanging the roles of the variables of $E$. For this we recall the order property, and consider its measurable analog.
\begin{definition}\label{def: ladders}
	\begin{enumerate}
		\item For $E \subseteq X \times Y$ and $d \geq 1$, by a \emph{ladder for $E$ of height $d$} we mean a tuple $\bar{a}^{\frown}\bar{b} = (a_i : i \in d)^{\frown}(b_i : i \in d)$ with $a_i \in X, b_i \in Y$ such that for every $i,j \in d$ we have $(a_i, b_j) \in E \iff i \leq j$.
		\item Let $\Lad^{E,d} \subseteq X^{d} \times Y^{d}$ be the set of all tuples $\bar{a}^{\frown}\bar{b} = (a_i : i \in d)^{\frown}(b_i : i \in d)$ so that $\bar{a}^{\frown}\bar{b}$ is a ladder for $E$ of height $d$. Note that if $E \in \mathcal{B}_{X \times Y}$, then $\Lad^{E,d} \in \mathcal{B}_{X^{d} \times Y^{d}}$.
		\item For $E \in \mathcal{B}_{X \times Y}$, by a \emph{$\mu$-ladder for $E$ of height $d$} we mean a tuple $\bar{b} = (b_j : j \in d)$ so that for every $i \in d$ we have 
		$\mu \left( \bigcap_{i \leq j} E_{b_j} \setminus \left( \bigcup_{j > i}E_{b_j}\right) \right) > 0$.
		\item For $E \in \mathcal{B}_{X \times Y}$, let $\Lad^{\mu,E,d} \in \mathcal{B}_{Y^d}$ be the set of all $\bar{b} = \left( b_i : i \in d \right)$ so that $\bar{b}$ is a $\mu$-ladder for $E$ of height $d$.
		\item $E$ is \emph{ladder $d$-stable} if there are no ladders of height $d$ for $E$, and \emph{ladder stable} if it is ladder $d$-stable for some $d \in \omega$.
		\item $E \in \mathcal{B}_{X \times Y}$ is \emph{ladder $d$-$\mu$-stable} if $\mu \left( \Lad^{\mu,E,d} \right) = 0$. And $E$ is \emph{ladder $\mu$-stable} if it is ladder $d$-$\mu$-stable for some $d \in \omega$.
	\end{enumerate}
\end{definition}

\begin{lemma}\label{lem: two defs of mu ladder stab}
	For $E \in \mathcal{B}_{X \times Y}$, $\mu \left( \Lad^{E,d} \right) = 0$ if and only if $\mu \left( \Lad^{E,d, \mu} \right) = 0$.
\end{lemma}
\begin{proof}
Let $B := \left\{\bar{b} \in Y^{d} : \mu \left(\Lad^{E,d}_{\bar{b}} \right) > 0\right\}$, then by Fubini $\mu \left( \Lad^{E,d} \right) > 0$ if and only if  $\mu(B) >0$ (where subscript $\bar{b}$ denotes the fiber  at $\bar{b}$ of the corresponding relation).
	Note that for every fixed $\bar{b} \in Y^d$ we have: 
  \begin{gather*}
  	\Lad^{E,d}_{\bar{b}} = \prod_{i \in d} \left \{ a_i \in X : \bigwedge_{j \in d} (a_i, b_j ) \in E \iff i \leq j \right\} \\
  	= \prod_{i \in d} \left( \bigcap_{i \leq j} E_{b_j} \setminus \left( \bigcup_{j > i}E_{b_j}\right) \right).
  \end{gather*} 
  Hence $\bar{b} \in B$ if and only if for every $i \in d$ we have $\mu \left(  \bigcap_{i \leq j} E_{b_j} \setminus \left( \bigcup_{j > i}E_{b_j}\right) \right) > 0$, that is if and only if $\bar{b} \in \Lad^{\mu,E,d}$.
\end{proof}
\begin{cor}\label{cor: ladder mu stab is symm}
	If $E \in \mathcal{B}_{X \times Y}$ is ladder $d$-$\mu$-stable, then the relation $E^* := \left\{(y,x) \in Y \times X : (x,y) \in E \right\} \in \mathcal{B}_{Y \times X}$ is also ladder $d$-$\mu$-stable.
\end{cor}
\begin{proof}
Note that 
\begin{gather*}
\left( a_0, \ldots, a_{d-1}, b_0, \ldots, b_{d-1} \right) \in \Lad^{E,d} \\
\iff \left( b_{d-1}, \ldots, b_{0}, a_{d-1}, \ldots, a_{0} \right)\in \Lad^{E^{*},d}.	
\end{gather*}
Since permuting the coordinates of a set in a graded probability space preserves its measure, we have $\mu \left( \Lad^{E,d} \right) = \mu \left( \Lad^{E^{*},d} \right)$, and the corollary follows by Lemma \ref{lem: two defs of mu ladder stab}.
%
\end{proof}
\begin{fac}\cite[Lemma 6.7.9]{hodges1993model}\label{fac: Hodges}
Assume $E \subseteq X \times Y$.
	\begin{enumerate}
		\item For any $d \geq 1$, if $\left(a_i : i \in 2^{d+1} \right)^{\frown}\left(b_i : i \in 2^{d+1} \right)$ is a ladder of height $2^{d+1}$ for $E$, then for some injective functions $f: 2^{d+1} \to 2^{d+1}, g: 2^{<d+1} \to 2^{d+1}$, the tuple $\left( a_{f(\sigma)} : \sigma \in 2^{d+1}\right)^{\frown}\left( b_{g(\sigma)} : \sigma \in 2^{<d+1}\right)$ is a tree of height $d+1$ for $E$.
	\item Given $d \geq 1$, let $D := 2^{d+1}-2$.  If $\left( a_{\sigma} : \sigma \in 2^{D}\right)^{\frown}\left( b_{\sigma} : \sigma \in 2^{< D}\right)$ is a tree of height $D$ for $E$, then for some injective functions $f: d \to 2^{D}, g: d \to 2^{< D}$ the tuple $\left(a_{f(i)} : i \in d \right)^{\frown}\left(b_{g(i)} : i \in d \right)$ is a ladder of height $d$ for $E$.
	\end{enumerate}
\end{fac}

\begin{prop}\label{prop: equiv of mu stab and ladder}
Assume $E \in \mathcal{B}_{X \times Y}$ and $d \geq 1$.
	\begin{enumerate}
		\item If $E$ is $(d+1)$-$\mu$-stable, then $E$ is ladder $2^{d+1}$-$\mu$-stable.
		\item If $E$ is ladder $d$-$\mu$-stable, then $E$ is $\left( 2^{d+1}-2 \right)$-$\mu$-stable.

		\item In particular, $E$ is $\mu$-stable if and only if $E$ is ladder $\mu$-stable.
	\end{enumerate}
\end{prop}
\begin{proof}
	(1) Assume that $E$ is not ladder $2^{d+1}$-$\mu$-stable, then $\mu \left( \Lad^{E,2^{d+1}} \right) > 0$.
Given injective functions $f: 2^{d+1} \to 2^{d+1}, g: 2^{<d+1} \to 2^{d+1}$, let
\begin{gather*}
	A_{f,g} := \Big\{ \left(a_i : i \in 2^{d+1} \right)^{\frown}\left(b_i : i \in 2^{d+1} \right) \in X^{2^{d+1}} \times  Y^{2^{d+1}}:  \\
	\left( a_{f(\sigma)} : \sigma \in 2^{d+1}\right)^{\frown}\left( b_{g(\sigma)} : \sigma \in 2^{<d+1}\right) \in \Tree^{E,d+1}  \Big\} \in \mathcal{B}_{X^{2^{d+1}} \times Y^{2^{d+1}}}.
\end{gather*}
Note that we have $\mu_{X^{2^{d+1}} \times Y^{2^{d+1}}}\left( A_{f,g} \right) = \mu_{X^{2^{d+1}} \times Y^{2^{d+1} - 1}}\left( \Tree^{E,d+1} \right)$ using axioms of graded probability spaces, and by Fact \ref{fac: Hodges}(1) we have $\Lad^{E,2^{d+1}} \subseteq \bigcup_{f,g} A_{f,g}$, where the union is over all (but finitely many) injective functions $f: 2^{d+1} \to 2^{d+1}, g: 2^{<d+1} \to 2^{d+1}$.
Hence $\mu(A_{f,g}) >0$ for at least one choice of $f,g$, hence $\mu \left( \Tree^{E,d+1} \right) > 0$.

 (2) Similar to (1), using Fact \ref{fac: Hodges}(2) instead.
 
 (3) Combining (1) and (2).
\end{proof}
\begin{cor}\label{cor: stab pres under opp}
$E \in \mathcal{B}_{X \times Y}$ is $\mu$-stable if and only if $E^{*} \in \mathcal{B}_{Y \times X}$ is $\mu$-stable.	
\end{cor}
\begin{proof}
	Combining Proposition \ref{prop: equiv of mu stab and ladder} and Lemma \ref{cor: ladder mu stab is symm}.
\end{proof}

\begin{lemma}\label{lem: Bool comb of mu-stab}

		If $E_1 \in \mathcal{B}_{X \times Y_1},  E_2 \in  \mathcal{B}_{X \times Y_2}$ are $\mu$-stable, then $E_1 \land E_2 \in \mathcal{B}_{X \times (Y_1 \times Y_2)}$ and $\neg E_1 \in \mathcal{B}_{X \times Y_1}$ are also $\mu$-stable.
\end{lemma}
\begin{proof}
	
	 The standard argument showing that stability is preserved under Boolean combinations (using Ramsey theorem, see for example \cite[Lemma 2.2.10]{chernikov2019lecture}) shows the following: for every $d \in \mathbb{N}$ there exists $D \in \mathbb{N}$ such that if $(a_i : i \in D)^{\frown}((b_i, b'_i) : i \in D) \in \Lad^{E_1 \land E_2, D}$, then we have either $\left( a_{f(i)} : i \in d \right)^{\frown} \left(b_{f(i)} : i \in d \right) \in \Lad^{E_1, d}$ for some injective $f : d \to D$, or $\left( a_{f(i)} : i \in d \right)^{\frown} \left(b'_{f(i)} : i \in d \right) \in \Lad^{E_2, d}$ for some injective $f : d \to D$. Hence $\Lad^{E_1 \land E_2, D}$ is contained in
	\begin{gather*}
		\bigcup_{f : d \to D \textrm{, injective}} \Big\{ (a_i : i \in D)^{\frown}((b_i, b'_i) : i \in D) \in X^D \times (Y_1 \times Y_2)^{D} :  \\\left( a_{f(i)} : i \in d \right)^{\frown} \left(b_{f(i)} : i \in d \right) \in \Lad^{E_1, d}\Big\} \cup\\
		\bigcup_{f : d \to D \textrm{, injective}} \Big\{ (a_i : i \in D)^{\frown}((b_i, b'_i) : i \in D) \in X^D \times (Y_1 \times Y_2)^{D} :  \\\left( a_{f(i)} : i \in d \right)^{\frown} \left(b'_{f(i)} : i \in d \right) \in \Lad^{E_2, d}\Big\}.
	\end{gather*}
	So if $\mu \left( \Lad^{E_1 \land E_2, D} \right) > 0 $, at least one of the sets in this finite union also has positive measure, hence at least one of the sets $\Lad^{E_1, d}, \Lad^{E_2, d}$ has positive measure, and we conclude by Lemma \ref{lem: two defs of mu ladder stab} and Proposition \ref{prop: equiv of mu stab and ladder}(3).	
\end{proof}

\begin{definition}\label{def:slice-wise}
\begin{enumerate}
\item A ternary relation $E \subseteq X \times Y \times Z$ is \emph{slice-wise ($d$-)stable} if: the slices $E_x \subseteq Y \times Z$ are ($d$-)stable for all $x \in X$, the slices $E_y \subseteq X \times Z$ are ($d$-)stable for  all $y \in Y$, and the slices $E_z  \subseteq X \times Y$ are ($d$-)stable for all $z\in Z$. 

	\item A ternary relation $E \in \mathcal{B}_{X \times Y \times Z}$ is \emph{slice-wise ($d$-)$\mu$-stable} if: the slices $E_x \in \mathcal{B}_{Y \times Z}$ are ($d$-)$\mu$-stable for almost all $x \in X$, the slices $E_y \in \mathcal{B}_{X \times Z}$ are ($d$-)$\mu$-stable for almost all $y \in Y$, and the slices $E_z \in \mathcal{B}_{X \times Y}$ are ($d$-)$\mu$-stable for almost all $z\in Z$. 
\end{enumerate}

\end{definition}

\begin{definition}\label{def: partition-wise stability}
\begin{enumerate}
	\item A $k$-ary relation $E \in \mathcal{B}_{X_1 \times \ldots \times X_k}$ is \emph{partition-wise $d$-stable} (\emph{partition-wise $d$-$\mu$-stable}) if for every $1 \leq \ell \leq k$, $E$ viewed as a binary relation on $X_{\ell} \times \left(\prod_{i \in [k] \setminus \{\ell\}} X_i \right)$ is $d$-stable (respectively, $d$-$\mu$-stable).
	\item A $k$-ary relation $E \in \mathcal{B}_{X_1 \times \ldots \times X_k}$  is \emph{partition-wise stable} (\emph{partition-wise $\mu$-stable}) if it is partition-wise $d$-stable (respectively, partition-wise $d$-$\mu$-stable) for some $d \in \mathbb{N}$.
\end{enumerate}
\end{definition}

\begin{prop}
	\begin{enumerate}
		\item If $E \subseteq X \times Y \times Z$ is partition-wise ($d$-)stable then it is also slice-wise ($d$-)stable (but not vice versa --- see Example \ref{ex: x=y<z stability props}).
		\item If $E \in \mathcal{B}_{X \times Y \times Z}$ is partition-wise ($d$-)$\mu$-stable then it is also slice-wise ($d$-)$\mu$-stable.
	\end{enumerate}
\end{prop}
\begin{proof}
(1) is immediate from the definitions, and we show (2).

Assume that $E$ is not slice-wise $d$-$\mu$-stable, so there is a set $Z_0 \in \mathcal{B}_{Z}$ with $\mu(Z_0) > 0$ such that $\mu \left( \Tree^{E_z(x;y),d}\right) > 0$ for every $z \in Z_0$. Let 
\begin{gather*}
	B := \left\{ \bar{x} ^{\frown} \bar{y} ^{\frown} z  \in X^{2^d} \times Y^{2^{d} - 1} \times Z: \bar{x} ^{\frown} \bar{y} \in \Tree^{E_z(x;y),d} \right\} \in \mathcal{B}_{X^{2^d} \times Y^{2^{d} - 1} \times Z}.
\end{gather*}
For every $z \in Z_0$ we have $\Tree^{E_z(x;y),d} \subseteq B_z$, so $\mu \left( B_z \right) > 0$, hence by Fubini $\mu \left( B \right) > 0$.

Let $C := \left\{ \bar{x}^{\frown}\bar{y} : \mu \left( B_{\bar{x}^{\frown} \bar{y}} \right) > 0 \right\} \in \mathcal{B}_{X^{2^d} \times Y^{2^{d} - 1}}$, by Fubini $\mu \left( C \right) > 0$.

Consider the set 
\begin{gather*}
	D := \Big\{ \bar{x}^{\frown}\bar{y}^{\frown}\bar{z} \in X^{2^d} \times   Y^{2^{d} - 1} \times Z^{2^{d} - 1} : \\
	 \bar{x}^{\frown} \left( (y_{\sigma}, z_{\sigma}) : \sigma \in 2^{<d}\right) \in \Tree^{E(x;(y,z)),d} \Big\} \in \mathcal{B}_{X^{2^d} \times Y^{2^{d} - 1} \times Z^{2^{d} - 1}}.
\end{gather*}
For each fixed $\bar{x}^{\frown}\bar{y} \in X^{2^d} \times   Y^{2^{d} - 1}$, we have 
\begin{gather*}
	\left\{ \left( z_{\sigma} : \sigma \in 2^{<d} \right)  \in Z^{2^{d}-1}: \bigwedge_{\sigma \in 2^{<d}} \left( z_{\sigma} \in B_{\bar{x}^{\frown}\bar{y}} \right) \right\} \subseteq D_{ \bar{x}^{\frown}\bar{y}},
\end{gather*}
hence for every $\bar{x}^{\frown}\bar{y} \in C$ we have $\mu \left( D_{ \bar{x}^{\frown}\bar{y}} \right) \geq \mu \left( B_{\bar{x}^{\frown}\bar{y}} \right)^{2^d - 1} > 0$. By Fubini (and since permuting the coordinates preserves measure) it follows that $\mu \left( D \right) = \mu \left( \Tree^{E(x;(y,z)),d} \right) > 0$, hence $E$ is not partition-wise $d$-$\mu$-stable.
\end{proof}

\section{Perfect and decent sets and partitions, and the measure theoretic regularity lemma for $\mu$-stable hypergraphs}\label{sec: perfect stable reg for graphs}

\begin{definition}\label{def: perfect and decent}
Assume $E \in \mathcal{B}_{X \times Y}$.
\begin{enumerate} 
	\item A set $A \in \mathcal{B}_{X}$ is \emph{perfect for $E$} if 
$$ \mu \left( \left\{b \in Y : E_b \ \mu \textrm{-splits } A \right\} \right) = 0.$$
\item For $\varepsilon, \delta \in [0,1]$, a set $A \in \mathcal{B}_{X}$ is \emph{$(\varepsilon, \delta)$-decent for $E$} if for 
$$A^{\#}_{\varepsilon} :=  \left\{b \in Y : \mu\left( A \cap E_b  \right) > \varepsilon \mu(A) \land  \mu\left( A \setminus E_b  \right) > \varepsilon  \mu(A) \right\} \in \mathcal{B}_{Y}$$
we have $ \mu \left( A^{\#}_{\varepsilon} \right) \leq \delta$.
We say that $A$ is \emph{$\varepsilon$-decent} if it is $(\varepsilon, \varepsilon)$-decent. So perfect is the same as $0$-decent.
\item We define perfection and decency for subsets of $Y$ symmetrically.
\end{enumerate}
	\end{definition}
	\begin{remark}\label{rem: basic props perf sets}
		\begin{enumerate}
		\item If $\mu(A) =0$, then $A$ is automatically perfect. 
			\item If $A \in \mathcal{B}_{X}$ is perfect and $A' \subseteq A, A' \in \mathcal{B}_{X}$, then $A'$ is also perfect.
			\item  A set is $(\varepsilon, 0)$-decent if it is \emph{$\varepsilon$-good} in the sense of \cite{malliaris2014regularity}. Allowing a set of exceptions of small measure $\delta$ is unavoidable when passing from stability to the more general context of $\mu$-stability.
		\end{enumerate} 
	\end{remark}
	

\begin{lemma}\label{lem: 01 dens on perf sets}
	If $E \in \mathcal{B}_{X \times Y}$, $\varepsilon \in [0,\frac{1}{2})$, $A \in \mathcal{B}_{X}$ and $B \in \mathcal{B}_{Y}$,   $A$ is $(\varepsilon^2, \varepsilon \mu(B))$-decent for $E$ and $B$ is  $(\varepsilon, \varepsilon \mu(A))$-decent for $E$, then 
	$$\frac{\mu \left(E \cap \left( A \times B \right) \right)}{\mu \left( A \times B \right)} \in [0, 3 \varepsilon) \cup (1 - 4 \varepsilon, 1].$$
	
		In particular, if  $A \in \mathcal{B}_{X}, B \in \mathcal{B}_{Y}$ are perfect, then $\frac{\mu \left(E \cap \left( A \times B \right) \right)}{\mu \left( A \times B \right)} \in \{0,1\}$.
\end{lemma}
\begin{proof}

%
%
	
	Assume $A,B$ are $\varepsilon$-good and $\mu \left(E \cap (A \times B) \right) \geq 3 \varepsilon \mu(A) \mu(B)$. 
	
	Let $A_0 := \left\{ x \in A: \mu(E_x \cap B) > \varepsilon \mu(B)\right\}$. 
		
	By Fubini and linearity of integration,
	\begin{gather*}
		3 \varepsilon \mu(A) \mu(B) \leq \mu \left(E \cap (A \times B) \right)  = \int_{A} \mu(E_{x} \cap B) dx \\
		= \int_{A_0}\mu(E_{x} \cap B) dx + \int_{A \setminus A_0}\mu(E_{x} \cap B) dx.
	\end{gather*}
	As \begin{gather*}
		\int_{A \setminus A_0}\mu(E_{x} \cap B) dx \leq \int_{A \setminus A_0} \varepsilon \mu(B) dx \leq  \varepsilon \mu(B) \mu (A ),
	\end{gather*} we get
	\begin{gather*}
		\left(3 \varepsilon - \varepsilon \right) \mu(A) \mu(B) \leq \int_{A_0}\mu(E_x \cap B)dx \leq \mu(A_0) \mu(B),
	\end{gather*}
 hence  $\mu(A_0) \geq 2 \varepsilon \mu(A)$. Let 
 \begin{gather*}
		A_1 := \left\{ x \in A: \mu(E_x \cap B) > (1- \varepsilon) \mu(B)\right\} \subseteq A_0.
	\end{gather*}
 
  As $B$ is $(\varepsilon, \varepsilon \mu(A))$-decent, we have 
  \begin{gather}
  	\mu(A_1 ) \geq \mu(A_0) - \varepsilon \mu(A) \geq \varepsilon \mu(A).\label{eq: 01 dens on perf sets 1}
  \end{gather}
Then, by Fubini again, 
\begin{gather*}
	\mu \left(E \cap (A_1 \times B) \right) = \int_{A_1} \mu(E_x \cap B) dx \geq (1-\varepsilon) \mu(A_1) \mu(B).
\end{gather*}
Let $B_0 := \left\{ y \in B: \mu(E_y \cap A_1) > \varepsilon \mu(A_1)\right\}$.

Similarly, by Fubini and linearity of integration,
	\begin{gather*}
		(1 -\varepsilon) \mu(A_1) \mu(B) \leq \mu \left(E \cap (A_1 \times B) \right)  = \int_{B} \mu(E_{y} \cap A_1) dy \\
		= \int_{B_0}\mu(E_{y} \cap A_1) dy + \int_{B \setminus B_0}\mu(E_{y} \cap A_1) dy.
	\end{gather*}
	
	As \begin{gather*}
		\int_{B \setminus B_0}\mu(E_{y} \cap A_1) dy \leq \int_{B \setminus B_0} \varepsilon \mu(A_1) dy \leq  \varepsilon \mu(B) \mu (A_1),
	\end{gather*} we get
	\begin{gather*}
		\left(1 - 2 \varepsilon \right) \mu(A_1) \mu(B) \leq \int_{B_0}\mu(E_y \cap A_1)dy \leq \mu(A_1) \mu(B_0),
	\end{gather*}
 hence  $\mu(B_0) \geq (1 - 2 \varepsilon) \mu(B)$. 
 
 By \eqref{eq: 01 dens on perf sets 1}, for every $b \in B_0$ we have $\mu(E_b \cap A) \geq  \mu(E_b \cap A_1) > \varepsilon \mu(A_1) \geq \varepsilon^2 \mu(A)$.
 Let 
  \begin{gather*}
		B_1 := \left\{ b \in B: \mu(E_b \cap A) > (1- \varepsilon^2) \mu(A)\right\} \subseteq B_0.
	\end{gather*}
 
As $A$ is $(\varepsilon^2, \varepsilon \mu(B))$-decent, we have $\mu(B_1) \geq \mu(B_0) - \varepsilon \mu(B) \geq (1 - 3\varepsilon) \mu(B)$. Hence, by Fubini, \begin{gather*}
	\mu\left( E \cap (A \times B) \right) = \int_{B} \mu(E_y \cap A) dy \geq \int_{B_1} \mu(E_y \cap A) dy \\
	\geq (1-3\varepsilon) \mu(B) (1-\varepsilon^2) \mu(A) \geq  \mu(A)\mu(B) \geq (1 - 4 \varepsilon)\mu(A)\mu(B) ,
	\end{gather*}
	which implies the claim.
\end{proof}

We next observe that for a stable relation, we can always find a perfect set of positive measure for it. In fact, as will be needed for our later analysis of slice-wise stability, we can do it uniformly for a family of stable relations, by finding a relation of higher arity such that a positive measure of its fibers are perfect. 

\begin{lemma}\label{lem: pos meas of perf sets for stable E}
Assume that $d \in \mathbb{N}_{\geq 1}$, $E \in \mathcal{B}_{X \times Y}$ is $d$-$\mu$-stable and $U \in \mathcal{B}_{X}$ with $\mu(U) > 0$. Then there exist $1 \leq m \leq d$ and $\sigma \in 2^{m}$ so that the set 
	$$\Perf^{E, \sigma} \left( U \right) := \left \{ \bar{b} \in  Y^{m}: U \cap E^{\sigma}_{\bar{b}} \textrm{ is perfect and } \mu\left( U \cap E^{\sigma}_{\bar{b}} \right) > 0 \right\} \in \mathcal{B}_{Y^{m}}$$
has positive measure.
	\end{lemma}
\begin{proof}

	As $E$ is $d$-$\mu$-stable, we have $\mu \left( \Tree^{\mu, E, d} \right) = 0$. 
	
If $\mu \left( \Tree^{\mu, E, 1}(U) \right) = 0$, then $U$ is already perfect. Hence for at least one $\sigma \in \{0,1\}$  we have $\mu \left( \Perf^{E,\sigma} \right) \geq \frac{1}{2}$.	
	
	Otherwise let $1 \leq m < d$ be maximal such that $\mu \left( \Tree^{\mu, E, m}  (U) \right) > 0$ (using Remark \ref{rem: basic props trees}).
	For every fixed $\bar{b} = \left(b_{\sigma} : \sigma \in 2^{<m} \right) \in \Tree^{\mu, E, m}(U)$, we have
	\begin{gather*}
		\prod_{\sigma \in 2^m} \left\{ b_{\sigma} \in Y : E_{b_{\sigma}} \  \mu \textrm{-splits } U \cap E^{\sigma}_{\left(b_{\sigma|_0}, \ldots, b_{\sigma|_{m-1}} \right)}\right\} \subseteq \Tree^{\mu, E, m+1}_{\bar{b}} \left( U \right).
	\end{gather*}
	Hence by maximality of $m$ and Fubini we have
	\begin{gather*}
		0 = \mu \left( \Tree^{\mu, E, m+1} \left( U \right) \right) = \int_{Y^{2^m - 1}} \mu \left( \left( \Tree^{\mu, E, m+1} \left( U \right) \right)_{\bar{b}} \right) d\mu \left(\bar{b} \right)\\
		\geq \int_{\Tree^{\mu, E, m}(U)} \mu \left( \left( \Tree^{\mu, E, m+1} \left( U \right) \right)_{\bar{b}} \right) d\mu \left(\bar{b} \right)\\
		\geq \int_{\Tree^{\mu, E, m}(U)} \prod_{\sigma \in 2^m} \mu \left( \left\{ b_{\sigma} \in Y : E_{b_{\sigma}} \  \mu \textrm{-splits } U \cap E^{\sigma}_{\left(b_{\sigma|_0}, \ldots, b_{\sigma|_{m-1}} \right)}\right\} \right) d\mu \left(\bar{b} \right).
	\end{gather*} 
	As the integral is equal to $0$, $\mu \left(\Tree^{\mu, E, m} \left( U \right)\right) > 0$ and we are integrating a non-negative function, there exists some $B \subseteq \Tree^{\mu, E, m}(U), B \in \mathcal{B}_{Y^{2^m - 1}}$ with $\mu \left(B \right) = \mu \left(\Tree^{\mu, E, m}(U)\right) > 0$ and such that for every $\bar{b} \in B$, 
	\begin{gather*}
		\prod_{\sigma \in 2^m} \mu \left( \left\{ b_{\sigma} \in Y : E_{b_{\sigma}} \  \mu \textrm{-splits } U \cap E^{\sigma}_{\left(b_{\sigma|_0}, \ldots, b_{\sigma|_{m-1}} \right)}\right\} \right) = 0.
	\end{gather*}
	But then there exist some $\sigma \in 2^{m}$ and $B' \subseteq B, B' \in \mathcal{B}_{Y^{2^m - 1}}$ with $\mu \left( B' \right) \geq \frac{1}{2^{m}} \mu \left( B \right) > 0$, such that for every $\bar{b} \in B'$ we have
	\begin{gather*}
		\mu \left( \left\{ b_{\sigma} \in Y : E_{b_{\sigma}} \  \mu \textrm{-splits } U \cap E^{\sigma}_{\left(b_{\sigma|_0}, \ldots, b_{\sigma|_{m-1}} \right)}\right\} \right) = 0.
	\end{gather*}
	 Let $B'' := \left\{ \left(b_{\tau} : \tau \in 2^{<m} \right) : \left( b_{\sigma|0}, \ldots, b_{\sigma|_{m-1}} \right) \in \Perf^{E, \sigma}(U)\right\} \in \mathcal{B}_{Y^{2^m - 1}}$. By Fubini we have $\mu \left(B'' \right) = \mu \left(\Perf^{E, \sigma}(U) \right)$.  Note also that for every $\bar{b} \in \Tree^{\mu, E, m}(U)$, and hence for every $\bar{b} \in B'$, by  Remark \ref{rem: basic props trees}(2) we have $\mu \left( U \cap E^{\sigma}_{\left(b_{\sigma|_0}, \ldots, b_{\sigma|_{m-1}} \right)} \right) > 0$. Hence $B' \subseteq B''$, and so $\mu \left(\Perf^{E, \sigma}(U) \right) \geq \mu(B') >0$.
\end{proof}

\begin{lemma}\label{lem: perf part stable}
	If $E \in \mathcal{B}_{X \times Y}$ is $\mu$-stable, then there exists a partition of $X$ into countably many perfect sets in $\mathcal{B}^E_{X}$.
\end{lemma}
\begin{proof}
	We choose perfect sets $A_i \in \mathcal{B}^E_{X}$ by induction on $i \in \omega$. Suppose we have already chosen $A_j \in \mathcal{B}^E_{X}$ for $j < i$. Let $U_i := \bigcup_{j <i} A_i$. If $\mu \left( X \setminus U_i \right) = 0$, taking $A_i := X \setminus U_i  $ (perfect by Remark \ref{rem: basic props perf sets}) we are done. Otherwise, by Lemma \ref{lem: pos meas of perf sets for stable E}, there exists a perfect subset $B \in \mathcal{B}^E_{X}$ of $X \setminus U_i$ of positive measure. We can choose a perfect $A_i \in \mathcal{B}^E_{X}$ so that if $B\in \mathcal{B}^E_{X}$ is any perfect subset of $X \setminus U_i$, then $\mu(A_i) \geq \frac{\mu(B)}{2}$.
	
	If  $\mu \left( \bigcup_{i \in \omega} A_i \right) < 1$, by  Lemma \ref{lem: pos meas of perf sets for stable E} again we could find a perfect $B \subseteq X \setminus \left( \bigcup_{i \in \omega} A_i\right)$, $B \in \mathcal{B}^E_{X}$, with $r := \mu(B) > 0$. As $B  \subseteq X \setminus U_i$ for all $i \in \omega$, by the choice of the $A_i$'s we must have $\mu(A_i) \geq \frac{r}{2}$ for all $i \in \omega$ --- a contradiction.
\end{proof}
\begin{remark}
	This is related to an important early observation of Keisler in \cite {keisler1987measures} that on a family of instances of a stable formula, every Keisler measure is a countable weighted sum of types.
\end{remark}

The following immediate corollary is a (non-quantitive) measure-theoretic variant of the stable graph regularity lemma \cite{malliaris2014regularity} (See \cite{MR3731711,malliaris2016stable} for some variants.)
\begin{prop}\label{prop: stable graph reg}
	Assume that $E \in \mathcal{B}_{X \times Y}$ is $\mu$-stable. Then there exist countable partitions $X = \bigsqcup_{i \in \omega} A_i$ and $Y = \bigsqcup_{j \in \omega} B_j$ into perfect sets, such that $A_i \in \mathcal{B}_{X}^{E}$ and $B_j \in \mathcal{B}_{Y}^{E}$.
	 In particular, for each $i,j \in \omega$, $\frac{\mu \left( E \cap \left(A_i \times B_j \right) \right)}{\mu \left( A_i \times B_j \right)} \in \{0,1\}$.
	\end{prop}
	\begin{proof}
		Combining Lemmas \ref{lem: perf part stable} and \ref{lem: 01 dens on perf sets}.
	\end{proof}
	
This generalizes to partition-wise stable hypergraphs (see Definition \ref{def: partition-wise stability}), giving a simplified version of \cite[Theorem 4.13]{chernikov2021definable}.

\begin{lemma}\label{lem: prod of perf is perf}
Given $E \in \mathcal{B}_{X \times Y \times Z}$, assume that $X_0 \in \mathcal{B}_{X}$ is perfect for $E$ viewed as a binary relation on $X \times \left(Y \times Z \right) $ and $Y_0 \in \mathcal{B}_{Y}$ is perfect for $E$ viewed as a binary relation on $Y \times \left(X \times Z \right) $.
	Then $X_0 \times Y_0 \in \mathcal{B}_{X \times Y}$ is perfect for $E$ viewed as a binary relation on $\left( X \times Y \right) \times Z$.
\end{lemma}
\begin{proof}
 As $X_0$ is perfect for $E(x; (y,z))$, we have 
	\begin{gather*}
		\mu \left( \left\{ (y,z) \in Y \times Z : \frac{\mu \left( E_{(y,z)} \cap X_0 \right)}{\mu \left( X_0 \right)} \notin \{0,1\} \right\} \right) = 0.
	\end{gather*}
	By Fubini this implies 
	\begin{gather*}
		\mu \left( \left\{ z \in Z : \mu \left( \left\{ y \in Y : \frac{\mu \left( E_{(y,z)} \cap X_0 \right)}{\mu \left( X_0 \right)} \notin \{0,1\} \right\}  \right) > 0 \right\} \right) = 0.
	\end{gather*}
	That is,
	\begin{gather*}
				\mu \left( \left\{ z \in Z : X_0 \textrm{ is } \textrm{perfect for } E_z(x,y)\right\} \right) = 1.
	\end{gather*}
Similarly, as $Y_0$ is perfect for $E(y; (x,z))$, we conclude 
\begin{gather*}
	\mu \left( \left\{ z \in Z : Y_0  \textrm{ is } \textrm{perfect for } E_z(x,y)\right\} \right) = 1.
\end{gather*}
But then for almost every $z \in Z$ we have that both $X_0, Y_0$ are perfect for $E_z$, so $\frac{\mu \left(E_z \cap \left(X_0 \times Y_0 \right) \right)}{\mu \left(X_0 \times Y_0  \right)} \in \{0,1\}  $ by Lemma \ref{lem: 01 dens on perf sets}. That is, the set $X_0 \times Y_0$ is perfect for $E \left( (x,y); z \right)$.
\end{proof}
\begin{remark}
	Lemma \ref{lem: prod of perf is perf} has a straightforward quantitive version similarly to Lemma \ref{lem: 01 dens on perf sets}, but we will not need it here.
\end{remark}

\begin{lemma}\label{lem: perfect regularity equivalents}
	Let $E \in \mathcal{B}_{X_1 \times \ldots \times X_k}$ and assume that for $1 \leq t \leq k$, $\left\{ A_{t,i} : i  \in \omega \right\}$ with $A_{t,i} \in \mathcal{B}_{X_t}$ is a countable partition of $X_t$. Then the following two conditions are equivalent:
	\begin{enumerate}
		\item for every $1 \leq t \leq k, i \in \omega$, the set $A_{t,i}$ is perfect for $E$ viewed as a binary relation on $X_t \times \left( \prod_{s \in [k] \setminus \{t\}}  X_s \right)$;
		\item for every $(i_1, \ldots, i_k) \in \omega^k$, we have $\frac{\mu \left( E \cap \left(A_{1,i_1} \times \ldots \times A_{k,i_k} \right) \right)}{\mu \left( A_{1,i_1} \times \ldots \times A_{k,i_k} \right)} \in \{0,1\}$.
	\end{enumerate}
\end{lemma}
\begin{proof}
	(1) implies (2).  Fix $(i_1, \ldots, i_k) \in \omega^k$.
	By induction on $1 \leq t \leq k-1$ we have that the set $\prod_{\ell = 1}^{t} A_{\ell, i_{\ell}}$ is 	perfect for $E$ viewed as a binary relation on $\prod_{\ell = 1}^{t} X_{\ell}$ and $\prod_{\ell = t+1}^{k} X_{\ell}$. Indeed, the base case $t=1$ is given by (1), and on the inductive step we get that $\prod_{\ell = 1}^{t+1} A_{\ell, i_{\ell}}$ is perfect by Lemma \ref{lem: prod of perf is perf} applied to $\prod_{\ell = 1}^{t} A_{\ell, i_{\ell}}$ (perfect by the inductive assumption) and $A_{t+1, i_{t+1}}$ (perfect by (1)). But then we have that $\prod_{\ell = 1}^{k-1} A_{\ell, i_{\ell}}$ and $A_{k, i_k}$ are perfect for $E$ viewed as a binary relation on $\prod_{\ell = 1}^{k-1} X_{\ell}$ and $X_k$, hence $\frac{\mu \left( E \cap \left(A_{1,i_1} \times \ldots \times A_{k,i_k} \right) \right)}{\mu \left( A_{1,i_1} \times \ldots \times A_{k,i_k} \right)} \in \{0,1\}$ by Lemma \ref{lem: 01 dens on perf sets}.

(2) implies (1). Fix $i_1 \in \omega$, we show that $A_{1,i_1}$ is perfect (the argument is the same for all $A_{t,i_t}$). For any $(i_2, \ldots, i_k)$, by (2) we have $\frac{\mu \left( E \cap \left(A_{1,i_1} \times \ldots \times A_{k,i_k} \right) \right)}{\mu \left( A_{1,i_1} \times \ldots \times A_{k,i_k} \right)} \in \{0,1\}$. By Fubini and countable additivity this implies that the measure of the set
\begin{gather*}
	\left\{  (x_2, \ldots, x_k) \in A_{2,i_2} \times \ldots \times A_{k,i_k} :  \frac{\mu \left(A_{1, i_1} \cap E_{(x_2, \ldots, x_k)} \right)}{\mu \left(A_{1,i_1} \right)} \notin \{0,1\} \right\}
\end{gather*}
is $0$. But then, taking the union over all $(i_2, \ldots, i_k) \in \omega^{k-1}$, by countable additivity we get 
\begin{gather*}
	\mu \left( \left\{  (x_2, \ldots, x_k) \in X_2 \times \ldots \times X_k :  \frac{\mu \left(A_{1, i_1} \cap E_{(x_2, \ldots, x_k)} \right)}{\mu \left(A_{1,i_1} \right)} \notin \{0,1\} \right\}  \right) = 0,
\end{gather*}
which means that $A_{1,i_1}$ is perfect.
\end{proof}

\begin{definition}\label{def: perf stab reg}
\begin{enumerate}
	\item We will say that $E \in \mathcal{B}_{X_1 \times \ldots \times X_k}$ satisfies \emph{perfect stable regularity} if  for $1 \leq t \leq k$ there exist  countable partitions $\left\{ A_{t,i} : i  \in \omega \right\}$ of $X_t$  with $A_{t,i} \in \mathcal{B}_{X_t}$ which satisfy either of the equivalent conditions in Lemma \ref{lem: perfect regularity equivalents}.
	\item We will say that $E \in \mathcal{B}_{X_1 \times \ldots \times X_k}$ satisfies \emph{definable perfect stable regularity} if moreover we can choose $A_{t,i} \in \mathcal{B}^{E}_{X_t}$.
\end{enumerate}
\end{definition}

\begin{prop}\label{prop: stable hypergraph reg}
	Assume that $E \in \mathcal{B}_{X_1 \times \ldots \times X_k}$ is partition-wise $\mu$-stable. Then it satisfies  definable	perfect stable regularity.
	\end{prop}
\begin{proof}
	By partition-wise $\mu$-stability, Lemma \ref{prop: stable graph reg} can be applied to choose a perfect partition on $X_t$ for each $t \in [k]$.
\end{proof}

\section{Various versions of the stable regularity lemma for hypergraphs and a transfer principle}\label{sec: transfer for stable regularities}

\subsection{Definitions}\label{sec: defs transf}

As Proposition \ref{prop: stable hypergraph reg} is established for arbitrary (possibly uncountable) graded probability spaces, using a standard compactness argument we can obtain some additional uniformity for free corresponding to the stable version of the \emph{strong regularity lemma} (see Definition \ref{def: strong stab reg}), and which is meaningful for families of $d$-stable finite hypergraphs (a similar argument in the case of graphs was observed in \cite{chavarria2021continuous}):

Given finite sets $X_1, \ldots, X_k$, unless specified otherwise, we will view them as a $k$-partite graded probability space with $\mathcal{B}_{Y}$ consisting of all subsets of $Y$ and $\mu_{Y}$ the uniform counting probability measure on $Y$, for $Y$ an arbitrary direct product of the $X_i$'s.

\begin{definition}\label{def: strong stab reg}
Let $\mathcal{H}$ be a  family of finite $k$-partite $k$-hypergraphs of the form $H = (E;X_1, \ldots, X_k)$ with $E \subseteq \prod_{i = 1}^{k} X_i$ and $X_i$ finite. We say that $\mathcal{H}$ satisfies:
	\begin{enumerate}
		
	\item  \emph{approximately perfect stable regularity} if for every $\varepsilon \in \mathbb{R}_{>0}$ and every function $f: \mathbb{N} \to (0,1)$ there exists some $N = N(f,\varepsilon)$ so that: for any $H \in \mathcal{H}$ there exists $N' \leq N$ and partitions $X_i = \bigsqcup_{0 \leq t < N'} A_{i,t}$  so that:
	\begin{enumerate}
	\item $\left \lvert A_{i,0} \right \rvert \leq \varepsilon \left \lvert X_{i} \right \rvert$ for all $1 \leq i \leq k$;
		\item for any $1 \leq i \leq k$ and any $ 1 \leq t < N'$, the set $A_{i,t}$ is $f(N')$-decent for $E$ viewed as a binary relation on $X_{i} \times \left( \prod_{j \neq i} X_j \right)$.
	\end{enumerate}

		\item \emph{strong stable regularity} if for every $\varepsilon \in \mathbb{R}_{>0}$ and every function $f: \mathbb{N} \to (0,1]$ there exists some $N = N(f,\varepsilon)$ so that: for any $H \in \mathcal{H}$ there exists $N' \leq N$ and partitions $X_i = \bigsqcup_{0 \leq t < N'} A_{i,t}$  so that:
	\begin{enumerate}
		\item $\left \lvert A_{i,0} \right \rvert \leq \varepsilon \left \lvert A_{i,1} \right \rvert$ for all $1 \leq i \leq k$;
		\item for any $ 1 \leq t_1, \ldots, t_k < N'$ we have 
		\begin{gather*}
			\frac{\left \lvert E \cap \left( A_{1,t_1} \times \ldots \times  A_{k,t_k} \right) \right \rvert }{\left \lvert A_{1,t_1} \times \ldots \times  A_{k,t_k}  \right \rvert } \in [0, f(N')) \cup (1 - f(N'), 1].
		\end{gather*}
	\end{enumerate}
	
	\item \emph{stable regularity} if for every $\varepsilon \in \mathbb{R}_{>0}$ there exists some $N = N(\varepsilon)$ such that: for any $H = (E;X_1, \ldots, X_k) \in \mathcal{H}$ there exists $N' \leq N$ and partitions $X_i = \bigsqcup_{1 \leq t < N'} A_{i,t}$  so that 
	for any $ 1 \leq t_1, \ldots, t_k \leq N'$ we have 
		\begin{gather*}
			\frac{\left \lvert E \cap \left( A_{1,t_1} \times \ldots \times  A_{k,t_k} \right) \right \rvert }{\left \lvert A_{1,t_1} \times \ldots \times  A_{k,t_k}  \right \rvert } \in [0, \varepsilon) \cup (1 - \varepsilon, 1].
		\end{gather*}
\end{enumerate}
\end{definition}

	


\begin{definition}\label{def: definable stable regularities}
	We will say that $\mathcal{H}$ satisfies \emph{definable}  approximately perfect stable (strong stable, stable) regularity if moreover for any $H \in \mathcal{H}$, the sets $A_{i,t}$ satisfying the requirement in Definition \ref{def: strong stab reg} can be chosen in $\mathcal{F}^{E,N}_{X_i}$ (see Definition \ref{def: algebras of fibers}).
\end{definition}

\begin{remark}\label{rem: stronger metastable version}
	If the family $\mathcal{H}$ is hereditarily closed, then in Definition \ref{def: strong stab reg}(1), we could equivalently strengthen (a) to: 
		\begin{itemize}
			\item[(a$'$)] $\left \lvert A_{i,0} \right \rvert \leq \varepsilon \left \lvert A_{i,1} \right \rvert$ for all $1 \leq i \leq k$.
		\end{itemize}
\end{remark}
\begin{proof}
	Assume $\mathcal{H}$ satisfies Definition \ref{def: strong stab reg}(1). For $g: \mathbb{N} \to (0,1]$ and $\delta \in \mathbb{R}_{>0}$, we let $N(g,\delta)$ denote the smallest $N \in \mathbb{N}$ satisfying  Definition \ref{def: strong stab reg}(1). Note that 
	\begin{gather}
		N(g,\delta) \geq N(g, \delta') \textrm{ for any function } g \textrm{ and } \delta < \delta'. \label{eq: N monotone}
	\end{gather}
	Let arbitrary $\varepsilon >0$ and $f: \mathbb{N} \to (0,1]$ be given. 
	  Let $f_2 (n) := f(n+1)$, $f_1(n) := f \left( N \left(f_2,\frac{\varepsilon}{2n} \right) + 1 \right)$, and let $N_1 := N(f_1, \frac{1}{2})$.
	Let $\varepsilon_2 := \frac{\varepsilon}{2 N_1}$, and $N_2 := N(f_2, \varepsilon_2)$.  We claim that $N_3 := N_2 + 1$ (only depending on $f$ and $\varepsilon$) satisfies (a$'$) and (b).
	
	Given an arbitrary $H = (E;X_1, \ldots, X_k) \in \mathcal{H}$, we consider partitions of $X_1$ (the argument is the same for the other $X_i$'s). By the choice of $N_1$, there exist some $N'_1 \leq N_1$ and a partition $A_{1,t} \subseteq X_1$ of $X$ for $0 \leq t \leq N'_1$  so that $\mu_{X_1} \left(A_{1,0} \right) \leq \frac{1}{2}$ (throughout the argument, we talk about the uniform measure on the set in the subscript, $X_1$ in this case) and  $A_{1,t}$ is $f_1(N'_1)$-decent for $1 \leq t \leq N'_1$. In particular, renumbering the sets if necessary, we may assume $\mu_{X_1} \left( A_{1,1}\right) \geq \frac{1}{2N'_1}$.  Let $X'_1 := X_1 \setminus A_{1,1}$. 
	
	As $\mathcal{H}$ is hereditarily closed, the hypergraph 
	$$H' := \left(E \restriction_{X'_1 \times X_2 \times \ldots \times X_k}; X'_1, X_2, \ldots, X_k \right)$$
	 is in $\mathcal{H}$.  Let $N'_2 := N \left(f_2, \frac{\varepsilon}{2N'_1} \right)$. By \eqref{eq: N monotone} we have $N'_2 \leq N_2$.
	So there exists a partition $A'_{1,0}, \ldots, A'_{1, N'_2}$ of  $X'_1$ so that $\mu_{X'_1} \left( A'_{1,0} \right) \leq  \frac{\varepsilon}{2N'_1}$, hence  $ \mu_{X_1} \left( A'_{1,0} \right) = \mu_{X'_1} \left( A'_{1,0} \right) \cdot \mu_{X_1} \left( X'_1 \right) \leq  \varepsilon \frac{1}{2N'_1} \leq \varepsilon \mu_{X_1}(A_{1,1})$ and $A'_{1,t} \subseteq X_{1}$ is $f_2(N'_2)$-decent for $1 \leq t \leq N'_2$ in $H'$, but then also in $H$.
	Finally, consider the partition $A'_{1,0}, \ldots, A'_{1,N'_2}, A'_{1, N'_2 + 1} := A_{1,1}$ of $X_1$ of size  $N'_2 + 1 \leq N_3$. Each of the sets $A'_{1,t}, 1 \leq t \leq N'_2$ is $f_2(N'_2)$-decent, hence $f(N'_2 + 1)$-decent by the choice of $f_2$, and $A_{1,1}$ is $f_1(N'_1)$-decent, hence $f \left(N \left(f_2, \frac{\varepsilon}{2N'_1} \right) + 1 \right)$-decent by the choice of $f_1$, hence $f(N'_2 + 1)$-decent by the choice of $N_2$. Renumbering the sets, we have thus found a partition satisfying (a$'$) and (b) in Definition \ref{def: strong stab reg}(1). 
\end{proof}

\begin{remark}\label{rem: meta implies strong implies stable}
\begin{enumerate}
	\item If $\mathcal{H}$ satisfies (definable) approximately perfect stable regularity, then it satisfies (definable) strong stable regularity.
	\item If $\mathcal{H}$ satisfies (definable) strong stable regularity, then it satisfies (definable) stable regularity.
\end{enumerate}
	\end{remark}
	\begin{proof}
	(1) Combining Theorem \ref{thm: metastable general equivalences}, namely (1) implies (3) there, and Proposition \ref{prop: one way transfer for strong stable reg}.
(Alternatively, this can be obtained directly via a quantitive version of Lemma \ref{lem: perfect regularity equivalents} for $\varepsilon$-decent sets.)

%
	
	(2) Replacing $A_{i, 1}$ in the partition of $X_i$ by $A_{i, 1} \cup A_{i,0}$, for each $1 \leq i \leq k$ (see e.g. \cite[Remark 4.7]{chavarria2021continuous} for the details).
\end{proof}

\begin{remark}\label{rem: relaxed strong stable reg implies stable reg}
	In the definition of  stable regularity (Definition \ref{def: strong stab reg}(3)) we can equivalently relax the requirement that $\left\{ A_{i,t} : 1 \leq t \leq N' \right\}$ give a partition of $X_i$ by: 
	\begin{itemize}
		\item for every $1 \leq i \leq k$, the sets $A_{i,t}, 1 \leq t \leq N'$ are pairwise-disjoint and $\lvert \bigsqcup_{1 \leq t < N'} A_{i,t} \rvert \geq (1 - \varepsilon) |X_i|$.
	\end{itemize}
(We note that this equivalence does not hold for the definable versions --- in order to get the equivalence of the definable versions, we would have to require $\lvert \bigsqcup_{1 \leq t < N'} A_{i,t} \rvert \geq  |X_i| - \varepsilon |A_{i,1}|$.)  Indeed, let $\varepsilon >0$ be given. Take an arbitrary $0 < \varepsilon' < \min \left\{\varepsilon, \frac{1}{2} \right\}$, and let $N$ be as given by this relaxed stable regularity for $\mathcal{H}$ with respect to $\varepsilon'$. Given $H \in \mathcal{H}$, there is $N' \leq N$ and $\{ A_{i,t} : 0 \leq t <N' \}$ partitions of $X_i$ satisfying $\lvert A_{i,0} \rvert \leq \varepsilon |X_i|$ and for any $ 1 \leq t_1, \ldots, t_k \leq N'$ we have $\frac{\left \lvert E \cap \left( A_{1,t_1} \times \ldots \times  A_{k,t_k} \right) \right \rvert }{\left \lvert A_{1,t_1} \times \ldots \times  A_{k,t_k}  \right \rvert } \in [0, \varepsilon) \cup (1 - \varepsilon, 1]$.
 Fix $i \in [k]$.  As $|A_{i,0}| \leq \varepsilon'|X_i|$ and $\varepsilon' < \frac{1}{2}$, hence $\varepsilon' < 2 \varepsilon' (1 - \varepsilon')$, we can partition (might not be possible to do definably in general) $A_{i,0}$ into $\left\{ A_{i,0}^t : 1 \leq t <N'  \right\}$ with $\lvert A_{i,0}^t \rvert \leq 2 \varepsilon'   \left \lvert A_{i,t}\right \rvert $. And let $A'_{i,t} := A_{i,t} \cup A_{i,0}^{t}$ for $1 \leq t < N'$, then $\left\{A'_{i,t} : 1 \leq t < N' \right\}$ is a partition of $X_i$. For any $1 \leq t_1, \ldots, t_{k} < N'$, by assumption we have $\frac{\left \lvert E \cap \left( A_{1,t_1} \times \ldots \times  A_{k,t_k} \right) \right \rvert }{\left \lvert A_{1,t_1} \times \ldots \times  A_{k,t_k}  \right \rvert } \in [0, \varepsilon') \cup (1 - \varepsilon', 1]$, hence $\frac{\left \lvert E \cap \left( A'_{1,t_1} \times \ldots \times  A'_{k,t_k} \right) \right \rvert }{\left \lvert A'_{1,t_1} \times \ldots \times  A'_{k,t_k}  \right \rvert } \in [0, \varepsilon' +  2 k \varepsilon') \cup (1 - (\varepsilon' + 2 k \varepsilon'), 1]$. So if we took $\varepsilon' >0$ so that $(2k+1) \varepsilon' < \varepsilon$  for all $n \in \mathbb{N}$, it follows that $N$ and the partitions $A'_{i,t}$ satisfy stable regularity (Definition \ref{def: strong stab reg}(3)), as wanted.
\end{remark}

\subsection{Working Through an Example}\label{sec:example}

To illustrate the difference between strong stable regularity and (approximately) perfect stable regularity, we consider a single  example in the infinitary setting in some detail and try to provide the corresponding intuition. An explicit hereditarily closed family associated with it and separating the two kinds of stable regularity properties is exhibited in  Example \ref{ex: x=y<z stability props} after we develop a transfer principle. 

\begin{example}\label{ex: x=y<z finite}
  Let $X = Y = Z := [0,1]$ and  $E := \left\{ (x,y,z) : x = y < z \right\}$.  Then $E$ is slice-wise stable (in all three directions).
\end{example}
To really examine the behavior of this hypergraph, we should consider it under an arbitrary probability measure, but the generic behavior comes out if we place the uniform measure on $Z$ while placing some atomic measures on $X$ and $Y$ (so that the event $x=y$ occurs with positive probability)---for instance, pick any probability measure on $[0,1]\cap\mathbb{Q}$ which assigns to every rational some positive measure.

If the reader prefers to consider hereditary families of finite $3$-hypergraphs, we may consider the family of finite $3$-hypergraphs $(X,Y,Z,E)$ such that:
\begin{itemize}
\item $Z\subseteq[0,1]$,
\item there are functions $\pi_X:X\rightarrow[0,1]$ and $\pi_Y:Y\rightarrow[0,1]$ such that $E=\left\{(x,y,z) : \pi_X(x)=\pi_Y(y)<z\right\}$.
\end{itemize}
The representative example is when $Z$ is ``uniform''---say, $Z$ is all rationals whose denominator, in least terms, is $<k$---while $X$ and $Y$ map many points (say, half of all points) to $1/2$, slightly fewer points to $1/3$ and $2/3$ (say, an eighth of all points map to each) and so on (or see Example \ref{ex: x=y<z stability props}).

First, it is not hard to see that this $3$-hypergraph cannot have perfect stable regularity. We do have good choices of countable partitions for $X$ and $Y$---they partition into atoms. (Analogously, in the finite setting, the approximately perfect pieces are the ``blobs'' which have a common value under $\pi_X$ or $\pi_Y$.) However there is no corresponding way to choose even a single partition piece from $Z$: for any subset $Z'$ of $Z$ with positive measure, there is some rational $q$ so that $Z'$ contains both positive measure above $q$ and positive measure below, and therefore the box $\{q\}\times \{q\}\times Z'$ will neither be disjoint from $E$ nor contained in $E$ (even up to measure $0$).

However this $3$-hypergraph does satisfy stable regularity, and even strong stable regularity. Fix some $\varepsilon>0$. To produce a strong stable regularity partition, we take the largest atoms of $X$ and $Y$---that is, we take $A_1=\{x_1\}$ where $\mu(\{x_1\})$ is maximal, then $A_2=\{x_2\}$ where $\mu(\{x_2\})\leq\mu(\{x_1\})$ is maximal among remaining points, etc. We keep choosing such atoms as partition components until the measure of the remaining points is $<\varepsilon\mu(A_1)$. We then let the exceptional set $A_0$ be the remainder. We do the same thing to $Y$ to obtain a partition into atoms $B_i$ and a remainder $B_0$.

The finite partitions we chose only involved finitely many values $x_i,y_i$, so we may partition $Z$ into intervals $C_i$ between these points; we may take the remainder $C_0$ to consist of the finitely many values $x_i$ or $y_i$, which is a finite set of points and so has measure $0$.

Then in any box $A_i\cap B_j\cap C_k$ with $0<\min\{i,j,k\}$, we have $A_i=\{x_i\}$ and $B_j=\{y_j\}$. If $x_i\neq y_j$ then this box is disjoint from $E$. If $x_i=y_j$ then, since $C_k$ is an interval which does not cross $x_i$, either $A_i\cap B_j\cap C_k\subseteq E$ or $A_i\cap B_j\cap C_k$ is disjoint from $E$.

As usual, we can obtain stable regularity from strong stable regularity by redistributing the error component (see Remark \ref{rem: meta implies strong implies stable}(2)): we replace each $A_i$ with some $A'_i$ of the form $A_i\cup S_i$ where $S_i$ is a subset of $A_0$ with measure approximately $\mu(A_0)\mu(A_i)<\varepsilon\mu(A_i)$, and similarly for each $B_i$.

We can ask why this strong stable regularity does not give us almost perfect stable regularity. One way to think about this is to notice what happens if we change $\varepsilon$. When $\varepsilon$ gets smaller, the $A_i$ with $i>0$ which we have already identified stay in our partition, and we need to break additional atoms off of $A_0$. The same happens for the partition of $Y$.
These new atoms force us to break up the partition of $Z$, however: if our new value of $\varepsilon$ is much smaller, we expect to have identified new atoms in every interval, so that \emph{no} set $C_k$ from our original partition is still usable --- all the sets $C_k$ need to be partitioned into yet smaller sets in order to obtain a strong stable regularity partition with this new value of $\varepsilon$.

This is precisely the thing that blocks approximately perfect stable regularity. We may think of strong stable regularity as giving us a sequence of partitions, one for each $\varepsilon$. Approximately perfect stable regularity asks that these partitions themselves stabilize---that the sets in our partition stay largely the same as we go along the sequence, and the main change as we go along this sequence is that pieces get broken off the remainder $A_0$ to become new partition components.

\subsection{Transfer}\label{sec: subsec tranfer}
First we establish a transfer principle between approximately perfect stable regularity for hereditarily closed families of hypergraphs (with respect to the uniform counting measures)  and perfect stable regularity (Definition \ref{def: perf stab reg}) for their ultraproducts. 

\begin{prop}\label{prop: proof of transfer perf part} Let $\mathcal{H}$ be a hereditarily closed family of finite bipartite graphs of the form  $H = (E;X,Y)$ with $E \subseteq X \times Y$. The following are equivalent:
	\begin{enumerate}
		\item for any ultraproduct $\widetilde{H} = (\widetilde{E}; \widetilde{X}, \widetilde{Y})$ of the graphs from $\mathcal{H}$, there exists a  partition of $\widetilde{X}$ into countably many perfect sets from $\mathcal{B}_{\widetilde{X}}$;

		\item for every $\varepsilon \in \mathbb{R}_{>0}$ and every function $f: \mathbb{N} \to (0,1]$ there exists some $N = N(f,\varepsilon)$ so that: for any $H \in \mathcal{H}$ there exists $N' \leq N$ and partition $X = \bigsqcup_{0 \leq t < N'} A_{t}$  so that:
	\begin{enumerate}
	\item $\left \lvert A_{0} \right \rvert \leq \varepsilon \left \lvert X \right \rvert$;
		\item for any $ 1 \leq t < N'$, the set $A_{t}$ is $f(N')$-decent for $E$.
	\end{enumerate}

	\end{enumerate}
\end{prop}
\begin{proof}

Before beginning the proof of the implications, we show that perfection and decency are type-definable properties in definable families of sets, which will be needed for the transfer and saturation arguments. Let $\mathcal{H}$ be a family of finite bipartite graphs (each equipped with the corresponding graded probability space given by the uniform counting measures on the algebra of all subsets), $\mathcal{U}$ a non-principal ultrafilter, and $\widetilde{H} = (\widetilde{E}; \widetilde{X}, \widetilde{Y})$ an ultraproduct of the graphs from $\mathcal{H}$ with respect to $\mathcal{U}$, with the underlying graded probability space (see Section \ref{sec: ultraprods of graded prob spaces}).

\begin{claim}\label{cla: decent is type-def}
For every $\varepsilon, \delta \in \mathbb{Q}_{(0,1]}$ there exists a countable quantifier-free partial  $\mathcal{L}_{\infty}$-type $\Gamma_{\varepsilon, \delta}(z)$	 so that: for any graded probability space $\mathfrak{P}$ on $(X,Y, Z \ldots)$, $E \in \mathcal{B}_{X \times Y}$, $S \in \mathcal{B}_{X \times Z}$,  $\mathcal{M} \propto \mathcal{M}_{\frak{P}, E,S}$ and $b$ in $\mathcal{M}$,  $\mathcal{M} \models \Gamma_{\varepsilon, \delta}(b) \iff (S)_b$ is $(\varepsilon, \delta)$-decent with respect to $E, \frak{P}$.
\end{claim}
\begin{proof}
	Let $\frak{P}$ be a  graded probability space on $(X,Y, \ldots)$, $E \in \mathcal{B}_{X \times Y}$. Fix  $A \in \mathcal{B}_{X}$ and let 
	\begin{gather*}
		A^{\#}_{\varepsilon} := \{ y \in Y : \varepsilon \mu_{X} (A) < \mu_{X} \left( E_y \cap A  \right) < (1 - \varepsilon) \mu_{X}(A)  \} \in \mathcal{B}_{Y}.
	\end{gather*} 
	Let $Q :=  \left\{ (r,q) \in  \mathbb{Q}_{[0,1]}^2  : r < q\right\}$, and for $(r,q) \in Q$ let
	\begin{gather}
		A^{\#, r,q}_{\varepsilon} := \Bigg\{ y \in Y :  \mu(A) < \frac{r}{\varepsilon} \ \land \  \mu(E_y \cap A) \geq r \label{eq: defining perfection 0} \\
		 \  \land \  \mu(E_y \cap A) < q \  \land  \ \mu(A) \geq \frac{q}{1 - \varepsilon}  \Bigg\} .\nonumber
	\end{gather}
	
	Then $A^{\#}_{\varepsilon} = \bigcup_{(r,q) \in Q} A^{\#, r,q}_{\varepsilon}$, and  by countable additivity 
	\begin{gather}
		A \textrm{ is } (\varepsilon, \delta) \textrm{-decent} \iff \mu_{Y} \left(A^{\#}_{\varepsilon} \right) \leq \delta  \label{eq: defining perfection 1} \\
		\iff  \textrm{ for all  finite } Q_0 \subseteq Q \textrm{ and }\gamma \in \mathbb{Q}_{(0,1]}, \mu_{Y} \left( \bigcup_{(r,q) \in Q_0} A^{\#, r,q}_{\varepsilon} \right) < \delta + \gamma. \nonumber
	\end{gather}

	For $(r,q) \in Q$, consider the quantifier-free $\mathcal{L}_{\infty}$-formula 
	\begin{gather*}
		\psi_{\varepsilon}^{r,q}(y,z) := m_x < \frac{r}{\varepsilon}. S(x,z) \ \land \  \neg m_x < r. \left( E(x,y) \land S(x,z) \right) \\
		\land \  m_x < q. \left(E(x,y) \land S(x,z) \right) \  \land \  m_x < \frac{q}{1 - \varepsilon}. S(x,z)
	\end{gather*}
%
	and the quantifier-free  partial $\mathcal{L}_{\infty}$-type  
	\begin{gather*}
		\Gamma_{\varepsilon, \delta}(z) := \bigwedge_{\gamma \in \mathbb{Q}_{(0,1]}} \bigwedge_{Q_0 \subseteq Q, |Q_0| < \aleph_0} m_{y} < (\delta + \gamma). \bigvee_{(r,q) \in Q_0} 	\psi_{\varepsilon}^{r,q}(y,z).
	\end{gather*}
	
	By \eqref{eq: defining perfection 0} and \eqref{eq: defining perfection 1}  we have: 
	\begin{gather}
		S_b \textrm{ is }(\varepsilon, \delta) \textrm{-decent} \iff \mathcal{M}_{\frak{P},E,S} \models 	\Gamma_{\varepsilon, \delta}(b).\label{eq: defining perfection 2}
	\end{gather}

	  And as $\mathcal{M} \propto \mathcal{M}_{\frak{P}, E,S}$, for any $\gamma \in \mathbb{Q}_{>0}$ we have 
	  \begin{gather*}
	  	\mathcal{M} \models m_{y} < (\delta + \gamma). \bigvee_{(r,q) \in Q_0} 	\psi_{\varepsilon}^{r,q}(y,z)\\
	  	\Rightarrow \mathcal{M}_{\frak{P},E,S} \models m_{y} < (\delta + \gamma). \bigvee_{(r,q) \in Q_0} 	\psi_{\varepsilon}^{r,q}(y,z) \\
 	\Rightarrow \mathcal{M} \models m_{y} < (\delta + 2\gamma). \bigvee_{(r,q) \in Q_0} 	\psi_{\varepsilon}^{r,q}(y,z).
	  \end{gather*}
	  Hence, using \eqref{eq: defining perfection 2} and the definition of $\Gamma_{\varepsilon,\delta}$, we conclude that 
	 $S_b$ is $(\varepsilon,\delta)$-decent $ \iff \mathcal{M} \models \Gamma_{\varepsilon,\delta}(b)$.
\end{proof}

\begin{claim}\label{cla: transfer claim 2}
Assume that $A = \prod_{i \in \mathbb{N}} A_i / \mathcal{U} \in \mathcal{B}^0_{\widetilde{X}}$	 with $A_i \in \mathcal{B}_{X_i}$ and $\varepsilon \in (0,\frac{1}{2} )$.
\begin{enumerate}
	\item If $A$ is $\varepsilon$-decent with respect to $\widetilde{E}, \widetilde{\mathfrak{P}}$, then  
	$$\left\{ i \in \mathbb{N} : A_i \textrm{ is } 2\varepsilon \textrm{-decent for }E_i \right\} \in \mathcal{U}.$$
	\item If there is $U \in \mathcal{U}$ so that $A_i$ is $\varepsilon$-decent for all $i \in U$, then $A$ is $2\varepsilon$-decent.
\end{enumerate} 
\end{claim}
\begin{proof}
(1) Let $r := \mu_{\widetilde{X}}(A)$. Fix some $\delta = \delta(r,\varepsilon) > 0$ so that $2 \varepsilon \delta < \varepsilon r$ and $\delta(1-2\varepsilon) < \varepsilon r$. As $\mu_{\widetilde{X}}(A) = \lim_{\mathcal{U}} \mu_{X_i}(A_i)$, we can choose $U \in \mathcal{U}$ such that $\mu_{X_i}(A_i) \in (r - \delta, r + \delta)$ for all $i \in U$. Let $b_i \in P_{\mathbf{X},i}$ be such that $A_i$ is defined by $S_{\mathbf{X}}(x,b_i)$ in $\mathcal{M}_{\mathfrak{P}_i, E_i}$, then $A$ is defined by $S_{\mathbf{X}}(x,\widetilde{b})$ in $\widetilde{\mathcal{M}}$ (or $\mathcal{M}_{\widetilde{\mathfrak{P}}, \widetilde{E} } $) for  $\widetilde{b} := (b_i : i \in \mathbb{N}) / \mathcal{U}$. Assume (1) fails, then intersecting $U$ with the complement of $\left\{ i \in \mathbb{N} : A_i \textrm{ is } 2\varepsilon \textrm{-decent for }E_i \right\}$, we may assume that $A_i$ is not $2 \varepsilon$-decent for all $i \in U$. Then for every $i \in U$, for 
\begin{gather*}
	\left( A_i \right)^{\#}_{2\varepsilon} = \left\{b \in Y_i : \frac{\mu_{X_i} \left( A_i \cap E_b\right)}{ \mu_{X_i}(A_i)} \in (2 \varepsilon, 1 - 2\varepsilon)  \right\}
\end{gather*}
we have $\mu_{Y_i}\left(\left( A_i \right)^{\#}_{2 \varepsilon}  \right) > 2 \varepsilon $. By the choice of $U$ this implies that $\mu_{Y_i}(A'_i) > 2 \varepsilon$ for  
\begin{gather*}
	A'_i := \left\{b \in Y_i : \mu_{X_i} \left( A_i \cap E_b\right) \in \left(2 \varepsilon (r - \delta), (1- 2\varepsilon) (r  + \delta)  \right) \right\}
\end{gather*}
By the choice of $\delta$, we have $\left(2 \varepsilon (r - \delta), (1- 2\varepsilon) (r  + \delta)  \right) \subseteq \left(\varepsilon r, (1 - \varepsilon) r \right)$,
so for all $i \in U$ we have 
\begin{gather*}
	\mathcal{M}_{\mathfrak{P}_i,E_i} \models \neg  m_{y} < 2 \varepsilon. \Big( \left( \neg m_{x} < \varepsilon r. S_{\mathbf{X}}(x,b_i) \land E(x,y) \right) \\
	\land  \left( m_{x} < (1-\varepsilon) r. S_{\mathbf{X}}(x,b_i) \land E(x,y) \right) \Big).
\end{gather*}
(If $\varepsilon, r$ were not rational, we can replace them by sufficiently close rational numbers.)
Hence, by \L os,  $\widetilde{b}$ satisfies the same sentence in $\widetilde{\mathcal{M}}$, so using \eqref{eq: Linfty ultraprod} we get
\begin{gather*}
	\mu_{\widetilde{Y}} \left( \left\{ y \in \widetilde{Y} :  \varepsilon r < \mu_{\widetilde{X}} \left( A \cap E_y \right) < (1 - \varepsilon) r \right\} \right) >  \varepsilon,
\end{gather*}
which by the choice of $r$ implies that $A$ is not $\varepsilon$-decent.

(2) 
By Claim \ref{cla: decent is type-def}, $\mathcal{M}_{\mathfrak{P}_i,E_i} \models \Gamma_{\varepsilon',\varepsilon'}(b_i)$ for all $i \in U$, where $\varepsilon'$ is an arbitrary rational number in $(\varepsilon, 2 \varepsilon)$. Hence by \L os $\widetilde{\mathcal{M}} \models \Gamma_{\varepsilon',\varepsilon'}(\widetilde{b})$, and as $\widetilde{\mathcal{M}} \propto \mathcal{M}_{\widetilde{\mathfrak{P}}, \widetilde{E}}$, we conclude that $A = \left(S_{\mathbf{X}} \right)_{\widetilde{b}}$ is $2 \varepsilon$-decent for $\widetilde{E}$.
\end{proof}

\textbf{(1) implies (2).} 	Assume that $\mathcal{H}$ satisfies (1), and assume towards contradiction that there exist some fixed $\varepsilon, f$ for which the conclusion of (2) fails.
 That is, for every $n \in \mathbb{N}$ there exist some $H_n = (E_n; X_n, Y_n) \in \mathcal{H}$ so that no partition of $X_n$ of size at most $n$ satisfies conditions (a) and (b) of (2) with respect to $E_n$ and $n$ in place of $N$.
 Let $\mathcal{U}$ be a non-principal ultrafilter on $\mathbb{N}$. We consider the ultraproduct $\widetilde{H} = (\widetilde{E}; \widetilde{X}, \widetilde{Y}) := \prod_{n \in \mathbb{N}} H_n / \mathcal{U}$ (with the underlying graded probability space).
 
 By Remark \ref{rem: measure preds almost agree} we have: for every $r \in \mathbb{Q}_{[0,1]}$ and $\varepsilon \in \mathbb{Q}_{>0}$ so that $r + \varepsilon \leq 1$, quantifier-free $\mathcal{L}_{\infty}$-formula $\varphi(\bar{x},\bar{y})$ and tuple $\bar{b}$ from $\widetilde{\mathcal{M}}$ of appropriate sort we have:
 \begin{gather}
 	\mathcal{M}_{\widetilde{\mathfrak{P}}, \widetilde{E} } \models m_{\bar{x}} < r. \varphi(\bar{x}, \bar{b}) \Rightarrow \widetilde{\mathcal{M}} \models m_{\bar{x}} < r. \varphi(\bar{x}, \bar{b}) \label{eq: Linfty ultraprod}\\
 	\Rightarrow \mathcal{M}_{\widetilde{\mathfrak{P}}, \widetilde{E} } \models m_{\bar{x}} < (r+\varepsilon). \varphi(\bar{x}, \bar{b}). \nonumber
 \end{gather}

By (1), there exists a countable partition $A_i \in \mathcal{B}_{\widetilde{X}}$ with $1 \leq i< \omega$  into perfect sets for $\widetilde{E}$.

Fix $\varepsilon' = \varepsilon'(\varepsilon) \in (0,1)$, to be determined later.
By countable additivity we can choose $N \in \mathbb{N}$ so that 
$$\mu_{\widetilde{X}} \left( \bigcup_{1 \leq i < N} A_i \right) \geq 1 - \varepsilon'.$$

Fix $\delta = \delta(f,N) \in (0,1)$, to be determined later. As $\mathcal{B}_{\widetilde{X}}$ is the $\sigma$-algebra generated by $\mathcal{B}^0_{\widetilde{X}}$, for each $1 \leq i < N$ we can choose $A'_i \in \mathcal{B}^0_{\widetilde{X}}$ with $\mu_{\widetilde{X}} \left( A'_i \triangle A_i \right) \leq \delta \mu(A_i)$ so that $A'_i$ are pairwise-disjoint and $\mu_{\widetilde{X}}(A'_i)>0$ for $1 \leq i < N$. Then
\begin{gather}
	\mu_{\widetilde{X}} \left( \bigcup_{1 \leq i < N} A'_i \right) \geq 1 - ( \varepsilon' + \delta).\label{eq: str stab hyper 1}
\end{gather}

As $\mu \left( A'_i \triangle A_i \right) \leq \delta \mu(A_i)$ we have $\mu(A_i) \leq \frac{1}{1-\delta}\mu(A'_i)$, so $\mu \left( A'_i \triangle A_i \right) \leq  \frac{\delta}{1-\delta} \mu \left(A'_i \right)$.

Then, as each $A_i$ is perfect with respect to $\widetilde{E}$, we still have: for every $1 \leq i < N$ and almost every $y \in Y$,
\begin{gather*}
	\frac{\mu \left( \widetilde{E}_y \cap A'_i \right)}{\mu \left( A'_i \right)} \in \left[0, \frac{\delta}{1-\delta} \right) \cup \left(1 - \frac{\delta}{1-\delta}, 1 \right]. 
\end{gather*}
(Indeed, if $\mu \left( \widetilde{E}_y \cap A_i  \right) = 0$, then 
\begin{gather*}
	\mu \left(\widetilde{E}_y \cap A'_i \right) \leq \mu \left(\widetilde{E}_y \cap A_i  \right) + 
	\mu \left(\widetilde{E}_y \cap (A'_i \setminus A_i) \right) \\
	\leq 0 + \frac{\delta}{1-\delta} \mu \left(A'_i \right) = \frac{\delta}{1-\delta} \mu \left(A'_i \right),
\end{gather*}
and similarly for $\mu \left( \widetilde{E}_y \cap A_i  \right) = \mu \left( A_i  \right)$.)
In particular, 
\begin{gather}
	A'_i \textrm{ is } \left(\frac{\delta}{1-\delta} \right) \textrm{-decent for } 1 \leq i < N. \label{eq: str stab hyper 2}	
\end{gather}

For each $i$ we have $A'_i = \prod_{n \in \mathcal{U}} A_{n,i}/\mathcal{U}$ for some sets $A_{n,i} \in \mathcal{B}_{X_n}$.

We can take $\varepsilon' := \frac{\varepsilon}{4}$ and $\delta > 0$
 so that $ \delta < \frac{\varepsilon}{4}$ and $2 \frac{\delta}{1-\delta} < f(N)$, 
 
 Then there is $U \in \mathcal{U}$ so that for every $n \in U$ we have:
 \begin{itemize}
 \item the sets $A_{n,i}$ for $1 \leq i < N$ are pairwise disjoint (by \L os);
 	\item $\mu_{X_n} \left( \bigcup_{1 \leq i < N} A_{n,i} \right) \geq 1 - 2 ( \varepsilon' + \delta) > 1 - \varepsilon$ (by \eqref{eq: str stab hyper 1} and definition of $\mu_{\widetilde{X}}$);
 	\item for every $1 \leq i < N$, the set $A_{n,i}$ is $\left( \frac{2 \delta}{1 - \delta} \right)$-decent with respect to $E_n$ (by \eqref{eq: str stab hyper 2} and Claim \ref{cla: transfer claim 2}(1)), hence $f(N)$-decent.
 \end{itemize}
 
 But then for any $n \in U, n \geq N$, the partition $A_{n,1}, \ldots, A_{n,N}, A_{n,0} := X_n \setminus \bigcup_{1 \leq i < N} A_{n,i}$ of $X_n$ contradicts the choice of $H_n$.

~
~

\noindent \textbf{(2) implies (1).} 
Suppose that (1) holds.

\begin{claim}\label{cla: finding up to epsilon in UP}
 For every $\varepsilon > 0$	 there exists some $K\in \mathbb{N}$ and disjoint $A_i \in \mathcal{B}^0_{\widetilde{X}}$ for $1 \leq i < K$ so that each $A_i$ is perfect with respect to $\widetilde{E}$ and $\mu_{\widetilde{X}} \left( \bigcup_{1 \leq i < K} A_i \right) \geq 1 - \varepsilon$.
\end{claim}
\begin{proof}
Fix $\varepsilon > 0$. For each $k \in \mathbb{N}$, let $g(k)$ be the infimum of $\delta \in [0,1] $ so that there exist pairwise disjoint sets $A_i \in \mathcal{B}^0_{\widetilde{X}}$ for $1 \leq i < k$  so that each $A_i$ is $\delta$-decent with respect to $\widetilde{E}$ and $\mu_{\widetilde{X}} \left( \bigcup_{1 \leq i < k} A_i \right) \geq 1 - \frac{\varepsilon}{2}$. Note that $0 \leq g(k) \leq 1$ for all $k \in \mathbb{N}$ since $\widetilde{X}$ itself is $1$-decent.

First we show that there exists some $K \in \mathbb{N}$ so that  $g(K) = 0$.

Suppose not, that is $0 < g(k) \leq 1$ for all $k \in \mathbb{N}$. Let $f: \mathbb{N} \to (0,1]$ be defined by $f(n) := \frac{g(n)}{3}$.

Let $N$ be as given for $\mathcal{H}$ by (1) with respect to $\frac{\varepsilon}{2}$ and $f$. 

Then for each $n \in \mathbb{N}$ we can find a partition $A_{n,i}, 0 \leq i < N'_{n}$ of $X_n$ so that:
\begin{itemize}
\item $N'_n \leq N$;
	\item $A_{n,i}$ is $f(N'_n)$-decent for $1 \leq i < N'_n$;
	\item $\mu_{X_n} \left( \bigcup_{1 \leq i < N'_n} A_{n,i} \right) \geq 1 - \frac{\varepsilon}{2}$.
\end{itemize}

For some $N' \leq N$, we have that  $U := \left\{ n \in \mathbb{N} : N'_n = N' \right\} \in \mathcal{U}$.
For $1 \leq i \leq N'$, let $A_i := \prod_{n \in \mathbb{N}} A_{n,i} / \mathcal{U} \in \mathcal{B}^0_{\widetilde{X}}$. Then we have:
\begin{itemize}
\item the sets $A_{i}, 1 \leq i < N'$ are pairwise disjoint (by \L os);
	\item $\mu_{\widetilde{X}} \left( \bigcup_{1 \leq i < N'} A_{i} \right) \geq 1 - \frac{\varepsilon}{2}$ (by the definition of $\mu_{\widetilde{X}}$);
	\item  $A_{i}$ is $2 f(N')$-decent with respect to $\widetilde{E}$, for each $1 \leq i < N'$ (by Claim \ref{cla: transfer claim 2}(2)).
\end{itemize}

But since $2 f(N') < g(N')$ by the choice of $f$, this contradicts the definition of $g$. 

Hence we fix $K \in \mathbb{N}$ so that  $g(K) = 0$.

For $\alpha, \beta \in \mathbb{Q}_{[0,1]}$, let $\Gamma_{\alpha,\beta}(z)$ be the countable partial $\mathcal{L}_{\infty}$-type as in Claim \ref{cla: decent is type-def}. Note that 
\begin{gather}
	\alpha' \geq \alpha, \beta' \geq \beta \implies \Gamma_{\alpha, \beta}(z)  \vdash \Gamma_{\alpha', \beta'}(z). \label{eq: Gamma props}
\end{gather}
We consider  the  $\mathcal{L}_{\infty}$-formula
\begin{gather*}
\psi(z_1, \ldots, z_{K-1}) := \left( \neg m_{x} < 1 - \varepsilon. \bigvee_{1 \leq i < K} S_{\mathbf{X}}(x, z_i) \right) \land\\
\land \left( \forall x \left( \bigwedge_{1 \leq i \neq j < K} S_{\mathbf{X}}(x, z_i) \leftrightarrow \neg S_{\mathbf{X}}(x, z_j) \right) \right)
\end{gather*}
and the partial countable $\mathcal{L}_{\infty}$-type
\begin{gather*}
	\Gamma(z_1, \ldots, z_{K-1}) :=\left\{ \psi(z_1, \ldots, z_{K-1}) \right\} \  \cup \  \bigcup_{\gamma \in \mathbb{Q}_{(0,1]}} \bigcup_{1 \leq i < K} \Gamma_{\gamma,\gamma}(z_i).
\end{gather*}

As $g(K) = 0$, by definition of $g$ and $S_{\mathbf{X}}(x,z)$ we have that for every $\gamma > 0$ there exist some $b^{\gamma}_1, \ldots, b^{\gamma}_{K-1} \in \widetilde{P}_{\mathbf{X}}$ so that each of the pairwise disjoint sets $\left( S_{\mathbf{X}}\right)_{b^{\gamma}_i} \in \mathcal{B}^0_{\widetilde{X}}, 1 \leq i < K$ is $\gamma$-decent and $\mu_{\widetilde{X}}\left(\bigcup_{1 \leq i <K } \left(S_{\mathbf{X}}\right)_{b^{\gamma}_i}  \right) \geq 1 - \frac{\varepsilon}{2}$, hence using \eqref{eq: Linfty ultraprod}  we have
\begin{gather*}
	\widetilde{\mathcal{M}} \models \left\{ \psi \left(b^{\gamma}_1, \ldots, b^{\gamma}_{K-1} \right) \right\} \  \cup \   \bigcup_{1 \leq i < K} \Gamma_{\gamma,\gamma}(b^{\gamma}_i).
\end{gather*}

As $\Gamma$ is a countable partial type, by $\aleph_1$-saturation of the ultraproduct $\widetilde{M}$ (and \eqref{eq: Gamma props}), this implies that there exist some $b_1, \ldots, b_{K-1} \in \widetilde{P}_{\mathbf{X}}$ so that
\begin{gather*}
	\widetilde{\mathcal{M}} \models \Gamma(b_1, \ldots, b_{K-1}).
\end{gather*}

By countable additivity of the measure, a set is perfect if and only if it is $\gamma$-decent for every $\gamma \in \mathbb{Q}_{>0}$. Hence, using Claim \ref{cla: decent is type-def}, each of the pairwise disjoint sets $\left( S_{\mathbf{X}} \right)_{b_i} \in \mathcal{B}^0_{\widetilde{X}}$ is perfect with respect to $\widetilde{E}$, and $\mu_{\widetilde{X}}\left(\bigcup_{1 \leq i <K } \left(S_{\mathbf{X}}\right)_{b^{\gamma}_i}  \right) \geq \varepsilon$ (using \eqref{eq: Linfty ultraprod}). This proves Claim \ref{cla: finding up to epsilon in UP}.
\end{proof}

Now we choose a partition of $\widetilde{X}$ into countably many perfect sets in $\mathcal{B}_{\widetilde{X}}$ by induction on $t \in \mathbb{N}_{\geq 1}$, in blocks of finite sizes $K_{t} \in \mathbb{N}$, applying Claim \ref{cla: finding up to epsilon in UP} to find an almost partition into perfect sets of what remains of $\widetilde{X}$ after the previous steps, so that the total measure of the constructed  perfect sets  converges to $1$. 
More precisely, assume we have already chosen $A^{t}_{i} \in \mathcal{B}^0_{\widetilde{X}}$ for $1 \leq t \leq  T, 1 \leq i < K_{t}$ so that:
\begin{itemize}
	\item all of the sets $\left\{ A^{t}_{i}: 1 \leq t \leq T, 1 \leq i < K_{t} \right\}$ are pairwise disjoint;
	\item letting  $B^T := \bigcup_{1 \leq t \leq T, 1 \leq i < K_t} A^t_{i}$, we have
	\begin{gather*}
		\mu_{\widetilde{X}} \left( B^T \right) \geq \sum_{t=1}^{T}\frac{1}{2^{t}};
	\end{gather*} 
	\item each set $ A^{t}_{i}$ is perfect with respect to $\widetilde{E}$.
\end{itemize}

We need to choose $K_{T+1} \in \mathbb{N}$ and $A^{T+1}_{i} \in \mathcal{B}_{\widetilde{X}}^0, 1 \leq i < K_{T+1}$. 

If we already have $\mu_{\widetilde{X}}\left( B^T \right) \geq \sum_{t=1}^{T+1} \frac{1}{2^t}$, we just let $K_{T+1} := 0$ (so no new sets are added on step $T+1$).  If $ \mu_{\widetilde{X}}\left( B^T \right) < \sum_{t=1}^{T+1} \frac{1}{2^t}$, consider the remaining part $\widetilde{X}^T := \widetilde{X} \setminus B^T$. Then $\mu_{\widetilde{X}} \left(\widetilde{X}^T \right) = 1 - \mu_{\widetilde{X}}\left( B^T \right)  >  1 -  \sum_{t=1}^{T+1} \frac{1}{2^t}$, and we can choose $\varepsilon_{T} > 0$ so that 
\begin{gather}
(1 - \varepsilon_T) \mu_{\widetilde{X}} \left(\widetilde{X}^T \right)  \geq \sum_{t=1}^{T+1} \frac{1}{2^t} -  \mu_{\widetilde{X}}\left( B^T \right).\label{eq: enough of remainder}
\end{gather}

 For every $1 \leq t \leq T, 1 \leq i < K_t$ we have $A^t_i = \prod A^t_{n,i} / \mathcal{U}$ for some $A^t_{n,i} \in \mathcal{B}_{X_n}$.
 We let 
 \begin{gather*}
 	X^T_n := X_n \setminus \bigcup_{1 \leq t \leq T, 1 \leq i < K_t} A^t_{n,i} \textrm{ and } E^T_n := E_n \restriction \left( X^T_n \times Y_n \right).
 \end{gather*}

As $\mathcal{H}$ is hereditarily closed (and this is the only place where we use it), we have that $H_n^T := \left(E^T_n; X^T_n, Y_n  \right) \in \mathcal{H}$ for every $n \in \mathbb{N}$.
Then by Claim \ref{cla: finding up to epsilon in UP} applied to 
$$\prod_{n \in \mathbb{N}} H_n^T/\mathcal{U} = \left( \widetilde{E}\restriction \left( \widetilde{X}^T \times \widetilde{Y} \right);  \widetilde{X}^T, \widetilde{Y} \right)$$
 with respect to $\varepsilon_{T}$ we find some $K_{T+1} \in \mathbb{N}$ and pairwise disjoint $A^{T+1}_i \in \mathcal{B}^0_{\widetilde{X}^T} \subseteq \mathcal{B}^0_{\widetilde{X}}$ for $1 \leq i < K_{T+1}$ so that: each $A^{T+1}_i$ is perfect with respect to the binary relation $\widetilde{E}\restriction \left( \widetilde{X}^T \times \widetilde{Y} \right)$ and measures $\mu_{\widetilde{X}^T} = \lim_{n \to \mathcal{U}} \mu_{X^T_n}$ on $\widetilde{X}^T$ and $\mu_{\widetilde{Y}}$ on $\widetilde{Y}$, where $\mu_{X^T_n}$ is the uniform measure on the finite set $X^T_n$; and $\mu_{\widetilde{X}^{T}} \left( \bigcup_{1 \leq i < K_{T+1}} A^{T+1}_i \right) \geq 1 - \varepsilon$.

As  for every $A \in \mathcal{B}^0_{\widetilde{X}^T}$ we have $\mu_{\widetilde{X}^T} (A) = \frac{\mu_{\widetilde{X}}(A)}{\mu_{\widetilde{X}}\left( \widetilde{X}^T \right)} $, it follows that  $\mu_{\widetilde{X}} \left( \bigcup_{1 \leq i < K_T} A^{T}_i \right) \geq (1 - \varepsilon) \mu_{\widetilde{X}} \left( \widetilde{X}^T \right)$, hence by \eqref{eq: enough of remainder} we get
\begin{gather*}
	\mu_{\widetilde{X}} \left( \bigcup_{1 \leq t \leq T+1, \  1 \leq i < K_t} A^{t}_i \right) \geq \sum_{t=1}^{T+1} \frac{1}{2^t}.
\end{gather*} 
Note also that each of the sets $A^{T+1}_i$  is still perfect for $\widetilde{E}$ (as it is perfect for $\widetilde{E}\restriction_{\left( \widetilde{X}^T \times \widetilde{Y} \right)}$, and $\widetilde{Y}$ and $\mu_{\widetilde{Y}}$ do not change passing to $\widetilde{E}$). This concludes the inductive construction.

Finally, we get $\mu \left( \bigsqcup_{t \in \mathbb{N}, 1 \leq i < K_t} A^{t}_{i} \right) = 1$, and to get a partition of $\widetilde{X}$ into perfect sets  satisfying (1) we just add the set $\widetilde{X} \setminus \left( \bigcup_{t \in \mathbb{N}, 1 \leq i \leq K_{t}} A^t_{i} \right)$ in $\mathcal{B}_{\widetilde{X}} $ of measure $0$ (hence automatically perfect).
\end{proof}

\begin{remark}\label{rem: transfer holds definably}
	It is immediate from the proof that the equivalence in Proposition \ref{prop: proof of transfer perf part} also holds for the corresponding definable versions, that is if in Proposition \ref{prop: proof of transfer perf part}(1) we require the existence of a partition of $\widetilde{X}$ into countably many perfect sets from $\mathcal{B}^{E}_{\widetilde{X}}$, and in \ref{prop: proof of transfer perf part}(2) we require that for every $H \in \mathcal{H}$, the sets $A_t$ for $1 \leq t \leq N'$ can be chosen in $\mathcal{F}^{ E, N}_{X}$.
\end{remark}

\subsection{Comparing stable hypergraph regularities}\label{sec: stab regs rels}

\begin{theorem}\label{thm: metastable general equivalences}
	Let $\mathcal{H}$ be a hereditarily closed family of finite $k$-partite $k$-hypergraphs of the form $H = (E;X_1, \ldots, X_k)$ with $E \subseteq \prod_{i = 1}^{k} X_i$ and $X_i$ finite. Then the following are equivalent:
	\begin{enumerate}
		\item $\mathcal{H}$ satisfies (definable or non-definable) approximately perfect stable regularity (in the sense of Definition \ref{def: strong stab reg}(1));
		\item there exists some $d \in \mathbb{N}$ so that every $H \in \mathcal{H}$ is partition-wise $d$-stable (Definition \ref{def: partition-wise stability});
		\item  any ultraproduct of hypergraphs from $\mathcal{H}$  satisfies (definable or non-definable) perfect stable regularity (Definition \ref{def: perf stab reg}) with respect to the ultralimit of uniform counting measures.
	\end{enumerate}
\end{theorem}
\begin{proof}
	\textbf{(1) implies (2).} 	Suppose that there is no $d \in \mathbb{N}$ so that every $H \in \mathcal{H}$ is partition-wise $d$-stable. That is, for every $n \in \mathbb{N}$ there exists some $H_n = \left( E_n; X_{1,n}, \ldots, X_{k,n} \right) \in \mathcal{H}$ and some $1 \leq k_n \leq k$ so that, viewing $E_n$ as a binary relation on $X_{k_n, n} \times \left( \prod_{\ell \in [k] \setminus \{k_n\}} X_{\ell, n} \right)$, it is not $n$-stable. By pigeon-hole we may assume that there is some $k' \in [k]$ so that $k_n = k'$ for all $n \in \mathbb{N}$, and without loss of generality $k' = 1$.
	
	We will show that $\mathcal{H}$ does not satisfy approximately perfect stable regularity (see Definition \ref{def: strong stab reg}(1)) with $f(n) = \frac{1}{(5n + 1)^{k-1}}$ and $\varepsilon = \frac{1}{2}$. Let $N \in \mathbb{N}$ be arbitrary. Let $n \in \mathbb{N}$ be sufficiently large, to be determined later.
	By assumption there is some $H = (E; X_1, \ldots, X_k) \in \mathcal{H}$ and some $(a_i, b_{2,i}, \ldots, b_{k,i} : i \in [n])$ with $a_i \in X_1$ and $b_{\ell, i} \in X_{\ell}$ for $2 \leq \ell \leq k $ so that for any $i,j \in [k]$ we have $(a_i, b_{2,j}, \ldots, b_{k,j}) \in E \iff i \leq j$. 
	
	We can choose $j_1 = 1 < j_2 < \ldots < j_{n'} = n$ with $n' \leq 5N +1 $ so that $\left\lvert [j_{t}, j_{t+1} ) \right \rvert \leq \frac{n}{5N}$ for all $t$. Let $X'_1 := X_1$ and for $2 \leq \ell \leq k$, let $X'_{\ell} := \left\{b_{\ell, j_t} : t \in [n'] \right\}$, let $E' := E \restriction \left( X'_1 \times \ldots \times X'_k\right)$ and $H' := \left(E'; X'_1, \ldots, X'_k \right)$. As $\mathcal{H}$ is hereditarily closed, we still have $H' \in \mathcal{H}$.
	
	Now let $A_1, \ldots, A_N \subseteq X_1$ be arbitrary pairwise-disjoint sets so that $\left \lvert \bigcup_{i \in [N]} A_i \right \rvert \geq (1 - \varepsilon) |X| = \frac{1}{2} n$. Then for at least one $i^* \in [N]$ we have $\left \lvert A_{i^*} \right \rvert \geq \frac{1}{2N} n$. By the choice of the $j_t$'s it follows that there is some $t^* \in [n']$ so that both sets 
	$$\left \{ i \in [n]: a_i \in A_{i^*} \ \land \ i \leq j_{t^*} \right \} \textrm{ and } \left \{ i \in [n]: a_i \in A_{i^*} \ \land \ i > j_{t^*} \right \}$$
	 have size at least $\frac{1}{2N} n$, hence 
	\begin{gather*}
		\mu_{X'_1}\left ( A_{i^*} \cap E_{ \left( b_{2,j_{t^*}}, \ldots, b_{k, j_{t^*}} \right) } \right ) \geq \frac{1}{2N} \textrm{ and } \mu_{X'_1} \left ( A_{i^*} \setminus E_{ \left( b_{2,j_{t^*}}, \ldots, b_{k, j_{t^*}} \right) } \right ) \geq \frac{1}{2N},
	\end{gather*}
	where $\mu_{X'_1}$ is the uniform probability measure on $X'_1$.
	Note also that for every $2 \leq \ell \leq k$, we have $\mu_{X'_{\ell}} \left( \left\{ b_{\ell, j_{t^*}} \right\} \right) = \frac{1}{|X'_{\ell}|} \geq \frac{1}{5N + 1}$, hence 
	\begin{gather*}
		\mu_{X'_2 \times \ldots \times X'_k} \left( \left\{\left( b_{2,j_{t^*}}, \ldots, b_{k, j_{t^*}} \right) \right\} \right) \geq \frac{1}{(5N + 1)^{k-1}}.
	\end{gather*} 
	This shows that $A_{i^*}$ is not $\frac{1}{(5N + 1)^{k-1}}$-decent, hence not $f(N)$-decent by the choice of $f$.
	
	\textbf{(2) implies (3).} Let $d$ be such that every $H \in \mathcal{H}$ is partition-wise $d$-stable. It follows by \L os' theorem that the ultraproduct $\widetilde{H} = \left(\widetilde{E}; \widetilde{X}_1, \ldots, \widetilde{X}_k \right)$ is also partition-wise $d$-stable. But then $\widetilde{H}$ satisfies  definable perfect stable regularity by Proposition \ref{prop: stable hypergraph reg}.
	
	\textbf{(3) implies (1).} For fixed $\varepsilon, f$, for every $i \in [k]$ let $N_i$ be as given by Proposition \ref{prop: proof of transfer perf part} with respect to the partitions of $\widetilde{X}_i$ for $\widetilde{E}$ viewed as a binary relation on $\widetilde{X}_i$ and $\prod_{j \in [k] \setminus \{i\}} \widetilde{X}_i$. Then $N := \max \{ N_i : i \in [k] \}$ satisfies (1) (and satisfies the definable version of (1) if the definable version of (3)  was satisfied, by Remark \ref{rem: transfer holds definably}).
\end{proof}

Similarly, we have a one way transfer for strong stable regularity:
\begin{prop}\label{prop: one way transfer for strong stable reg}
		Let $\mathcal{H}$ be an arbitrary family of finite $k$-partite $k$-hypergraphs of the form $H = (E;X_1, \ldots, X_k)$ with $E \subseteq \prod_{i = 1}^{k} X_i$ and $X_i$ finite. 
		Assume that  any ultraproduct $\widetilde{H} = \left(\widetilde E; \widetilde{X}_1, \ldots, \widetilde{X}_{k} \right)$ of the graphs from $\mathcal{H}$   satisfies the following:
		\begin{itemize}
			\item for every $\varepsilon > 0$ there exist finitely many pairwise-disjoint sets $A^{t}_{i} \in \mathcal{B}_{\widetilde{X}_t}, 1 \leq t \leq k, 1 \leq i \leq n_t \in \mathbb{N}$ so that $A^{t}_{i}$ is perfect for $\widetilde{E}$ viewed as a binary relation on $\widetilde{X}_{t} \times \left( \prod_{s \in [k] \setminus \{t\}} \widetilde{X}_s \right)$ and $\mu_{\widetilde{X}_t} \left( \bigcup_{i=1}^{n_t} A^t_{i} \right) \geq 1 - \varepsilon \mu_{\widetilde{X}_t} \left( A^t_1 \right)$.
		\end{itemize} 
Then $\mathcal{H}$ satisfies strong stable regularity (Definition \ref{def: strong stab reg}(2)). (Moreover, if we could chose $A^{t}_i \in \mathcal{B}_{\widetilde{X}_t}$, then $\mathcal{H}$ satisfies definable strong stable regularity, Definition \ref{def: definable stable regularities}).
\end{prop}
\begin{proof}
	Similar to the proof of (2)$\implies$(1) in Proposition \ref{prop: proof of transfer perf part} (we will skip some details explained there). 

	Assume $\mathcal{H}$ does not satisfy strong stable regularity. That is, there exist $\varepsilon >0$ and $f: \mathbb{N} \to (0,1]$ so that for every $n \in \mathbb{N}$ there exists some $H_n = (E_n; X_{1,n}, \ldots, X_{k,n}) \in \mathcal{H}$ so that there are no partitions of the $X_{t,n}, 1 \leq t \leq k$ satisfying the conditions (a) and (b) in Definition \ref{def: strong stab reg}(2) for any $N' \leq n$. Let $\widetilde{H} := \prod_{n \in \mathbb{N}} H_n / \mathcal{U}$ for $\mathcal{U}$ a non-principal ultrafilter.
	
	Fix $0 < \varepsilon' < \varepsilon$. By assumption, we can choose $A^{t}_{i} \in \mathcal{B}_{\widetilde{X}_t}, 1 \leq t \leq k, 1 \leq i \leq n_t \in \mathbb{N}$ so that $A^{t}_{i}$ is perfect for $\widetilde{E}$ and $\mu_{\widetilde{X}_t} \left( \bigcup_{i=1}^{n_t} A^{t}_{i} \right) \geq 1 - \varepsilon' \mu_{\widetilde{X}_t} \left(A^t_1 \right)$. Fix $\delta >0$, and choose pairwise disjoint $A^{0,t}_{i} \in \mathcal{B}^0_{\widetilde{X}_t}$ so that $\mu_{\widetilde{X}_t}\left(A^{t}_{i}   \triangle A^{0,t}_{i}\right) < \delta \mu_{\widetilde{X}_t}(A^t_1)  \mu_{\widetilde{X}_t}(A^t_i)$ for all $1 \leq i \leq n_t$. Then $\mu_{\widetilde{X}_t} \left( \bigcup_{i = 1}^{n_t} A^{0,t}_i \right) \geq 1 - (\varepsilon' + \delta) \mu_{\widetilde{X}_t} \left(A^t_1 \right)$  and $\mu_{\widetilde{X}_t} \left(A^{t}_1 \right) \leq   \frac{1}{1 -\delta \mu_{\widetilde{X}_t}(A^t_1) } \mu_{\widetilde{X}_t} \left(A^{0,t}_1 \right) \leq \frac{1}{1 -\delta } \mu_{\widetilde{X}_t} \left(A^{0,t}_1 \right)$, so
	\begin{gather}
		\mu_{\widetilde{X}_t} \left( \bigcup_{1 \leq i < n_t} A^{0,t}_i \right) \geq 1 - \frac{(\varepsilon' + \delta)}{1 - \delta} \mu_{\widetilde{X}_t} \left(A^{0,t}_1 \right) \textrm{ for all }1 \leq t \leq k. \label{eq: str stab hyper 10}
	\end{gather} 
	
	By assumption and Lemma \ref{lem: perfect regularity equivalents}, for any $(i_1, \ldots, i_k) \in [n_1] \times \ldots \times [n_k]$ we have  $\frac{\mu \left( \widetilde{E} \cap \prod_{t=1}^{k} A^{t}_{i_t} \right)}{\mu \left( \prod_{t=1}^{k} A^{t}_{i_t} \right)} \in \{0,1\}$. Hence we still have
\begin{gather}
	\frac{\mu \left( \widetilde{E} \cap \prod_{t=1}^{k} A^{0,t}_{i_t} \right)}{\mu \left( \prod_{t=1}^{k} A^{0,t}_{i_t} \right)} \in \left[0, \frac{k\delta}{1-\delta} \right) \cup \left(1 - \frac{k\delta}{1-\delta}, 1 \right]. \label{eq: str stab hyper 11}
\end{gather}
For each $t,i$ we have $A^{0,t}_i = \prod_{n \in \mathcal{U}}A^{t}_{n,i}/\mathcal{U}$ for some sets $A^{t}_{n,i} \in \mathcal{B}_{X_{t,n}}$.

Then, using definition of $\mu_{\widetilde{X}_1 \times \ldots \times  \widetilde{X}_k}$, \eqref{eq: str stab hyper 10},  \eqref{eq: str stab hyper 11}, \L os' theorem, the fact that ultralimits commute with continuous functions,  and that $\mathcal{U}$ is closed under finite intersections, 
we can choose $S \in \mathcal{U}$ so that for every $n \in S$ we have:
\begin{itemize}
	\item for all $(i_1, \ldots, i_k) \in [n_1] \times \ldots \times [n_k]$, we have 
	\begin{gather*}
		\frac{\mu_{X_{1,n} \times \ldots \times {X}_{k,n}} \left( E_n \cap \prod_{t=1}^{k}A^t_{n,i_t}  \right)}{\mu_{X_{1,n} \times \ldots \times {X}_{k,n}}\left(\prod_{t=1}^{k}A^t_{n,i_t} \right)} \in \left[0, \frac{k\delta}{1-\delta} + \delta\right) \cup \left(1 - \left( \frac{k\delta}{1-\delta} + \delta \right) , 1 \right];
	\end{gather*}
\item the sets $\left( A^t_{n,i} : 1 \leq i \leq n_t \right)$ are pairwise disjoint, for all $1 \leq t \leq k$;
\item $\mu_{X_{t,n}} \left( \bigcup_{1 \leq i \leq n_t} A^t_{n,i} \right)  \geq 1 - 2 \frac{( \varepsilon' +  \delta)}{1 - \delta} \mu_{X_{t,n}}\left(A^t_{n,1} \right)$.
\end{itemize} 

But if we took $\varepsilon', \delta > 0$
 so that 
 $$2 \frac{( \varepsilon' +  \delta)}{1 - \delta} < \varepsilon \textrm{ and } \frac{k\delta}{1-\delta} + \delta < f(\max \{n_1, \ldots, n_k \}),$$
 this contradicts the choice of the $H_n$'s for $n \in S$ with $n > \max \{n_1, \ldots, n_k \}$.
 
 The ``moreover'' part follows from the proof: if we could choose the sets $A^{t}_{i} \in \mathcal{B}_{\widetilde{X}_t}^{E}$, as $\mathcal{B}_{\widetilde{X}_t}^{E}$ is the $\sigma$-algebra generated by $\bigcup_{K \in \mathbb{N}} \mathcal{F}_{\widetilde{X}_t}^{E, K} $, we could choose all of  $A^{0,t}_i$ in $\mathcal{F}_{\widetilde{X}_t}^{E, K} $ for some  $K \in \mathbb{N}$, hence by \L os' theorem we could choose $A^{t}_{n,i} \in \mathcal{F}^{E,K}_{X_{t,n}}$ for all $t, n,i$, and conclude that definable strong stable regularity holds with $N$ replaced by $\max\{N,K\}$.
\end{proof}

We will show next that restricting to graphs, all versions of finitary and infinitary stable regularity are equivalent, and are equivalent to stability. We will need a well-known observation that any bounded size partition of large half-graphs contains irregular pairs (see \cite[Section 1.8]{komlos1995szemeredi} for a discussion and some references). We include a proof in our setting for completeness:
\begin{lemma}\label{lem: half-graphs fail stab reg}
	Let $A = B = [n]$ and $E \subseteq A \times B$ be the bipartite graph $E = \left\{(a,b) \in A \times B : a < b \right\}$. For any $t \in \mathbb{N}$ there exists $n$ large enough so that: if $A=\bigsqcup_{i\in I} A_i, B = \bigsqcup_{j \in J} B_j$ are partitions with $|I|,|J| \leq t$, then for some $(i,j) \in I \times J$ we have $\frac{\left \lvert E \cap \left( A_i \times B_j \right) \right\rvert }{\left \lvert A_i \times B_j \right\rvert} \in \left( \frac{1}{4}, \frac{3}{4}\right)$.
 \end{lemma}
\begin{proof}
	Fix $0 < \varepsilon \leq \frac{1}{4}$, and let $\delta := (3\varepsilon)^{\frac{1}{2}} < 1$. Let $\mu$ be the uniform probability measure on $[n]$.
	Given a partition of $[n]$ of size $t$, for each $A_i$ we choose some $a_i, a'_i \in A_i$ so that:
	\begin{enumerate}
		\item $a_i < a'_i$,
		\item $\mu(\{ a \in A_i : a < a_i \}) = \delta \mu(A_i)$,
		\item $\mu(\{ a \in A_i : a'_i < a \}) = \delta \mu(A_i)$.
	\end{enumerate}
	(More precisely, we can get arbitrary close to equality in (2) and (3) by choosing $n$ large enough with respect to $t$, we will ignore this for notational simplicity.)
We refer to  $(a_i,a'_i) := \{ a \in A : a_i < a < a'_i \}$ as the \emph{middle interval} of $A_i$ (though it may be not contained in $A_i$). Similarly, we define the middle intervals $(b_i,b'_i)$ of $B_i$. Then $\mu(\bigcup_{i \in I}(a_i, a'_i)) \geq (1 - 2 \delta)$, and similarly for $B$. Hence 
\begin{gather*}
1 - 4 \delta \leq 	\mu \left( \left( \bigcup_{i \in I}(a_i, a'_i) \right) \cap   \left( \bigcup_{j \in J}(b_j, b'_j)  \right) \right) = \mu \left( \bigcup_{i \in I, j \in J} (a_i,a'_i) \cap (b_j, b'_j) \right),
\end{gather*}
so $\mu((a_i,a'_i) \cap (b_j, b'_j)) \geq \frac{1}{t^2} (1 - 4 \delta)$ for some $(i,j) \in I \times J$.  In particular, $(a_i,a'_i) \cap (b_j, b'_j) \neq \emptyset$, so $a_i < b'_j$ and $b_j < a'_i$. Let $A' := \{ a \in A_i : a < a_i \}$, $B' := \{ b \in B_j : b'_j < b \}$, $A'' := \{ a \in A_i : a'_i < a \}$, $B'' := \{ b \in B_j : b < b_j \}$. Note that $A',A'' \subseteq A_i$ and $B',B'' \subseteq B_j$.
Then $A' \times B' \subseteq E$, $ \left(A'' \times B''\right) \cap E = \emptyset$ and $ \mu(A' \times B'), \mu(A'' \times B'') \geq \delta^2 \mu(A_i) \mu(B_j)$. Hence $E$ cannot be $\varepsilon$-regular on $(A_i, B_j)$ assuming $\frac{\varepsilon}{\delta^2} < \frac{1}{2}$, which holds by our choice of $\delta$.
\end{proof}

	\begin{theorem}\label{thm: all versions of reg coincide for stable graphs}
	Let $\mathcal{H}$ be a hereditarily closed family of finite bipartite graphs of the form  $H = (E;X,Y)$ with $E \subseteq X \times Y$. The following are equivalent:
	\begin{enumerate}
		\item any ultraproduct of the graphs from $\mathcal{H}$   satisfies (definable) perfect stable regularity with respect to the ultralimit of uniform counting measures;
		\item the family $\mathcal{H}$ satisfies (definable or non-definable) approximately perfect stable regularity;

		\item the family $\mathcal{H}$ satisfies (definable or non-definable) strong stable regularity;

		\item the family $\mathcal{H}$ satisfies (definable or non-definable) stable regularity;

		\item there is some $d \in \mathbb{N}$ so that all graphs in $\mathcal{H}$ are $d$-stable.
	\end{enumerate}
\end{theorem}
\begin{proof}
	(1) implies (2). By Theorem \ref{thm: metastable general equivalences} for $k = 2$ (namely, (3) implies (1) there).

	(2) implies (3), (3) implies (4). By Remark \ref{rem: meta implies strong implies stable}.

	(4) implies (5). Fix $\varepsilon = \frac{1}{4}$. Let $N \in \mathbb{N}$ be arbitrary.  Assume that for every $d$, there is some $H_d \in \mathcal{H}$ which is not $d$-stable. Then, using that  $\mathcal{H}$ is hereditarily closed, for every $n$ the family $\mathcal{H}$ contains a half-graph $H_n$ with both sides of size $n$.
	But then, for $n$ sufficiently large with respect to $N$, in view of Lemma  \ref{lem: half-graphs fail stab reg} there is no partition of the vertices of size $N$ satisfying Definition \ref{def: strong stab reg}(3).
		
	(5) implies (1). By Theorem \ref{thm: metastable general equivalences} for $k = 2$ (namely, (2) implies (3) there).
\end{proof}

%
%
%
%
%
%

However, there exists a hereditarily closed family of 3-hypergraphs satisfying strong stable regularity, but not approximately perfect stable regularity:
\begin{example}\label{ex: x=y<z stability props}
	For $n \in \mathbb{N}$, let $H_n$ be the 3-hypergraph on $X_n \times Y_n \times Z_n$, $X_n = Y_n = Z_n := [n]$ with the edge relation 
	$$E_n = \left\{ (x,y,z) \in X_n \times Y_n \times Z_n: x \leq y = z \right\}.$$
	 Let $\mathcal{H}$ be the hereditary closure of the family $\left\{ H_n : n \in \mathbb{N}\right\}$.
	
	As for every $n$, $E_n$ is not $n$-stable viewed as a binary relation  between $X_n$ and $Y_n \times Z_n$, by (1)$\Rightarrow$(2) in Theorem \ref{thm: metastable general equivalences}, $\mathcal{H}$ does not satisfy approximately perfect stable regularity.
	
	On the other hand, for all $n \in \mathbb{N}$ we have: for every $z \in Z_n$, the binary fiber $\left(E_n \right)_z \subseteq X_n \times Y_n$ is $2$-stable, and $E_n$ viewed as a binary relation between $X_n \times Y_n$ and $Z_n$ is also $2$-stable (and this also holds for all graphs in $\mathcal{H}$). Hence $\mathcal{H}$ satisfies (definable) strong stable regularity by Corollary \ref{cor: finitary one dir stable other slice-wise}.
\end{example}

\begin{problem}
	Is there a hereditarily closed family of hypergraphs satisfying stable regularity, but not strong stable regularity?
\end{problem}

\section{One Direction of slice-wise Stability}\label{sec: one dir or slicewise stable}

The weakest property we consider is the case of $3$-graphs which fail to be even slice-wise stable, but which have some ``directions'' of slice-wise stability viewed as partite $3$-graphs. That is, we have $E\subseteq X\times Y\times Z$ and, say, $E_x$ is $d$-stable for all $x\in X$. For example, if $X=Y=Z=(0,1)$ and $E=\{(x,y,z)\mid x<\max\{y,z\}\}$ then $E_x$ is $2$-stable for all $x$, but any $E_y$ or $E_z$ is unstable (just choose a ladder entirely below $y$ or $z$ respectively).

Most of what we can hope to show in this case is nearly obvious: we can apply stable regularity to each $E_x$ separately, so for each $x$ we get partitions of $Y=\bigsqcup_i A^i_x$ and $Z=\bigsqcup_j B^j_x$ into perfect sets for $E_x$. We can then assemble these into sets of pairs---taking, for instance, $A^i$ to be $\{(x,y)\mid y\in A^i_x \}$.

In fact, for later purposes, we do need to extract slightly more from this situation: we wish to obtain a partition $X\times Y=\bigsqcup_i A^i$ so that each $A^i_x$ is perfect for $E_x$, and also the sets $A^i$ are measurable. (In the finite setting, this amounts to showing that the sets $A^i$ can be chosen from $\mathcal{F}^{E,N}_{X\times Y}$ for a suitable bound $N$, see Definition \ref{def: algebras of fibers}; we do not show this here, although it follows from the following proposition by the same arguments we use in later sections.)

\begin{theorem}\label{lem: uniform perf decomp}
Suppose that $E \in \mathcal{B}_{X \times Y \times Z}$ and $E_{x} \in \mathcal{B}_{Y \times Z}$ is $\mu$-stable for almost all $x \in X$. Then there is a partition of $X \times Y$ into countably many sets $A^i \in \mathcal{B}^{E}_{X \times Y}, i \in \omega$, so that for almost every $x \in X$, $\left(A^i_x : i \in  \omega \right)$ is a partition of $Y$ into countably many sets perfect for $E_x$ (viewed as a binary relation on $\left( X \times Y \right) \times Z$).
\end{theorem}
\begin{proof}

By transfinite induction on a countable ordinal $\alpha$ we define  sets $A^{\alpha} \in \mathcal{B}^{E}_{X \times Y}$ satisfying: for almost every $x \in X$, 
\begin{itemize}
	\item  the fiber $A^{\alpha}_x \in \mathcal{B}^{E}_{Y}$ is perfect for $E_x \in \mathcal{B}_{Y \times Z}$;
	\item the sets $A^{\beta}_x, \beta \leq \alpha$  are disjoint. 
\end{itemize}
Let $\alpha$ be an arbitrary countable ordinal, and assume we have already defined $A^{\beta} \in \mathcal{B}^{E}_{X \times Y}$ for all $\beta < \alpha$.

Let $\bar{A}^{\alpha} := \bigcup_{\beta < \alpha}A^{\beta}$, then $\bar{A}^{\alpha}\in \mathcal{B}^{E}_{X \times Y}$ as $\alpha$ is countable (so $\bar{A}^{0} = \emptyset$).
Let  $X_{\alpha, \emptyset} := \left\{ x \in X : \mu \left( Y \setminus \bar{A}^{\alpha}_x \right) = 0 \right\}$, we have $X_{\alpha, \emptyset} \in \mathcal{B}^{E}_{X}$ by Lemma \ref{lem: alg of fibers closed under meas} (we will continue to use Lemma \ref{lem: alg of fibers closed under meas} freely throughout the proof to conclude measurability of various sets with respect to the algebras of $E$-fibers).

~

\noindent \textbf{Case 1:} $\mu \left( X_{\alpha, \emptyset} \right) = 1$.

Then we terminate the construction. It follows  that for almost every $x \in X$, $\left( A^{\beta}_x : \beta < \alpha \right)$ is a partition of $Y$ (up to measure $0$) into countably many sets in $\mathcal{B}^{E}_{Y}$ perfect for $E_x$. Hence by Fubini, also $\left( A^{\beta}: \beta < \alpha \right)$ is a partition, up to measure $0$, of $X \times Y$ into countably many sets.

~

\noindent \textbf{Case 2:} $\mu \left( X_{\alpha, \emptyset} \right) < 1$.

We will need the following claim:
\begin{claim}\label{cla: perfect sets measure one}
There exists a countable sequence $\bar{c}^{\alpha} = \left(c^\alpha_i : i \in \omega \right) \in Z^{\omega}$ so that for the set  $P^{\alpha, \bar{c}^{\alpha}} \in \mathcal{B}^{E}_{X}$ defined by 
\begin{gather*}
 \Bigg\{ x \in X \setminus X_{\alpha, \emptyset} : \bigvee_{m \in \omega} \bigvee_{I \in  \omega^m} \bigvee_{\sigma \in 2^m} \Big(\left(E_x \right)^{\sigma}_{\left(c_i : i \in I \right)} \cap \left(Y \setminus \bar{A}^{\alpha}_x \right)  \textrm{ is perfect for } E_x \\
	\textrm{ and } \mu_{Y}\left( \left(E_x \right)^{\sigma}_{\left(c_i : i \in I \right)} \cap \left(Y \setminus \bar{A}^{\alpha}_x \right)  \right) > 0 \Big) \Bigg\} 
\end{gather*}
 we have $P^{\alpha,\bar{c}^{\alpha}} =^0 X \setminus X_{\alpha, \emptyset}$.
\end{claim}
\begin{proof}
	For $d \in \omega$, let $X_d \in \mathcal{B}^{E}_{X}$  be the set
	$$\left\{ x \in X \setminus X_{\alpha,\emptyset} : d \textrm{ is minimal so that }E_x\restriction_{\left(Y \setminus \bar{A}^{\alpha}_x \right) \times Z} \textrm{ is } d\textrm{-}\mu\textrm{-stable} \right\}.$$
	 By assumption, as $E_x$ is $\mu$-stable for almost every $x\in X$, we have  $X \setminus X_{\alpha,\emptyset} =^0 \bigsqcup_{d \in \omega} X_d$. Fix $d$, without loss of generality $\mu_{X}(X_d) >0$, and consider the set
\begin{gather*}
	Q^{\alpha,d} := \Bigg\{ (x, \bar{z}) \in X_d \times Z^d : \bigvee_{\sigma \in 2^d} \Big( \left( E_x \right)^{\sigma}_{\bar{z}} \cap \left(Y \setminus \bar{A}^{\alpha}_x \right) \textrm{ is perfect for } E_x \\
	\textrm{and } \mu_{Y} \left( \left( E_x \right)^{\sigma}_{\bar{z}} \cap \left(Y \setminus \bar{A}^{\alpha}_x \right) \right) > 0\Big)
	\Bigg\} \in \mathcal{B}^{E}_{X \times Z^d}.
\end{gather*}
For every $x \in X_d \subseteq X \setminus X_{\alpha, \emptyset}$, by Lemma \ref{lem: pos meas of perf sets for stable E} we have $\mu_{Z^d} \left( Q^{\alpha,d}_{x} \right) > 0$. Then by Fubini $\mu_{X \times Z^d} \left( Q^{\alpha,d}\right) > 0$, and so $\mu_{X}\left(Q^{\alpha,d}_{\bar{c}^{\alpha,d,0}} \right) > 0$ for some $\bar{c}^{\alpha,d,0} \in Z^d$.
If $\mu_{X} \left( Q^{\alpha,d}_{\bar{c}^{\alpha,d,0}} \right) < \mu_{X} (X_d)$, replacing $X_d$ by $X_d \setminus Q^{\alpha,d}_{\bar{c}^{\alpha,d,0}}$ and repeating the same argument, we find some $\bar{c}^{\alpha,d,1} \in Z^{d}$ so that $\mu_{X} \left( Q^{\alpha,d}_{\bar{c}^{\alpha},d,1} \setminus Q^{\alpha,d}_{\bar{c}^{\alpha},d,1} \right) >0 $.  Continuing by induction on 
 a countable ordinal $\beta < \omega_1$ we can  choose tuples $\bar{c}^{\alpha, d, \beta} \in Z^d$ so that $\mu_{X} \left(\bigcup_{\gamma < \beta} Q^{\alpha,d}_{\bar{c}^{\alpha, d, \gamma}} \right)$ is strictly increasing with $\beta$, which is impossible since $\mathbb{R}$ does not contain strictly increasing sequences of order type $\omega_1$, so we must get $X_d =^0 \bigcup_{\gamma < \beta} Q^{\alpha,d}_{\bar{c}^{\alpha, d, \gamma}} $  for some countable ordinal $\beta$. We let $\bar{c}^{\alpha,d}$ be the countable sequence obtained by concatenating $ \left(\bar{c}^{\alpha,d,\gamma} : \gamma < \beta\right)$.
Similarly, we obtain a sequence $\bar{c}^{\alpha,d}$ for each $d \in \omega$, and let $\bar{c}^{\alpha}$ be the countable sequence obtained by concatenating $\left(\bar{c}^{\alpha,d} : d \in \omega \right)$ (and reordering to make the index set $\omega$). By construction we have $X_d \subseteq P^{\alpha,\bar{c}_{\alpha}} $ for all $d \in \omega$, hence $P^{\alpha,\bar{c}_{\alpha}} =^{0} X \setminus X_{\alpha, \emptyset}$, proving the claim.
\end{proof}

Now let $\bar{c}^{\alpha}$ be as given by Claim \ref{cla: perfect sets measure one}.
Let $\prec$ be an arbitrary well-ordering on the countable set $\bigcup_{m \in \omega}\left( \omega^{m} \times 2^{m} \right)$.
For $I \in \omega^{m}$ and $\sigma \in 2^{m}$, let 
\begin{gather*}
	X^{0}_{\alpha, I, \sigma} := \Bigg\{ x \in X \setminus X_{\alpha, \emptyset} : \left(E_x \right)^{\sigma}_{\left(c_i : i \in I \right)} \cap \left(Y \setminus \bar{A}^{\alpha}_x \right)  \textrm{ is perfect for } E_x \\
	\textrm{ and } \mu_{Y}\left( \left(E_x \right)^{\sigma}_{\left(c_i : i \in I \right)} \cap \left(Y \setminus \bar{A}^{\alpha}_x \right)  \right) > 0  \Bigg\} \in \mathcal{B}^{E}_{X},
\end{gather*}
and let 
\begin{gather*}
	X_{\alpha,I, \sigma} := X^0_{\alpha, I, \sigma} \setminus \left( \bigcup_{(J,\tau) \prec (I, \sigma)}  X^{0}_{\alpha, J ,\tau}\right) \in \mathcal{B}^{E}_{X}.
\end{gather*}
By Claim \ref{cla: perfect sets measure one} and construction, $\left\{ X_{\alpha, \emptyset} \right\} \cup \left\{ X_{\alpha,I,\sigma} : (I,\sigma) \in \bigcup_{m \in \omega} \left( \omega^n \times 2^{m} \right) \right\}$ is a partition of $X$, up to measure $0$.
We define $A^{\alpha} \in \mathcal{B}^{E}_{X \times Y}$ via 
\begin{gather*}
	A^{\alpha} := \bigcup_{(I, \sigma) \in \bigcup_{m \in \omega} \left(\omega^{m} \times 2^{m} \right)} \Bigg\{ (x,y) \in X \times Y : x \in X_{\alpha, I, \sigma} \ \land  \\
	y \in  \left(  \left(E_x \right)^{\sigma}_{\left(c_i : i \in I \right)} \cap \left(Y \setminus \bar{A}^{\alpha}_x \right) \right)  \Bigg\}.
\end{gather*}

Finally, we claim that the construction must terminate in Case 1 for some countable ordinal $\alpha$.  By construction, for each $\alpha$ we have $\mu \left( X \setminus X_{\alpha, \emptyset} \right) > 0$ and for every $x \in X \setminus X_{\alpha, \emptyset}$ we have $\mu \left( A^{\alpha}_x \right) > 0$
, hence by Fubini 
$\mu \left( A^{\alpha} \right) \geq \int_{X \setminus X_{\alpha, \emptyset}} \mu \left(A^{\alpha}_x \right)> 0$. It follows that $\mu \left( \bar{A}^{\alpha} \right) < \mu \left( \bar{A}^{\beta} \right)$ for all countable ordinals $\alpha < \beta $. But this is a contradiction as there are no strictly increasing sequences of order type $\omega_1$ in $\mathbb{R}$.
\end{proof}

\section{Two Directions of slice-wise Stability}\label{sec: two dirs of slicewise stab}

A more interesting situation is when $E$ is slice-wise stable in two directions---that is, we have $E\subseteq X\times Y\times Z$ so that all $E_x$ and also all $E_y$ are $d$-stable. For a simple (but dull) example, if $X=Y=Z=(0,1)$ then $E=\{(x,y,z)\mid x<y\}$ has this property.

In this case, we will show that we can actually obtain a partition of $X\times Y$ into perfect sets.

The following example (which is just an infinitary version of the example used in the proof of (1)$\Rightarrow$(2) in Theorem \ref{thm: metastable general equivalences})  demonstrates that we cannot hope to have also  a partition of $Z$ into perfect sets for $E \subseteq \left(X \times Y \right) \times Z$, as we did with ordinary stability, even when the slices of $E$ are stable in all three directions.

\begin{example}\label{ex: x=y<z}
	Let $X = Y = Z := [0,1]$ and  $E := \left\{ (x,y,z) : x = y < z \right\}$, then $E$ is slice-wise stable (in all three directions).
	Place the Lebesgue measure on $Z$, and place discrete measures on $X$ and $Y$ which place a positive measure on each rational number in $[0,1]$. Now if $A \subseteq Z$ has positive Lebesgue measure, we can always choose $q \in \mathbb{Q} \cap [0,1]$ so that both $A \cap [0,q)$ and $A \cap (q, 1]$ have positive measure, that is $0 < \mu \left( E_{(q,q)} \cap A \right) < \mu \left(A \right)$. But $\mu \left(\left\{ (q,q) \right\} \right) >0$, so the set  $A$ is not perfect.
\end{example}

We first need the following technical lemma, showing the existence of a maximum local symmetrization.

\begin{lemma}\label{lem: loc sym lemma}
	Assume $X$ and $Y$ are sorts in a graded probability space, and $A \subseteq X \times Y$ with $A \in \mathcal{B}_{X \times Y}$ and $\mu(A) > 0$. Then there exist sets $U \in \mathcal{B}_{X}, V \in \mathcal{B}_{Y}$ so that:
	\begin{enumerate}
		\item $\mu \left( A \cap \left( U \times Y \right) \right) > 0$,
		\item $\mu \big( \left( A \cap \left( U \times Y \right) \right) \triangle \left( A \cap \left( X \times V \right) \right) \big) =  0$,
		\item for any $U' \subseteq U, U' \in \mathcal{B}_{X}$ with both $\mu \left( A \cap \left( U' \times Y \right) \right) >0$ and $\mu \left( A \cap \left( \left( U \setminus U' \right) \times Y \right) \right) >0$, for any $V' \subseteq Y, V' \in \mathcal{B}_{Y}$ we have 
		\begin{equation*}
			\mu \Big( \left( A \cap \left( U' \times Y \right) \right) \triangle \left( A \cap \left( U \times V' \right) \right) \Big) >0.
		\end{equation*}
	\end{enumerate}
	In particular, by (1) and (2) we have $\mu \left(A \cap (U \times V) \right) > 0$, hence by Fubini $\mu(U), \mu(V) >0$.
	\end{lemma}
\begin{proof}
	Assume that conclusion fails. Let $Y^{-} := \left\{ y \in Y : \mu(A_{y}) = 0 \right\} \in \mathcal{B}_{Y}$. As $\mu(A) > 0$, by Fubini we have $\mu(Y^{-}) < 1$; and also there exists some $\varepsilon > 0$ such that the set $A^{+} := \left\{ x \in X: \mu(A_x) \geq \varepsilon \right\} \in \mathcal{B}_{X}$ has positive measure. Let $U_0 := X, V_{0} := Y \setminus Y^{-}$, in particular $\mu_{X}(V_0) > 0$.
	
By transfinite induction on a countable ordinal $\alpha$, we try to choose sets $U_{\alpha} \in \mathcal{B}_{X}, V_{\alpha} \in \mathcal{B}_{Y}$ satisfying:
		\begin{enumerate}[(a)]
		\item $\mu \left( U_{\alpha} \cap A^{+} \right) >0$ (which implies $\mu \left( A \cap \left( U_{\alpha} \times Y\right) \right) > 0$ by Fubini),
		\item $\mu \big( \left( A \cap \left( U_{\alpha}\times Y \right) \right) \triangle \left( A \cap \left( X \times V_{\alpha} \right) \right) \big) =  0$.
	\end{enumerate}
	Note that $U_0, V_0$ satisfy (a),(b), and suppose we have chosen $U_{\alpha}, V_{\alpha}$. Note also that if $U_{\alpha}, V_{\alpha}$ satisfy (a),(b), then they also satisfy: 
 \begin{itemize}[(c)]
 	\item $\mu \left( V_{\alpha} \right) \geq \varepsilon$.
 \end{itemize}
 Indeed, suppose not, that is $\mu \left( V_{\alpha} \right) < \varepsilon$. Let $U'_{\alpha} := U_{\alpha} \cap A^{+}$, by (a) we have $\gamma := \mu \left( U'_{\alpha} \right) > 0$. On the one hand, by Fubini,  $\mu \left(A \cap  \left( U'_{\alpha} \times V_{\alpha}\right)\right) < \varepsilon \gamma$. On the other, by definition of $A^{+}$ and Fubini, $\mu \left(A \cap \left( U'_{\alpha} \times Y \right) \right) \geq \varepsilon \gamma$. Then $\mu \left(A \cap \left(U'_{\alpha} \times (Y \setminus V_{\alpha}) \right) \right) > 0$. But $U'_{\alpha} \times (Y \setminus V_{\alpha})$ is contained in $U_{\alpha}\times Y$ and disjoint from $X \times V_{\alpha} $, contradicting (b). This proves (c).

	Since the condition (3) of the lemma is not satisfied for $U_{\alpha}, V_{\alpha}$, there exist sets $U' , V'$ such that:
 \begin{itemize}
 \item $U' \subseteq U_{\alpha}$, $U' \in \mathcal{B}_{X}$ and $V' \subseteq Y, V' \in \mathcal{B}_{Y}$,
 	\item $\mu \left( A \cap \left( U' \times Y \right) \right) >0$,
\item $\mu \left( A \cap \left( \left( U_{\alpha} \setminus U' \right) \times Y \right) \right) >0$,
	\item $\mu \left( \left( A \cap \left( U' \times Y \right) \right) \triangle \left( A \cap \left( U_{\alpha} \times V' \right) \right) \right) =0$.
 \end{itemize}
Then, using in particular (b) for $U_{\alpha}, V_{\alpha}$, the sets $U_{\alpha}\setminus U', V_{\alpha} \setminus V'$ also satisfy these last four properties. And as $U_{\alpha}$ satisfies (a), at least one of $U' \cap A^{+}$ or $\left( U_{\alpha} \setminus U' \right) \cap A^{+}$ has positive measure, so we take it to be $U_{\alpha+1}$, and the corresponding choice of $V'$ or $V_{\alpha} \setminus V'$ to be $V_{\alpha + 1}$. Then $U_{\alpha+1}, V_{\alpha+1}$ also satisfy (a) and (b).

 Assume $\lambda$ is a countable limit ordinal, and we have decreasing families of sets $U_{\alpha} \in \mathcal{B}_{X}$ and $V_{\alpha} \in \mathcal{B}_{Y}$ satisfying (a),(b) for  $\alpha < \lambda$. Let $U_{\lambda} := \bigcap_{\alpha < \lambda} U_{\alpha} \in \mathcal{B}_{X}$ and  $V_{\lambda} := \bigcap_{\alpha < \lambda} V_{\alpha} \in \mathcal{B}_{Y}$. Then $U_{\lambda}, V_{\lambda}$ still satisfy (b) as
 \begin{gather*}
 	\left( A \cap \left( \left( \bigcap_{\alpha < \lambda} U_{\alpha} \right) \times Y \right) \right) \triangle \left( A \cap \left( X \times \left( \bigcap_{\alpha < \lambda} V_{\alpha} \right) \right) \right)\\
 	=\left(\bigcap_{\alpha < \lambda} \left( A \cap \left(  U_{\alpha}  \times Y \right)  \right) \right) \triangle \left( \bigcap_{\alpha < \lambda} \left( A \cap \left( X \times V_{\alpha} \right) \right) \right)\\
 	\subseteq \bigcup_{\alpha < \lambda} \left( \left( A \cap \left( U_{\alpha}\times Y \right) \right) \triangle \left( A \cap \left( X \times V_{\alpha} \right) \right)  \right),
 \end{gather*}
and $\mu(V_{\lambda}) \geq \varepsilon$, as $\mu$ is a countably additive probability measure and $V_{\lambda}$ is an intersection of a countable family of decreasing sets all of measure $\geq \varepsilon$ (using (a)+(b) imply (c)). Assume (a) fails, that is $\mu \left( U_{\lambda} \cap A^{+} \right) =0$.  Then we can restart the transfinite construction with $U'_0 := U_0 \setminus U_{\lambda}$ and $V'_0 := V_0 \setminus V_{\lambda}$, noting that  $U'_0, V'_0$ satisfy (a),(b).
Note that $\mu(V'_0) \leq \mu(V_0) - \varepsilon$ (as $V_{\lambda} \subseteq V_0$ and $\mu(V_{\lambda}) \geq \varepsilon$), hence every time we do this we remove measure $\varepsilon$ from our starting set $V_0$. 
As (a),(b) imply (c), we must have $\mu(V'_0) \geq \varepsilon$, so this can happen only finitely many times. This means that (after finitely many restarts), we can choose $U_{\alpha}, V_{\alpha}$ satisfying (a),(b) for every countable ordinal, and moreover by construction we have $\mu(U_{\alpha}) < \mu(U_{\beta})$ for all countable ordinals $\beta < \alpha$. Which is a contradiction as $[0,1]$ cannot contain a strictly decreasing sequence of order type $\omega_1^{*}$ (the inverse ordering on the ordinal $\omega_1$).
\end{proof}

By a repeated application of Lemma \ref{lem: loc sym lemma} we get the following:
\begin{cor}\label{cor: part of binary into symm sets}
	Assume $(X,Y)$ is a graded probability space and $A \subseteq X \times Y$ with $A \in \mathcal{B}_{X \times Y}$. Then there exist countable partitions $X = \bigsqcup_{i \in \omega} U_i$ with $U_i \in \mathcal{B}_{X}$ and $Y = \bigsqcup_{i \in \omega} V_i$ with $V_i \in \mathcal{B}_{Y}$ such that for each $i \in \omega$ we have:
	\begin{enumerate}
		\item $\mu \left( \left( A \cap \left( U_i \times Y \right) \right) \triangle \left( A \cap \left( X \times V_i \right) \right)  \right) = 0$,
		\item for any $U' \subseteq U_i, U' \in \mathcal{B}_{X}$ such that both $\mu \left(A \cap \left(U' \times Y \right) \right) >0$ and $\mu \left(A \cap \left((U_i \setminus U') \times Y \right) \right) >0$, for any $V' \subseteq V_i, V' \in \mathcal{B}_{Y}$ we have $\mu \big( \left( A \cap \left( U' \times Y \right) \right) \triangle \left( A \cap \left( U_i \times V' \right) \right)  \big) > 0$.
	\end{enumerate}
\end{cor}
\begin{proof}
		Let $ U_0 := \emptyset, V_0 := \emptyset$. 
		Given a countable ordinal $\alpha$, assume we have already chosen $U_{\beta}, V_{\beta}$ of positive measure satisfying the requirement  for all $\beta < \alpha$. Let 
		$$A' := A \cap \left( \left( X \setminus \left( \bigcup_{\beta < \alpha} U_{\beta} \right) \right) \times \left( Y \setminus \left( \bigcup_{\beta < \alpha} V_{\beta} \right) \right) \right) \in \mathcal{B}_{X,Y}.$$
If $\mu(A') > 0$, we let $U_{\alpha}, V_{\alpha}$ with $\mu(U_{\alpha}), \mu(V_{\alpha}) > 0$ be as given by Lemma \ref{lem: loc sym lemma}, in particular $\mu \left(A \cap \left( U_{\alpha} \times V_{\alpha} \right) \right) > 0$.
From the definition of $A'$, we may assume that these new sets are disjoint from all the previous ones. As $\mu(A') < \mu(A)$, for some countable ordinal $\alpha$ we must get $\mu(A') = 0$. Then we let $U_{\alpha} := X \setminus \left( \bigcup_{\beta < \alpha} U_{\beta} \right), V_{\alpha} := Y \setminus \left( \bigcup_{\beta < \alpha} V_{\beta} \right)$ and obtain a required partition.
\end{proof}

\begin{remark}
Note that for a partition satisfying condition (1) in  Corollary \ref{cor: part of binary into symm sets}, $A$ is almost contained in the rectangles on the diagonal, that is $\mu \left(A \setminus \bigcup_{i \in \omega} \left(U_i \times V_i  \right)\right) = 0$.
\end{remark}

\begin{theorem}\label{prop: part XY into perf sets}
Suppose that $E \in \mathcal{B}_{X \times Y \times Z}$, $E_{x} \in \mathcal{B}_{Y \times Z}$ is $\mu$-stable for almost all $x \in X$, and $E_y \in \mathcal{B}_{X \times Z}$ is $\mu$-stable for almost all $y \in Y$. Then there is a partition of $X \times Y$ into $\mathcal{B}^{E}_{X \times Y}$-measurable sets perfect for $E$, viewed as a binary relation on $\left( X \times Y \right) \times Z$.
\end{theorem}
\begin{proof}
	As $E_x$ is $\mu$-stable for almost all $x \in X$, applying Lemma \ref{lem: uniform perf decomp} twice (for $E_x$ and for $E^*_x$, using Corollary \ref{cor: stab pres under opp}), we find partitions $X \times Y = \bigsqcup_{i \in \omega} A^i$ and $X \times Z = \bigsqcup_{j \in \omega} B^j$ so that:
	\begin{enumerate}
		\item $A^i \in \mathcal{B}^{E}_{X \times Y}, B^j \in \mathcal{B}^{E}_{X \times Z}$ for all $i,j \in \omega$;
		\item for almost every $x \in X$, $\left( A^i_x : i \in \omega \right)$ is a partition of $Y$ into perfect sets for $E_x \in \mathcal{B}_{Y \times Z}$;
		\item for almost every $x \in X$, $\left(B^j_x : i \in \omega \right)$ is a partition of $Z$ into perfect sets for $E_x \in \mathcal{B}_{Y \times Z}$.
	\end{enumerate}
	It follows by Lemma \ref{lem: 01 dens on perf sets} that for each $i,j \in \omega$ and almost every $x \in X$, we have that:
	\begin{itemize}
		\item  either $A^{i}_x \times B^{j}_x \subseteq^0 E_x$,
		\item or $\left(A^{i}_x \times B^{j}_x\right) \cap E_x =^0 \emptyset$.
	\end{itemize} 
	
	For each $i = \{ 0, \ldots, i-1\} \in \omega$ and $s \subseteq i$, let
	$$A^{i,s} := \left\{ (x,y) \in A^{i} : \bigwedge_{j \in i} \left(A^i_{x} \times B^j_x \subseteq^0 E_x  \iff j \in s \right)\right\}.$$
	
	Then $A^{i,s} \in \mathcal{B}^{E}_{X \times Y}$ (as before, we use Lemma \ref{lem: alg of fibers closed under meas} freely throughout the proof), and $\left(A^{i,s} : s \subseteq i \right)$ is a partition of 
	$A^i$.	
	
	Also, for $j \in \omega$ and $s \subseteq j+1$, we let 
	$$B^{j,s} := \left\{ (x,z) \in B^{j} : \bigwedge_{i \in j+1} \left(A^i_{x} \times B^j_x \subseteq^0 E_x  \iff i \in s \right)\right\},$$
	then $B^{j,s} \in \mathcal{B}^{E}_{X \times Z}$, and $\left(B^{j,s}  : s \subseteq j \right)$ is a partition of 
	$B^j$.	
	
	Then, using Fubini, for any $i,j \in \omega, s \subseteq i, s' \subseteq j+1$ we have either $A^{i,s} \land B^{j,s'} = \left\{(x,y,z) \in X \times Y \times Z : (x,y) \in A^{i,s} \land (x,z) \in  B^{j,s'} \right\} \subseteq^0 E$, or $A^{i,s} \land B^{j,s'} \cap E =^0 \emptyset$ (guaranteed by the definition of $A^{i,s}$ when $i > j$, and by the definition of $B^{j,s}$ when $i \leq j$).
	Reorganizing the partition, we may thus assume: 
	\begin{enumerate}[(4)]
		\item for each $i,j \in \omega$, either
		\begin{gather*}
			A^i \land B^j = \left\{ (x,y,z) \in X \times Y \times Z : (x,y) \in A^i \land (y,z) \in B^j \right\} \subseteq^0 E,\\
			\textrm{or } \left( A^i \land B^j \right) \cap E =^0 \emptyset.
		\end{gather*}
	\end{enumerate}
Using that $E_y \in \mathcal{B}_{X \times Z}$ is $\mu$-stable for almost all $y \in Y$, we can similarly produce partitions $X \times Y = \bigsqcup_{i \in \omega} C^i$ with $C^i \in \mathcal{B}^{E}_{X \times Y}$ and $Y \times Z = \bigsqcup_{k \in \omega} D^k$ with $D^{k} \in \mathcal{B}^{E}_{Y \times Z}$ so that
\begin{enumerate}[(5)]
	\item for each $i,k \in \omega$, either
		\begin{gather*}
			C^i \land D^k = \left\{ (x,y,z) \in X \times Y \times Z : (x,y) \in C^i \land (y,z) \in D^k \right\} \subseteq^0 E,\\
			\textrm{or } \left( C^i \land D^k \right) \cap E =^0 \emptyset.
				\end{gather*}
\end{enumerate}

 Intersecting the partitions $(A^i)_{i \in \omega}$ and $(C^i)_{i \in \omega}$ and renumbering the parts, we may assume additionally
 \begin{enumerate}[(6)]
 	\item $A^i = C^i$ for all $i \in \omega$.
 \end{enumerate}
From now on we will only use that (4),(5) and (6) hold. Fix $i \in \omega$. Applying Corollary \ref{cor: part of binary into symm sets} to $A^i \in \mathcal{B}^{E}_{X \times Y}$, working in the graded probability space $\mathfrak{P}^{E}$ (see Remark \ref{rem: graded prob space of fibers}),  we find countable partitions $X = \bigsqcup_{t \in \omega} U^i_{t}$ with $U^i_t \in \mathcal{B}^{E}_{X}$ and $Y = \bigsqcup_{t \in \omega} V^i_t$ with $V^i_t \in \mathcal{B}^{E}_{Y}$ such that for each $t \in \omega$ we have:
	\begin{itemize}
		\item $\mu \left( \left( A^i \cap \left( U^i_t \times Y \right) \right) \triangle \left( A \cap \left( X \times V^i_t \right) \right)  \right) = 0$,
		\item for any $U' \subseteq U^i_t, U' \in \mathcal{B}^{E}_{X}$ such that both $\mu \left(A^i \cap \left(U' \times Y \right) \right) >0$ and $\mu \left(A \cap \left((U^t_i \setminus U') \times Y \right) \right) >0$, for any $V' \subseteq V_i, V' \in \mathcal{B}^{E}_{Y}$ we have $\mu \big( \left( A \cap \left( U' \times Y \right) \right) \triangle \left( A \cap \left( U_i \times V' \right) \right)  \big) > 0$.
	\end{itemize}
	
	For $t \in \omega$, let $A^i_t := A^i \cap \left( U^i_t \times V^i_t \right) \in \mathcal{B}^{E}_{X \times Y}$. From the first item, $A^i =^0 \bigsqcup_{t \in \omega} A^i_t$ is a partition of $A^i$. And for any fixed $t \in \omega$, from the second item we have: for any $U \in \mathcal{B}^{E}_{X}$ with $\mu \left( A^i_t \cap \left(U \times Y \right) \right) >0$ and $\mu \left(A^i_t \cap \left( \left(X \setminus U \right) \times Y \right) \right) > 0$, for any $V \in \mathcal{B}^{E}_{Y}$ we have 
	$$\mu \left( \left( A^i_t \cap \left(U \times Y \right) \right) \triangle \left( A^i_t \cap \left(X \times V \right) \right) \right) > 0.$$
	Partitioning each $A^i $ into $\left(A^i_{t} : t \in \omega\right)$ with  $A^i_{t} \in \mathcal{B}^{E}_{X \times Y}$ in this manner and renumbering the parts, we may assume that the partition $\left(A^i \right)_{i \in \omega}$ additionally satisfies the following.
\begin{enumerate}[(7)]
	\item For each $i \in \omega$, for any $U \in \mathcal{B}^{E}_{X}$ with $\mu \left( A^i \cap \left(U \times Y \right) \right) >0$ and $\mu \left(A^i \cap \left( \left(X \setminus U \right) \times Y \right) \right) > 0$, for any $V \in \mathcal{B}_{Y}^{E}$ we have 
	$$\mu \left( \left( A^i \cap \left(U \times Y \right) \right) \triangle \left( A^i \cap \left(X \times V \right) \right) \right) > 0.$$
\end{enumerate}

Now fix an arbitrary $i \in \omega$ and consider the set $A^i \in \mathcal{B}^{E}_{X \times Y}$, we claim that it is perfect for $E \subseteq \left(X \times Y \right) \times Z$. Let 
\begin{gather*}
	J := \left\{ j \in \omega : A^i \land B^j \subseteq^0 E \right\}, K := \left\{ k \in \omega : A^i \land D^k \subseteq^0 E \right\}.
\end{gather*}
Then, by (4), (5) and (6),  
$$A^i \land \left( \bigcup_{j \in J} B^j \right) =^0 A^i \land \left( \bigcup_{k \in K} D^k \right) =^0 A^i \land E.$$

Hence for almost every $z \in Z$ we have
$A^i \land \left( \bigcup_{j \in J}B^j \right)_z =^0 A^i \land \left( \bigcup_{k \in K} D^k\right)_z$. Therefore, by (7), the set $A^i \cap E_z =^0 A^i \land \left( \bigcup_{j \in J} B^j \right)_z$ must either have measure $0$, or be all of $A^i$ up to measure $0$.
\end{proof}

We prove a finitary counterpart and converse.

\begin{theorem}\label{thm: two dirs slicewise stab finitary version}
  Let $\mathcal{H}$ be a hereditarily closed family of $3$-partite $3$-hypergraphs. Then the following are equivalent.
  \begin{enumerate}
  \item There exists $d\in\mathbb{N}$ so that for every $H=(E; X,Y,Z)\in\mathcal{H}$, every $E_x$ and every $E_y$ is $d$-stable.
  \item Any ultraproduct  $\widetilde{H} = \left(\widetilde{E}; \widetilde{X}, \widetilde{Y}, \widetilde{Z} \right)$ of hypergraphs from $\mathcal{H}$  satisfies the conclusion of Theorem \ref{prop: part XY into perf sets}.
  \item (Approximately perfect stable regularity) For every $\varepsilon>0$ and $f:\mathbb{N}\rightarrow(0,1]$ there exists $N=N(\varepsilon,f)\in\mathbb{N}$ satisfying the following. For any $H=(E;X,Y,Z)\in\mathcal{H}$ there exists $N'\leq N$, a partition $X\times Y=\bigsqcup_{i\in[N']}A^i$ with $A^i\in\mathcal{F}^{E,N}_Z$, and sets $\Sigma_{X,Y}\subseteq[N']$ and $\Sigma_Z\subseteq Z$ so that $\mu_{X\times Y}(\bigsqcup_{i\not\in\Sigma_Z}A^i)<\varepsilon$, $\mu_Z(Z\setminus \Sigma_Z)<\varepsilon$, and for each $i\in\Sigma_{X,Y}$ and $z\in \Sigma_Z$,
      \[  \frac{\mu_{X\times Y}(A^i\cap E_z)}{\mu_{X\times Y}(A^i)}\]
      is in $[0,f(N'))\cup (1-f(N'), 1]$.
  \end{enumerate}
\end{theorem}
\begin{proof}
  \textbf{(1) implies (2)} By \L os's Theorem, (1) implies that $\widetilde{H}$ satisfies slice-wise stability in two directions, so the claim follows from Proposition \ref{prop: part XY into perf sets}.

	\ 

	\noindent \textbf{(2) implies (3)} Let $\varepsilon>0$ and $f:\mathbb{N}\rightarrow(0,1]$ be given, but suppose that (3) fails. Then for each $N$, there is an $H_N\in\mathcal{H}$ witnessing that the bound $N$ does not suffice, so we may take $\widetilde{H}$ to be a non-principal ultraproduct of this sequence. Applying (2) to $\widetilde{H}$ gives a countable partition $\widetilde X\times \widetilde Y=\bigsqcup_{i\in\omega}A^i$.

        Fix $\varepsilon'\in (0,\varepsilon)$. By countable additivity, we may choose an $N$ so that $\mu_{X\times Y}(\bigsqcup_i A^i)\geq 1-\varepsilon'$. The result follows by transfer.

        \

        \noindent \textbf{(3) implies (1)} We show the contrapositive. If (1) fails then for each $d\in\mathbb{N}$ we may find either $H_d=(E_d;X_d,Y_d,Z_d)$ with $X_d=Z_d:=[d]$, $Y_d:=[1]$, and $(x,1,z)\in E_d\Leftrightarrow x<z$, or the same with $X_d$ and $Y_d$ swapped. Without loss of generality, we assume the former. Then Theorem \ref{thm: metastable general equivalences} with $k=2$ implies the failure of (3).
\end{proof}

\section{All three directions of slice-wise stability}\label{sec: all three dirs of slice-wise stab}

We now consider the case where all slices in all directions are stable. The following example, which elaborates on Example \ref{ex: x=y<z}, has this property while failing to have partition-wise stability under any partition of the variables into two groups.

\begin{example}
  Let $X = Y = Z := [0,1]$. Let $E \subseteq X \times Y \times Z$ consist of the triples $(x,y,z)$ such that $\left \lvert \left\{ x,y,z \right\} \right \rvert = 2$, and the value $u$ which appears twice in $(x,y,z)$ is smaller than the value that appears once. This relation is slice-wise stable, but not stable as a binary relation under any partition of the coordinates into two groups. It satisfies stable regularity under any measures on $X,Y,Z$, because $E$ has to concentrate on atoms (that is, on singletons with positive measure).
\end{example}

As even the simpler Example \ref{ex: x=y<z} illustrates, we cannot hope to obtain as nice a partition of the individual sets $X$, $Y$, and $Z$ in this situation. Instead we obtain a collection of partitions of the binary sets $X\times Y$, $X\times Z$, and $Y\times Z$ so that the intersection of any two components is almost homogeneous.

\subsection{Slice-wise stable regularity lemma}\label{sec: slicewise stab reg proof}

\begin{definition}
  If $\mathcal{B}_{X}$ is a $\sigma$-algebra of subsets of $X$, $\mu_X$ a probability measure on $\mathcal{B}_{X}$ and $X' \in \mathcal{B}_{X}$, we denote by $\mathcal{B}_{X'}$ the $\sigma$-algebra $\left\{ X' \cap S : S \in \mathcal{B}_{X}  \right\}$ of subsets of $X'$, and by $\mu_{X'}$ the probability measure on $\mathcal{B}_{X'} \subseteq \mathcal{B}_{X}$ defined by $\mu_{X'}(S \cap X') := \frac{\mu_X(S \cap X')}{\mu_{X}(X')}$ for every $S \in \mathcal{B}_{X}$.
\end{definition}

\begin{theorem}\label{thm: all three parts perfect rects}
Suppose that $E \in \mathcal{B}_{X \times Y \times Z}$ is slice-wise $\mu$-stable. Then there exist countable partitions $X \times Y = \bigsqcup_{i \in \omega} A^i$, $X \times Z = \bigsqcup_{j \in \omega}B^j$ and $Y \times Z = \bigsqcup_{k \in \omega} D^k$ with  $A^i \in \mathcal{B}_{X \times Y}, B^j \in \mathcal{B}_{X \times Z}, D^k \in \mathcal{B}_{Y \times Z}$ so that:
\begin{enumerate}
	\item each $A^i = A^{i,X} \times A^{i,Y}, B^j = B^{j,X} \times B^{j,Z}, D^k = D^{k,Y} \times D^{k,Z}$ is a rectangle with $A^{i,X},B^{j,X} \in \mathcal{B}^{E}_{X},  A^{i,Y}, D^{k,Y} \in \mathcal{B}^{E}_{Y}, B^{j,Z}, D^{k,Z} \in \mathcal{B}^{E}_{Z}$, and each $A^i, B^j, D^k$ is perfect for $E$ viewed as a binary relation on the corresponding partition of $X,Y,Z$;
	\item \begin{enumerate}
 	\item for each $i,j \in \omega$,  either $A^i \land B^j \subseteq^0 E$ or $\left( A^i \land B^j \right) \cap E =^0 \emptyset$;
 	\item for each $i,k \in \omega$, either $A^i \land D^k \subseteq^0 E$ or $\left(A^i \land D^k \right) \cap E =^0 \emptyset$;
 	\item for each $j,k \in \omega$, either $B^j \land D^k \subseteq^0 E$ or $\left(B^j \land D^k \right) \cap E =^0 \emptyset$.

 \end{enumerate} 

\end{enumerate}
 \end{theorem}

 The main technical step towards it is the following proposition.

\begin{prop}\label{prop: part perf sets and rects}
Suppose that $E \in \mathcal{B}_{X \times Y \times Z}$, the slices $E_x \in \mathcal{B}_{Y \times Z}$ are $\mu$-stable for almost all $x \in X$, and the slices $E_y \in \mathcal{B}_{X \times Z}$ are $\mu$-stable for almost all $y \in Y$. Then for any countable partition $X \times Y = \bigsqcup_{i \in \omega} A^i$ with each $A^i \in \mathcal{B}_{X \times Y}$ perfect for the relation $E \subseteq (X \times Y) \times Z$ (note that such a partition with $A^i \in \mathcal{B}^{E}_{X \times Y}$ exists by Proposition \ref{prop: part XY into perf sets}), there exists a countable partition $Y \times Z = \bigsqcup_{j \in \omega} B^j$ into rectangles $B^j = B^{j,Y} \times B^{j,Z}$ for some $B^{j,Y} \in \mathcal{B}_{Y},  B^{j,Z} \in \mathcal{B}_{Z}$, so that for each $i,j \in \omega$, either $A^i \land B^j \subseteq^0 E$ or $\left( A^j \land B^j \right) \cap E =^0 \emptyset$.
\end{prop}
\begin{proof}
Fix a partition $X \times Y = \bigsqcup_{i \in \omega} A^i$ with each $A^i \in \mathcal{B}_{X \times Y}$ perfect for the relation $E \subseteq (X \times Y) \times Z$ (at least one such partition exists by Proposition \ref{prop: part XY into perf sets}).
First of all, we may assume that almost all fibers are uniformly $\mu$-stable: 
	\begin{claim}\label{cla: proof rects 0}
	We may assume that there exists $d^* \in \omega$ so that for almost all $y \in Y$, the	relation $E_y \in \mathcal{B}_{X \times Z}$ is $d^*$-stable.
	\end{claim}
	\begin{claimproof}

By assumption we have a partition (up to measure zero) $Y =^0 \bigsqcup_{d \in \omega} Y^d$, with $\mu\left( Y^d \right) >0$ for all $d \in \omega$ (ignoring some $d$'s if necessary), where 
$$Y^d := \left\{ y \in Y : E_y \textrm{ is } (d+1)\textrm{-}\mu\textrm{-stable, but not } d\textrm{-}\mu\textrm{-stable}\right\} \in \mathcal{B}_{Y}.$$ 

Note that for each $d \in \omega$, the relation  $E^d := E \cap \left( X \times Y_d \times Z \right) \in \mathcal{B}_{X \times Y^d \times Z}$  still satisfies the assumption of the proposition with $d$-$\mu$-stability (with respect to the measures $\mu_{X}, \mu_{Y^d}, \mu_{Z}$).
 Also, for each $d,i\in \omega$, let $A^{d,i} := A^i \cap (X \times Y_d) \in \mathcal{B}_{X \times Y^d}$. Then $\left( A^{d,i} : i \in \omega \right)$ is a countable partition of $X \times Y_d$, with each $A^{d,i}$ perfect for the binary relation $E^d \subseteq \left( X \times Y_d \right) \times Z$, with respect to the corresponding measures $\mu_{X \times Y^d}$ and $\mu_{Z}$ (indeed, for every $z\in Z$ with $\mu_{X \times Y^d} \left( E^d_{z} \cap A^{d,i} \right) > 0$ we have $\mu_{X \times Y} \left( E_{z} \cap A^{i} \right) > 0$, hence $\mu_{X \times Y} \left( A^i \setminus E_{z} \right) = 0$ by perfection of $A^i$, and so $\mu_{X \times Y^d} \left( A^{d,i} \setminus E^d_{z} \right) = 0$).

 Now assume that for each $d \in \omega$ and the $d$-$\mu$-stable relation $E^d$, there exists a partition $Y^d \times Z = \bigsqcup_{j \in \omega} B^{d,j}$ as required, i.e.~each $B^{d,j} = B^{d,j, Y} \times B^{d,j,Z}$ for some $B^{d,j, Y} \in \mathcal{B}_{Y^d}, B^{d,j, Z} \in \mathcal{B}_{Z}$, and so that for any $i,j \in \omega$, either $A^{d,i} \land B^{d,j} \subseteq^0 E^d$ or $\left( A^{d,i} \land B^{d,j} \right) \cap E^d =^0 \emptyset$ (with respect to the measure $\mu_{X \times Y^d \times Z} $).
 Then $\left( B^{d,j} : d,j \in \omega \right)$ is the required countable partition of $Y \times Z$ for $E$. Indeed, each $B^{d,j}$ is still a rectangle with respect to the algebras $\mathcal{B}_{Y}, \mathcal{B}_{Z}$, and for any $i \in \omega$ and $(d,j) \in \omega^2$ we have $A^i \land B^{d,j} = A^{d,i} \land B^{d,j}$, hence either $A^i \land B^{d,j} \subseteq^0 E $ or $\left( A^{i} \land B^{d,j} \right) \cap E =^0 \emptyset$.
	\end{claimproof}

	By induction on a countable ordinal $\alpha$, we will define pairwise-disjoint rectangles $B^{\alpha} = B^{\alpha,Y} \times B^{\alpha,Z}$ with $B^{\alpha,Y} \in \mathcal{B}_{Y}, B^{\alpha,Z} \in \mathcal{B}_{Z}$ and $\mu \left(B^{\alpha} \right) > 0$ satisfying the requirement of the proposition.
Given a countable ordinal $\alpha$, assume we have already chosen $B^{\beta}$ for $\beta < \alpha$. Let
$\bar{B}^{\alpha} :=  \bigcup_{\beta < \alpha}B^{\alpha}  \in \mathcal{B}_{Y \times Z}$. If $\mu \left(\bar{B}^{\alpha} \right) = 1$, then $\left(B^{\beta} : \beta < \alpha \right)$ is the desired partition. Otherwise we have $\mu \left(\left( Y \times Z \right) \setminus \bar{B}^{\alpha}  \right) > 0$. And since the complement of a rectangle is a union of three rectangles, by countable additivity there exists a rectangle $D = Y' \times Z'$ with $Y' \in \mathcal{B}_{Y}, Z' \in \mathcal{B}_{Z}$ so that $\mu \left(D \right) >0$ and $D$ is disjoint from $B^{\beta}$ for all $\beta < \alpha$. We will find a rectangle of positive measure contained in $D$ satisfying the requirement.

	We define the \emph{$Y'$-support of $A^i$} as $\supp_{Y'}(A^i) := \left\{ y \in Y' : \mu \left(A^i_y \right) > 0 \right\} \in \mathcal{B}_{Y}$. 	
	\begin{claim}\label{cla: supps form a tree}
	We may assume that for any $i < j \in \omega$, either $\supp_{Y'}(A^i) \supseteq \supp_{Y'}(A^j)$ or $\supp_{Y'}(A^i) \cap \supp_{Y'}(A^j) = \emptyset$.
	\end{claim}
	\begin{claimproof}
		By induction on $i \in \omega$ we choose $n_i \in \omega$ and a partition  $A^i = \bigsqcup_{j \in n_i} \bar{A}^{N_i + j}$, where $N_i = \sum_{j \in i} n_j$ and each $\bar{A}^i \in \mathcal{B}_{X \times Y}$ is perfect, so that the sets $\left(\bar{A}^j : j \in \left( N_i + n_i \right)\right)$ satisfy the requirement.
		
		Fix $i \in \omega$, and assume we have already chosen $n_j \in \omega$ for all $j \in i$ and $\bar{A}^k$ for all $k \in N_i = \sum_{j \in i} n_j$. For each $j \in N_i$, let 
		$$\bar{A}^{i,j,1} := A^i \land \supp_{Y'}\left(\bar{A}^j \right), \bar{A}^{i,j,0} := A^i \land \left( Y' \setminus \supp_{Y'}\left(\bar{A}^j \right) \right).$$
		Then $\bar{A}^{i,j,1}, \bar{A}^{i,j,0} \in \mathcal{B}_{X \times Y}$ and $A^i = \bar{A}^{i,j,1} \sqcup \bar{A}^{i,j,0}$. For $\sigma \in 2^{N_i}$, let 
		$$\bar{A}^{i, \sigma} := \bigcap_{j \in N_i} \bar{A}^{i,j,\sigma(j)} \in \mathcal{B}_{X \times Y}.$$
		Then $\left(\bar{A}^{i,\sigma} : \sigma \in 2^{N_i} \right)$ is a partition of $A^i$ and each $\bar{A}^{i, \sigma}$ is perfect (as a subset of a perfect set). And for every $j \in N_i$, if $\sigma(j) = 1$ then 
		$$\supp_{Y'} \left( \bar{A}^j \right) \supseteq \supp_{Y'} \left(\bar{A}^{i,j,1}\right) \supseteq \supp_{Y'} \left( \bar{A}^{i, \sigma} \right);$$ 
		and if $\sigma(j) = 0$, then  
		$$\supp_{Y'} \left( \bar{A}^j \right) \cap \supp_{Y'} \left(\bar{A}^{i,\sigma}\right) \subseteq  \supp_{Y'} \left( \bar{A}^j \right) \cap \supp_{Y'} \left(\bar{A}^{i,j,0}\right) = \emptyset.$$
	Let $n_i := 2^{N_i}$	 and let $\left(\bar{A}^{N_i + j} : j \in n_i \right)$ list $\left( \bar{A}^{i, \sigma} : \sigma \in 2^{N_i}\right)$ in an arbitrary order. Then the sets $\left(\bar{A}^j : j \in \left( N_i + n_i \right)\right)$ satisfy the requirement.
	\end{claimproof}
It follows from Claim \ref{cla: supps form a tree} that the $Y'$-supports of the $A^i$'s form a forest of subsets of $Y'$ under inclusion. That is, there exists an injection $f: \omega \to \omega^{<\omega}$ so that for any $i,j \in \omega$, 
\begin{gather*}
	f(i) \trianglelefteq f(j) \iff \supp_{Y'} \left(A^i \right) \supseteq \supp_{Y'}\left(A^j \right), \textrm{ and}\\
	f(i) \perp f(j) \iff \supp_{Y'} \left(A^i \right) \cap \supp_{Y'}\left(A^j \right) = \emptyset.
\end{gather*}

We re-enumerate the partition $\left(A^i : i \in \omega \right)$ according to $f$, that is let $T := \ima \left( f \right) \subseteq \omega^{< \omega}$, and for $\lambda \in T$ let $A^{\lambda} := A^{f^{-1}(\lambda)}$ and $Y_{\lambda} := \supp_{Y'}\left( A^{\lambda}\right)$. Then for any $\lambda,\nu \in T$, $Y_\nu \supseteq Y_{\lambda} \iff \nu \trianglelefteq \lambda$; and $Y_{\nu} \cap Y_{\lambda} = \emptyset \iff \nu \perp \lambda$.

	\begin{claim}\label{cla: proof rects 1.5}
			Let $T' := \left\{ \lambda \in T : \mu \left( Y_{\lambda}\right) > 0 \right\}$. 
			We may assume $T' \neq \emptyset$.
	\end{claim}
\begin{proof}
Note that for any $\lambda \in T$, if $\mu \left(Y_{\lambda} \right) = \supp_{Y'} \left(A^{\lambda}\right) = 0$, then $A^{\lambda} \land Y' =^0 \emptyset$. Hence if $\mu \left(Y_{\lambda} \right) = 0$ for all $\lambda \in T$, the rectangle $B^{\alpha} := Y' \times Z'$ satisfies the requirement.
\end{proof}

For $\lambda \in T'$ and $z \in Z'$, we define
\begin{gather*}
	S (\lambda) := \left\{ \nu \in T' : \nu \trianglelefteq \lambda \right\} = \left\{ \nu \in T' : \supp_{Y'}\left(A^{\nu} \right) \supseteq  \supp_{Y'}\left(A^{\lambda} \right)\right\},\\
	S (\lambda, z) :=  \left\{ \nu \in S(\lambda) : \mu \left( A^{\nu} \setminus E_z \right) = 0\right\}.
\end{gather*}

\begin{claim} \label{cla: proof rects 1} For almost all $z \in Z$, for all $\lambda \in T'$ we have: $\nu \in T' \setminus S(\lambda,z) \implies \mu \left( A^\nu \cap E_z \right) = 0$.
\end{claim}
\begin{claimproof}
	For a given $\nu \in T'$, as $A^{\nu}$ is perfect, the set 
	$$Z_{\nu} := \left\{z \in Z : 0 < \mu \left(A^{\nu} \cap E_z \right) < \mu \left( A^{\nu} \right) \right\}$$
	has measure $0$. Then the countable union $\bar{Z} := \bigcup_{\nu \in T'} Z_{\nu}$ also has measure $0$. It follows that if $z \in Z \setminus \bar{Z}$, then for any $\nu \in T'$ we have $A^{\nu} \subseteq^0 E_z$ or $A^{\nu} \cap E_z =^0 \emptyset$, and the claim follows.
	\end{claimproof}

%
%
%
%
%
%
%

 We define a relation $F \subseteq T' \times Z'$ via $(\lambda,z) \in F \iff A^{\lambda} \subseteq^0 E_z$.
\begin{claim}\label{cla: proof rects 2}
Given any $\lambda \in T'$, for almost all $(x,y) \in A^{\lambda} \land Y_{\lambda}$ we have $E_{x,y} \land Z' =^0 F_{\lambda}$.
\end{claim}
\begin{claimproof}
%
Fix $z \in F_{\lambda}$, then $A^{\lambda} \subseteq^0 E_z$ (by the definition of $F$), so by Fubini we have $\left(A^{\lambda} \land Y_{\lambda} \right) \times F_{\lambda} \subseteq^0 E$.

On the other hand, for almost  every $z \in Z' \setminus F_{\lambda}$ we have $A^{\lambda} \cap E_z =^0 \emptyset$ (by Claim \ref{cla: proof rects 1}). So by Fubini $\left(A^{\lambda} \land Y_{\lambda} \right) \times \left(Z \setminus F_{\lambda} \right) \cap E =^0 \emptyset$.

By Fubini again we thus have that for almost every $(x,y) \in A^{\lambda} \land Y_{\lambda}$, both $F_{\lambda} \subseteq^0  E_{x,y} \land Z'$ and $\left( E_{x,y} \land Z' \right) \setminus F_{\lambda} =^0 \emptyset$, that is $E_{x,y} \land Z' =^0 F_{\lambda}$.
\end{claimproof}

The proof of the following claim utilizes that the relation $F$ is $\mu$-stable when restricted to a branch of $T'$.
\begin{claim}\label{cla: proof rects 3}
There exists some $\lambda^* \in T'$ and some $S^* \subseteq S(\lambda^*)$ (in particular $S^*$ is finite) so that, taking $Z^* := \left\{ z \in Z' : S(\lambda,z) = S^* \right\} \in \mathcal{B}_{Z}$, we have the following: 
\begin{enumerate}
	\item $\mu\left(Z^*\right) > 0$;
	\item for every $\lambda \in T'$ with $\lambda \not \perp \lambda^*$ we have either $A^{\lambda} \land \left(Y^{\lambda} \times Z^*\right) \subseteq^0 E$ or $\left(A^{\lambda} \land \left(Y^{\lambda} \times Z^* \right)  \right) \cap E =^0 \emptyset$.
\end{enumerate}
\end{claim}
\begin{claimproof}
	Fix $d \in \omega, d \geq 1$, and assume that $R$ is a (possibly infinite) branch in $T'$. Assume that $\left( \lambda_{\sigma} : \sigma \in 2^{<d} \right)$  with $\lambda_{\sigma} \in R$ for all $\sigma \in 2^{<d}$ is a $\mu$-tree of height $d$ for $F^* \subseteq Z' \times T'$ (where $F^*$ is $F$ with the roles of the variables exchanged), that is: $F_{\lambda_{\emptyset}}$ $\mu$-splits $Z'$, and for every $\sigma \in 2^{<d}$ with $|\sigma| \geq 1$, 
		$F_{\lambda_{\sigma}} \ \mu \textrm{-splits }  F^{\sigma}_{\left(\lambda_{\sigma|_0}, \ldots, \lambda_{\sigma|_{|\sigma|-1}} \right)}$, where as usual  
		$$F^{\sigma} = \left \{ (z, r_0, \ldots, r_{|\sigma|-1}) \in Z' \times R^{n} : \bigwedge_{i = 0}^{|\sigma|-1} \left( (z,r_i) \in \left(F^* \right)^{\sigma(i)} \right) \right\}.$$

		By Claim \ref{cla: proof rects 2}, for each $\sigma \in 2^{<d}$ we have: for almost all $(x,y) \in A^{\lambda_{\sigma}} \land Y_{\lambda_{\sigma}}$, $E_{x,y} \land Z' =^0 F_{\lambda_{\sigma}}$. 
		As $\lambda_\sigma \in R$ for all $\sigma \in 2^{<d}$, there exists some $\sigma^* \in 2^{<d}$ so that $\lambda_{\sigma} \trianglelefteq \lambda_{\sigma^*}$ for all $\sigma \in 2^{<d}$; let $\lambda := \lambda_{\sigma^*}$. Then $\mu \left(Y_{\lambda} \right) > 0$ as $\lambda \in T'$, and  $Y_{\lambda} \subseteq Y_{\lambda_{\sigma}} = \supp_{Y'}\left(A^{\lambda_{\sigma}} \right)$ for all $\sigma \in 2^{<d}$. Hence by Fubini we have: for almost all $y \in Y_{\lambda}$, 
		\begin{enumerate}
			\item $\mu \left( \prod_{\sigma \in 2^{<d}} A^{\lambda_{\sigma}}_y \right) = \prod_{\sigma \in 2^{<d}} \mu \left(A^{\lambda_{\sigma}}_y \right)> 0$;
			\item for almost every tuple $\bar{x} = \left(x_{\sigma} : \sigma \in 2^{<d} \right) \in \prod_{\sigma \in 2^{<d}} A^{\lambda_{\sigma}}_y$, $E_{x_{\sigma},y} \land Z' =^0 F_{\lambda_{\sigma}}$ for every $\sigma \in 2^{<d}$. And since $\left( \lambda_{\sigma} : \sigma \in 2^{<d} \right)$ is a $\mu$-tree for $F^*$, it follows that $\bar{x}$ is a $\mu$-tree for the relation $E^*_y \subseteq Z \times X$ (that is, $E_y \subseteq X \times Z$ with the roles of the variables reversed).
		\end{enumerate}		
As $\mu \left( Y_{\lambda} \right) >0$, it follows that for a positive measure set of $y \in Y$, the relation $E^{*}_y$ is not $d$-$\mu$-stable.
		Hence by Claim \ref{cla: proof rects 1} and Corollary \ref{cor: stab pres under opp}, there is some $D^* \in \omega$ depending only on $d^*$ so that  no branch $R$ of $T'$ contains a $\mu$-tree of height $D^*$ for $F^*$.
	
	Now by Claim \ref{cla: proof rects 1.5}, $T' \neq \emptyset$.
Let  $\left( \lambda_{\sigma} : \sigma \in 2^{<d} \right)$ with  $\lambda_{\sigma} \in T'$ be a $\mu$-tree of maximal height for $F^*$ with all $\lambda_\sigma \in T'$ for $\sigma \in 2^{<d}$ pairwise $\trianglelefteq$-comparable. We try to complete it to a $\mu$-tree of height $d+1$ adding one node at a time: by induction on $\sigma \in 2^{d}$ with the lexicographic ordering $<_{\lex}$ on $2^d$, we try to chose $\lambda_{\sigma}$ so that $\lambda_{\sigma}$ is comparable to every element of $R_{\sigma} := \left\{ \lambda_{\tau} : \tau \in 2^{d} \land \tau <_{\lex} \sigma \right\} \cup \left\{\lambda_{\tau} : \tau \in 2^{<d} \right\}$ and $F_{\lambda_{\sigma}}$ $\mu$-splits $F^{\sigma}_{\left(\lambda_{\sigma|_0}, \ldots, \lambda_{\sigma|_{d-1}} \right)}$ (this set has positive measure by assumption). By maximality of $d$, we must get stuck for some $\sigma^* \in 2^{d}$ (non-maximal with respect to $<_{\lex}$): that is, for every $\lambda \in T'$  $\trianglelefteq$-comparable to every element in $R_{\sigma^*}$, we have either $F^{\sigma^*}_{\left(\lambda_{\sigma^*|_0}, \ldots, \lambda_{\sigma^*|_{d-1}} \right)} \subseteq^0 F_{\lambda}$ or $F^{\sigma^*}_{\left(\lambda_{\sigma^*|_0}, \ldots, \lambda_{\sigma^*|_{d-1}} \right)} \cap F_{\lambda} =^0 \emptyset$. 
Let $\lambda^*$ be the $\trianglelefteq$-maximal element of $R_{\sigma^*}$, then
$R_{\sigma^*} \subseteq S \left(\lambda^* \right)$, in particular $\lambda_{\sigma^*|_{i}} \in S \left(\lambda^* \right)$ for each $0 \leq i \leq d-1$.
Then for at least one set $S^* \subseteq S(\lambda^*)$ so that $\lambda_{\sigma^*|_i} \in S^* \iff \sigma^*(i) = 1$ for all $0 \leq i \leq d-1$, the set 
$$Z^* := \left\{z \in Z' :  \bigwedge_{\lambda \in S(\lambda^*)} \left( (\lambda,z) \in F \iff \lambda \in S\right)\right\} \subseteq F^{\sigma^*}_{\left(\lambda_{\sigma^*|_0}, \ldots, \lambda_{\sigma^*|_{d-1}} \right)}$$
still has positive measure. By Claim \ref{cla: proof rects 1}, $Z^*$ is of the required form. Note also that every element $\lambda \in T'$ with $\lambda \not \perp \lambda^*$ is $\trianglelefteq$-comparable to every element of $R_{\sigma^*}$. Hence for every $\lambda \in T'$ with $\lambda \not \perp \lambda^*$, we have either $Z^* \subseteq^0 F_{\lambda}$ or $Z^* \cap F_{\lambda} =^0 \emptyset$. By Claim \ref{cla: proof rects 2} and Fubini, in the first case we have $A^{\lambda} \land \left(Y^{\lambda} \times Z^*\right) \subseteq^0 E$, and in the second $\left(A^{\lambda} \land \left(Y^{\lambda} \times Z^* \right)  \right) \cap E =^0 \emptyset$.
\end{claimproof}

Let $\lambda^* \in T'$ and $Z^*$ be as given by Claim \ref{cla: proof rects 3}. We take  $B^{\alpha,Y} := Y_{\lambda^*}, B^{\alpha,Z} := Z^*, B^{\alpha} := Y_{\lambda^*} \times Z^*$, then $\mu \left(B^{\alpha}\right) > 0$ and the rectangle $B^{\alpha}$ is disjoint from $B^{\beta}$ for all $\beta < \alpha$ as $B^{\alpha} \subseteq Y' \times Z'$. Let $\lambda \in T$ be arbitrary. If $\lambda \notin T'$, then $\supp_{Y'}\left(A^{\lambda} \right) =^0 \emptyset$; if $\lambda \in T'$ and $\lambda \perp \lambda^*$, then by construction of the forest $\supp_{Y'}\left(A^{\lambda} \right) \cap Y_{\lambda^*} =^0 \emptyset$. In either case by Fubini we have $A^{\lambda} \land Y_{\lambda^*} =^0 \emptyset$, so in particular $A^{\lambda} \land B^{\alpha} =^0 \emptyset$. And if $\lambda \in T'$ and $\lambda \not \perp \lambda^*$, by Claim \ref{cla: proof rects 3} either $A^{\lambda} \land B^{\alpha} \subseteq^0 E$ or $A^{\lambda} \land B^{\alpha} \cap E =^0 \emptyset$, so $B^{\alpha}$ satisfies the requirements.

Finally, note that the sequence $\mu \left( \bar{B}^{\alpha} \right) \in [0,1]$ is strictly increasing with $\alpha$, so the construction must stop at some countable ordinal $\alpha$.
\end{proof}

\begin{prop}\label{prop: part perf rects}
Suppose that $E \in \mathcal{B}_{X \times Y \times Z}$ is slice-wise $\mu$-stable. Assume that we are given partitions $X \times Z = \bigsqcup_{j \in \omega}B^j$ with each $B^j \in \mathcal{B}_{X \times Z}$ perfect for $E$ viewed as a binary relation on $(X \times Z) \times Y$, and  $Y \times Z = \bigsqcup_{j \in \omega} D^k$ with each $D^k \in \mathcal{B}_{Y \times Z}$ perfect for $E$ viewed as a binary relation on $\left(Y \times Z \right) \times X$ (note that such partitions always exist by Proposition \ref{prop: part XY into perf sets} applied for the appropriate permutations of the coordinates).

 Then there exists a countable partition $X \times Y = \bigsqcup_{i \in \omega} A^i$ so that:
 \begin{enumerate}
 	\item each $A^i$ is perfect for the relation $E \subseteq (X \times Y) \times Z$;
 	\item $A^i = A^{i,X} \times A^{i,Y}$ is a rectangle with $A^{i,X} \in \mathcal{B}_{X}, A^{i,Y} \in \mathcal{B}_{Y}$;
 	\item for each $i,j \in \omega$,  either $A^i \land B^j \subseteq^0 E$ or $\left( A^j \land B^j \right) \cap E =^0 \emptyset$;
 	\item for each $i,k \in \omega$, either $A^i \land D^k \subseteq^0 E$ or $\left(A^i \land D^k \right) \cap E =^0 \emptyset$.
 \end{enumerate} 
\end{prop}
\begin{proof}
	Since the slices $E_x$ are $\mu$-stable for almost all $x \in X$ and the slices $E_z$ are $\mu$-stable for almost all $z \in Z$, applying Proposition \ref{prop: part perf sets and rects} with the appropriate permutation of the variables, we find a partition $X \times Y = \bigsqcup_{i \in \omega} A^i$ so that: each $A^i = A^{i,X} \times A^{i,Y}$ with $A^{i,X} \in \mathcal{B}_{X}, A^{i,Y} \in \mathcal{B}_{Y}$ is a rectangle,  and for every $i,j \in \omega$ we have either $A^i \land B^j \subseteq^0 E$ or $\left( A^i \land B^j \right) \cap E =^0 \emptyset$.
	
	Since the slices $E_y$ are $\mu$-stable for almost all $y \in Y$ and the slices $E_z$ are $\mu$-stable for almost all $z \in Z$, applying Proposition \ref{prop: part perf sets and rects} again with the appropriate permutation of the variables, we find a partition $X \times Y = \bigsqcup_{i \in \omega} C^i$ so that: each $C^i = C^{i, X} \times C^{i,Y}$ with $C^{i,X}\in \mathcal{B}_{X}, C^{i,Y} \in \mathcal{B}_{Y}$ is a rectangle, and for every $i,k \in \omega$ we have either $C^i \land D^k \subseteq^0 E$ or $\left(C^i \land D^k \right) \cap E =^0 \emptyset$.
	
	Intersecting the partitions $(A^i)_{i \in \omega}$ and $\left(C^i \right)_{i \in \omega}$ of $X \times Y$, i.e.~considering $\left(A^i \cap C^j : (i,j) \in \omega^2 \right)$, we obtain a countable partition of $X \times Y$ which still consists of rectangles and satisfies the condition on intersections with $E$ simultaneously for the $B^j$'s and for the $D^k$'s; hence we may assume that $A^i = C^i$.
	
	So the partitions $A^i, C^i, B^j, D^j$ satisfy the conditions (4), (5) and (6) in the proof of Proposition \ref{prop: part XY into perf sets}. Hence the proof of Proposition \ref{prop: part XY into perf sets} goes through unchanged to produce a partition of $X\times Y$ into perfect sets for $E \subseteq (X\times Y) \times Z$. This partition is obtained from $\left( A^i : i \in \omega \right)$ by subdividing each $A^i$ into countably many sets $A^i =^0 \bigsqcup_{t \in \omega} A^i_t$, where for each $t \in \omega$ we have  $A^i_t = A^i \cap \left( U^i_t \times V^i_t \right)$ for some $U^i_t \in \mathcal{B}_{X}, V^i_t \in \mathcal{B}_{Y}$. Since each $A^i$ is a rectangle, $A^i_t$ is also a rectangle for each $i,t \in \omega$. It follows that every set in the resulting partition of $X \times Y$ into perfect sets is also a rectangle.
	
	For (3): given $(i,t) \in \omega^2$ and $j \in \omega$, if $\mu \left( \left( A^i_t \land B^j \right) \cap E \right) > 0$, then $\mu(\left( A^i \land B^j \right) \cap E) > 0$, so $A^i \land B^j \subseteq^0 E$, hence in particular $A^i_t \land B^j \subseteq^0 E$.
	
	And (4) is similar.
\end{proof}



\begin{proof}[Proof of Theorem \ref{thm: all three parts perfect rects}]
	By Proposition \ref{prop: part XY into perf sets} applied for the appropriate permutations of the coordinates, we find partitions $X \times Z = \bigsqcup_{j \in \omega}B^j$ with each $B^j \in \mathcal{B}^{E}_{X \times Z}$ perfect for $E$ viewed as a binary relation on $(X \times Z) \times Y$, and $Y \times Z = \bigsqcup_{j \in \omega} D^k$ with each $D^k \in \mathcal{B}^{E}_{Y \times Z}$ perfect for $E$ viewed as a binary relation on $\left(Y \times Z \right) \times X$ (note that these are not necessarily rectangles).
	
	Applying Proposition \ref{prop: part perf rects}, working in the graded probability space $\mathfrak{P}^{E}$ (see Remark \ref{rem: graded prob space of fibers}), we then find a partition $X \times Y = \bigsqcup_{i \in \omega} A^i$ so that each $A^i \in \mathcal{B}^{E}_{X} \times \mathcal{B}^{E}_{Y}$ is a rectangle, perfect for the relation $E \subseteq (X \times Y) \times Z$.
 
 Applying Proposition \ref{prop: part perf rects} again in  the graded probability space $\mathfrak{P}^{E}$ for the appropriate permutation of the coordinates and starting with the partitions $X \times Y = \bigsqcup_{i \in \omega} A^i$ and  $Y \times Z = \bigsqcup_{j \in \omega} D^k$ (both consisting of perfect sets), we find a new partition $X \times Z = \bigsqcup_{j \in \omega}\widetilde{B}^j$ with each $\widetilde{B}^j \in \mathcal{B}^{E}_{X} \times \mathcal{B}^{E}_{Z}$ a rectangle, perfect for $E \subseteq \left( X \times Z \right) \times Y$ and so that
 	 \begin{enumerate}
 \item[(a)] for each $i,j \in \omega$,  either $A^i \land \widetilde{B}^j \subseteq^0 E$ or $\left( A^j \land \widetilde{B}^j \right) \cap E =^0 \emptyset$;
 	\item[(b)] for each $j,k \in \omega$, either $\widetilde{B}^j \land D^k \subseteq^0 E$ or $\left(\widetilde{B}^j \land D^k \right) \cap E =^0 \emptyset$.
 \end{enumerate}
 
 Finally, applying Proposition \ref{prop: part perf rects} again in  $\mathfrak{P}^{E}$ for the appropriate permutation of the coordinates and starting with the partitions $X \times Y = \bigsqcup_{i \in \omega} A^i$ and  $X \times Z = \bigsqcup_{j \in \omega}\widetilde{B}^j$ (both consisting of perfect sets), we find a partition $Y \times Z = \bigsqcup_{j \in \omega} \widetilde{D}^k$ with each $\widetilde{D}^k \in \mathcal{B}^{E}_{Y} \times \mathcal{B}^{E}_{Z}$ perfect for $E \subseteq \left(Y \times Z \right) \times X$ and so that:
\begin{enumerate}
 \item[(c)] for each $i,k \in \omega$,  either $A^i \land \widetilde{D}^k \subseteq^0 E$ or $\left( A^j \land \widetilde{D}^k \right) \cap E =^0 \emptyset$;
 	\item[(d)] for each $j,k \in \omega$, either $\widetilde{B}^j \land \widetilde{D}^k \subseteq^0 E$ or $\left(\widetilde{B}^j \land \widetilde{D}^k \right) \cap E =^0 \emptyset$.
 \end{enumerate}
 
Hence $A^i, \widetilde{B}^j, \widetilde{D}^k$ are partitions into perfect rectangles, and (2a), (2b), (2c) hold by (a), (c) and (d), respectively.
 \end{proof}
 
 \begin{remark}
Assume that $E \in \mathcal{B}_{X \times Y \times Z}$  is partition-wise $\mu$-stable, hence satisfies perfect stable regularity (Definition \ref{def: perf stab reg}, Proposition \ref{prop: stable hypergraph reg}): $X = \bigsqcup_{i \in \omega} X_i, Y = \bigsqcup_{i \in \omega} Y_i, Z = \bigsqcup_{i \in \omega} Z_i$ with $X_i \in \mathcal{B}_{X}, Y_i \in \mathcal{B}_{Y}, Z_i \in \mathcal{B}_{Z}$ and $\frac{\mu \left( E \cap X_i \times Y_j \times Z_k \right)}{\mu \left( X_i \times Y_j \times Z_k \right)} \in \{ 0,1\}$ for all $i,j,k \in \omega$. Then, considering the rectangles $A^{(i,j)} := X_i \times Y_j, B^{(i,k)} := X_i \times Z_k$ and $D^{(j,k)} := Y_j \times Z_k$, we have that $X \times Y = \bigsqcup_{(i,j) \in \omega^2} A^{(i,j)}$, $X \times Z = \bigsqcup_{(i,k) \in \omega^2} B^{(i,k)}$ and $Y \times Z = \bigsqcup_{(j,k) \in \omega^2} D^{(j,k)}$ are countable partitions into perfect rectangles satisfying the conclusion of Theorem \ref{thm: all three parts perfect rects}. 
  
 On the other hand, we stress that in Theorem \ref{thm: all three parts perfect rects} even though e.g.~$X \times Y = \bigsqcup_{i \in \omega} A^i$ and each $ A^i = A^{i,X} \times A^{i,Y}$ is a rectangle with $A^{i,X} \in \mathcal{B}_X, A^{i,Y} \in \mathcal{B}_{Y}$, this does not imply that $\left( A^{i,X} : i \in \omega \right)$ may be assumed to give a partition of $X$, and in fact these sets can overlap in a complicated way for a general slice-wise stable hypergraph, as the example in Section \ref{sec: counterexample} shows. At least we can always get the following  approximation for our partition of pairs $X \times Y = \bigsqcup_{i \in \omega} A^i$ and each $ A^i = A^{i,X} \times A^{i,Y}$ (which  obviously holds if our partition of pairs comes from partitions of singletons as in the previous paragraph):
 \begin{claim}\label{cla: pf of slice reg 1}
		For any $F \in \mathcal{B}_{Y}$ with $\mu \left( F \right) < 1$ and $\delta > 0$, there exists a finite set $I \subseteq \omega$  so that $\mu \left(\bigcup_{i \in I} A^{i,X} \right) \geq 1 - \delta \cdot \max_{i \in I} \left\{\mu \left(A^{i,X} \right) \right\}$
		and $\mu \left( \left( \bigcap_{i \in I} A^{i,Y} \right) \setminus F\right) > 0$.
	\end{claim}
	\begin{claimproof}
	
	We define a tree of subsets of $Y$ as follows. For $\sigma \in 2^{< \omega}$, let $Y^{\sigma} := \bigcap_{i \in |\sigma|} \left(A^{i,Y} \right)^{\sigma(i)} \in \mathcal{B}_{Y}$ (where as usual $\left(A^{i,Y} \right)^1 = A^{i,Y}$ and $\left(A^{i,Y} \right)^0 = Y \setminus A^{i,Y}$). Consider a subtree of $2^{<\omega}$ defined by $T := \left\{ \sigma \in 2^{< \omega} : \mu \left( Y^{\sigma}  \setminus F \right) > 0 \right\}$, and let $Y' := Y \setminus \left( \bigcup_{\sigma \in 2^{< \omega} \setminus T} Y^{\sigma} \right)$. Then $Y' \in \mathcal{B}_{Y}$ and $\mu \left(Y' \setminus F \right) = \mu \left(Y \setminus F \right) > 0$, by countable additivity. Let $\Br(T)$ be the set of all branches of the tree $T$ (where, as usual, a branch is maximal subset of $T$ linearly ordered by $\triangleleft$).  For $\beta \in \Br(T)$, let $Y^{\beta} := \bigcap_{\sigma \in \beta} Y^{\sigma} \in \mathcal{B}_{Y}$ (note that while $\mu \left(Y^{\sigma} \setminus F \right) > 0$ for every $\sigma \in T$, we might have $\mu \left( Y^{\beta}  \setminus F \right) = 0$).
	Then $\left(Y^{\beta} \setminus F : \beta \in \Br(T) \right)$ is a partition of $Y' \setminus F$ (into possibly uncountably many sets). Take an arbitrary $y \in  Y' \setminus F $ (exists as $\mu \left(Y' \setminus F \right) > 0$), and let $\beta \in \Br(T)$ be such that $y \in Y^{\beta}$. Let $\supp(\beta) := \left\{ i \in \omega : \sigma(i) = 1 \textrm{ for some } \sigma \in \beta \right\}$.
	Assume that $\mu \left( \bigcup_{i \in \supp(\beta)} A^{i,X}\right) < 1$. As in particular $\left(X \setminus \bigcup_{i \in \supp(\beta)} A^{i,X} \right) \times \{y\} \subseteq \bigcup_{i \in \omega} A^{i,X} \times A^{i,Y}$, there must exist some $i^* \in \omega$ such that $y \in A^{i^*,Y}$ and $\mu \left( A^{i^*,X} \setminus \bigcup_{i \in \supp(\beta)} A^{i,X} \right) > 0$. But by the construction of the tree and maximality of the branch, if $y \in \left( Y' \setminus F \right) \cap A^{i^*,Y}$ then we must already have $i^* \in \supp(\beta)$, a contradiction.	
	
	Hence, taking $I' := \supp(\beta)$ we have: $\mu \left(\bigcup_{i \in I'} A^{i,X}\right) = 1$, and for every finite $I \subseteq I'$, $\mu \left( \left( \bigcap_{i \in I} A^{i,Y} \right) \setminus F \right) > 0$. Let $\delta >0$ be given, and take an arbitrary $i' \in I'$, then $\mu \left(A^{i',X} \right) > 0$, and by countable additivity there exists a finite subset $I$ of $I'$ containing $i'$ so that $\mu \left( \bigcup_{i \in I} A^{i,X} \right) \geq 1 - \delta \mu \left(A^{i',X}\right) \geq 1 - \delta \max_{i \in I} \left\{\mu \left(A^{i,X} \right) \right\}$.
		\end{claimproof}
 \end{remark}
 
We prove a finitary counterpart and a converse to Theorem \ref{thm: all three parts perfect rects}, demonstrating that it characterizes slice-wise stable hypergraphs:

\begin{theorem}\label{thm: finitary slice-wise reg}
	Let $\mathcal{H}$ be a hereditarily closed family of $3$-partite $3$-hypergraphs. Then the following are equivalent.
	\begin{enumerate}
		\item There exists $d \in \mathbb{N}$ so that every $H \in \mathcal{H}$ is slice-wise $d$-stable.
		\item Any ultraproduct  $\widetilde{H} = \left(\widetilde{E}; \widetilde{X}, \widetilde{Y}, \widetilde{Z} \right)$ of hypergraphs from $\mathcal{H}$  satisfies both (1) and (2) in the conclusion of Theorem \ref{thm: all three parts perfect rects}.
\item (Strong slice-wise stable regularity) For every $\varepsilon > 0$ and $f: \mathbb{N} \to (0,1]$, there exists $N = N(\varepsilon,f) \in \mathbb{N}$ satisfying the following. For any $H = (E; X,Y,Z) \in \mathcal{H}$ there exists $N' \leq N$, partitions $X = \bigsqcup_{i \in [N']} A^i$ with $A^i \in \mathcal{F}^{E,N}_{X}$, $Y = \bigsqcup_{j \in [N']} B^j$ with $B^j \in \mathcal{F}^{E,N}_{Y}$ and $Z = \bigsqcup_{k \in [N']} C^k$ with $C^k \in \mathcal{F}^{E,N}_{Z}$, and $\Sigma_{X,Y}, \Sigma_{X,Z}, \Sigma_{Y,Z} \subseteq [N']^2$ so that:
\begin{enumerate}
\item $\mu_{X \times Y} \left( \bigsqcup_{(i,j) \notin \Sigma_{X,Y}} A^i \times B^j \right) < \varepsilon, \mu_{X\times Z} \left( \bigsqcup_{(i,j) \notin \Sigma_{X,Z}} A^i \times C^k \right) < \varepsilon,
	\mu_{Y \times Z} \left( \bigsqcup_{(j,k) \notin \Sigma_{Y,Z}} B^j \times C^k \right) < \varepsilon$;
	\item for all  
$(i,j,k) \in \left( \Sigma_{X,Y} \land \Sigma_{X,Z} \right) \cup \left(\Sigma_{X,Y} \land \Sigma_{Y,Z}  \right) \cup \left( \Sigma_{X,Z} \land \Sigma_{Y,Z} \right) $,
$\frac{\mu_{X \times Y \times Z}\left(E \cap \left( A^i \times B^j \times C^k \right) \right)}{\mu_{X \times Y \times Z}\left(  A^i \times B^j \times C^k \right)} \in \left[0, f(N') \right) \cup \left( 1 - f(N'), 1\right]$.
\end{enumerate}
\item (Slice-wise stable regularity) For every $\varepsilon >0$  there exists some $N = N(\varepsilon) \in \mathbb{N}$ satisfying the following. For any $H = (E; X,Y,Z) \in \mathcal{H}$ there exists $N' \leq N$ and partitions 
$X \times Y = \bigsqcup_{0 \leq i < N'} A^i$, $X \times Z = \bigsqcup_{0 \leq j< N'}B^j$ and $Y \times Z = \bigsqcup_{0 \leq k < N'} D^k$   so that:
	\begin{enumerate}
	\item $\mu_{X \times Y} \left ( A^0 \right ), \mu_{X \times Z} \left (B^0 \right ), \mu_{Y \times Z} \left ( D^0 \right ) < \varepsilon$;
			\item for each $1 \leq i,j, k < N'$, each of 
	\begin{gather*}
		\frac{ \mu_{X \times Y \times Z} \left ( E \land A^i \land B^j \right)}{\mu_{X \times Y \times Z} \left (  A^i \land B^j \right) },
		\frac{\mu_{X \times Y \times Z}\left ( E \land A^i \land D^k \right )}{\mu_{X \times Y \times Z}\left ( A^i \land D^k \right )}, \frac{\mu_{X \times Y \times Z}\left ( E \land B^j \land D^k \right )}{\mu_{X \times Y \times Z}\left ( B^j \land D^k \right )}
	\end{gather*}
is in  $\left[0, f(N') \right) \cup \left(1 - f(N') , 1 \right]$.
	\end{enumerate}

	\end{enumerate}
\end{theorem}
\begin{proof}
	\textbf{(1) implies (2)} By \L os, if every $H \in \mathcal{H}$ is slice-wise $d$-stable, then $\widetilde{H}$ is also slice-wise $d$-stable, hence Theorem \ref{thm: all three parts perfect rects} applies.
	
	\

	\noindent \textbf{(2) implies (3)} Let $\varepsilon >0$ and $f: \mathbb{N} \to (0,1]$ be given. Assume that (3) fails, and let $\widetilde{H}$ be a non-principal ultraproduct of a sequence of finite hypergraphs witnessing this. 
	Using (2), choose countable partitions $\widetilde{X} \times \widetilde{Y} = \bigsqcup_{i \in \omega} A^i$, $\widetilde{X} \times \widetilde{Z} = \bigsqcup_{j \in \omega}B^j$ and $\widetilde{Y} \times \widetilde{Z} = \bigsqcup_{k \in \omega} D^k$ with  each $A^i = A^{i,X} \times A^{i,Y}, B^j = B^{j,X} \times B^{j,Z}, D^k = D^{k,Y} \times D^{k,Z}$ a rectangle with $A^{i,X},B^{j,X} \in \mathcal{B}^{E}_{\widetilde{X}},  A^{i,Y}, D^{k,Y} \in \mathcal{B}^{E}_{\widetilde{Y}}, B^{j,Z}, D^{k,Z} \in \mathcal{B}^{E}_{\widetilde{Z}}$ so that  for every $i,j, k \in \omega$ we have:
	\item \begin{enumerate}
 	\item[(a)] either $A^i \land B^j \subseteq^0 E$ or $\left( A^i \land B^j \right) \cap E =^0 \emptyset$;
 	\item[(b)] either $A^i \land D^k \subseteq^0 E$ or $\left(A^i \land D^k \right) \cap E =^0 \emptyset$;
 	\item[(c)] either $B^j \land D^k \subseteq^0 E$ or $\left(B^j \land D^k \right) \cap E =^0 \emptyset$.
 \end{enumerate} 
 
\noindent The following is immediate (and similarly for taking subsets of the sides of the rectangles in the other pairs of coordinates):
\begin{claim}\label{cla: subdividing keeps densities}
Given $i \in \omega$, let $A^{i,X,0} \subseteq A^{i,X}$ with $A^{i,X,0} \in \mathcal{B}_{\widetilde{X}}$ 	be arbitrary. Then for every $j,k \in \omega$, (a) and (b) are satisfied with $A^{i}$ replaced by $A^{i,0} := A^{i,X,0} \times A^{i,Y}$.
\end{claim}

Fix arbitrary $0 < \varepsilon' < \varepsilon$.
By countable additivity, we can choose $N \in \mathbb{N}$ so that 
\begin{enumerate}
	\item[(d)] $\mu_{\widetilde{X} \times \widetilde{Y}} \left ( \bigsqcup_{1 \leq i <N} A^i \right ) \geq 1 - \varepsilon' $,  $\mu_{\widetilde{X} \times \widetilde{Z}} \left ( \bigsqcup_{1 \leq j <N} B^j \right ) \geq 1 - \varepsilon' $ and  $\mu_{\widetilde{Y} \times \widetilde{Z}} \left ( \bigsqcup_{1 \leq k <N} D^k \right ) \geq 1 - \varepsilon' $.
\end{enumerate}

Let $\left\{ \bar{A}^{i} : 1 \leq i \leq N'_{X} \right\}, \bar{A}^{i} \in \mathcal{B}^{E}_{\widetilde{X}}$ be the set of atoms of the Boolean algebra generated 
by $\left\{ A^{i,X} : 1 \leq i \leq N\right\} \cup \left\{ B^{j,X} : 1 \leq i \leq N\right\} $ that have positive measure, we have $N'_{X} \leq 2^{2N}$ and $\bigsqcup_{1 \leq i \leq N'_{X}}  \bar{A}^{i} =^0 \widetilde{X}$.

Let $\left\{ \bar{B}^{j} : 1 \leq j \leq N'_{Y} \right\}, \bar{B}^{j} \in \mathcal{B}^{E}_{\widetilde{Y}}$ be the set of atoms of the Boolean algebra generated 
by $\left\{ D^{k,Y} : 1 \leq k \leq N\right\} \cup \left\{ B^{j,X} : 1 \leq j \leq N\right\}$ that have positive measure. Then $N'_{Y} \leq 2^{2N}$ and $\bigsqcup_{1 \leq j \leq N'_{Y}}  \bar{B}^{j} =^0 \widetilde{Y}$.

And let $\left\{ \bar{C}^{k} : 1 \leq k \leq N'_{Z} \right\}, \bar{C}^{k} \in \mathcal{B}^{E}_{\widetilde{Z}}$ be the set of atoms of the Boolean algebra generated 
by $\left\{ B^{j,Z} : 1 \leq j \leq N\right\}  \cup \left\{ D^{k,Z} : 1 \leq k \leq N\right\} $ that have positive measure, we have $N'_{Z} \leq 2^{2N}$ and $\bigsqcup_{1 \leq i \leq N'_{Z}}  \bar{C}^{k} =^0 \widetilde{Z}$. Let 
\begin{gather*}
	\Sigma_{X,Y} := \left\{ (i,j) \in [N'_{X}]\times [N'_{Y}] : \bar{A}^i \times \bar{B}^j  \subseteq A^{i'} \textrm{ for some } i' \in [N] \right\},\\
	\Sigma_{X,Z} := \left\{ (i,k) \in [N'_{X}]\times [N'_{Z}] : \bar{A}^i \times \bar{C}^k  \subseteq B^{j'} \textrm{ for some } j' \in [N] \right\},\\
	\Sigma_{Y,Z} := \left\{ (j,k) \in [N'_{Y}]\times [N'_{Z}] : \bar{B}^j \times \bar{C}^k  \subseteq D^{k'} \textrm{ for some } k' \in [N] \right\}.
\end{gather*}
Note that if $(i,j) \notin \Sigma_{X,Y}$, then $\bar{A}^i \times \bar{B}^j$ is disjoint from the union of all $A^{i'}$, and similarly for the other pairs of coordinates.

By construction and (d) we have 
\begin{gather*}
	\mu_{\widetilde{X} \times \widetilde{Y}} \left( \bigsqcup_{(i,j) \notin \Sigma_{X,Y}} \bar{A}^i \times \bar{B}^j \right) < \varepsilon', \mu_{\widetilde{X} \times \widetilde{Z}} \left( \bigsqcup_{(i,j) \notin \Sigma_{X,Z}} \bar{A}^i \times \bar{C}^k \right) < \varepsilon',\\
	\mu_{\widetilde{Y} \times \widetilde{Z}} \left( \bigsqcup_{(j,k) \notin \Sigma_{Y,Z}} \bar{B}^j \times \bar{C}^k \right) < \varepsilon'.
\end{gather*}
And using Claim \ref{cla: subdividing keeps densities} (and its analogs for various permutations of the coordinates), for every 
$$(i,j,k) \in \left( \Sigma_{X,Y} \land \Sigma_{X,Z} \right) \cup \left(\Sigma_{X,Y} \land \Sigma_{Y,Z}  \right) \cup \left( \Sigma_{X,Z} \land \Sigma_{Y,Z} \right) $$
 we have 
\begin{gather*}
	\mu \left( E \cap \left( \bar{A}^i \times \bar{B}^j \times \bar{C}^k  \right) \right) \in \left\{ 0,  \mu \left( \bar{A}^i \times \bar{B}^j \times \bar{C}^k \right) \right\}.
\end{gather*}

Using transfer from Section \ref{sec: transfer for stable regularities}, similarly to the proofs of (3)$\Rightarrow$(1) of Theorem \ref{thm: metastable general equivalences} and of Proposition \ref{prop: one way transfer for strong stable reg}, we get the finitary statement in (6).

\ 

\noindent	\textbf{(3) implies (4)}  Given $\varepsilon >0$, let $N$ be as given by (3) for $\varepsilon$ and $f(n) := \varepsilon$. Then, given $H = (E; X,Y,Z) \in \mathcal{H}$, let $N' \leq N$ and $A^i \subseteq X, B^j \subseteq Y, C^k \subseteq Z$ for $i,j,k \in [N']$ and $\Sigma_{X,Y}, \Sigma_{X,Z}, \Sigma_{Y,Z} \subseteq [N']^2$ satisfy (3). Let $\bar{A}^{(i,j)} := A^i \times B^j, \bar{B}^{(i,k)} := A^i \times C^k, \bar{D}^{(j,k)} := B^j \times C^k$.
Then the partitions 
\begin{gather*}
	\left\{ \bar{A}^{(i,j)} : (i,j) \in \Sigma_{X,Y} \right\} \cup \left\{ \left(X \times Y \right) \setminus   \bigsqcup_{(i,j) \in \Sigma_{X,Y}}\bar{A}^{(i,j)}\right\} \textrm{ of } X \times Y,\\
	\left\{ \bar{B}^{(i,k)} : (i,k) \in \Sigma_{X,Z} \right\} \cup \left\{ \left(X \times Z \right) \setminus   \bigsqcup_{(i,k) \in \Sigma_{X,Z}}\bar{B}^{(i,k)}\right\} \textrm{ of } X \times Z,\\
	\left\{ \bar{D}^{(j,k)} : (j,k) \in \Sigma_{Y,Z} \right\} \cup \left\{ \left(Y \times Z \right) \setminus   \bigsqcup_{(j,k) \in \Sigma_{Y,Z}}\bar{D}^{(j,k)}\right\} \textrm{ of } Y \times Z
\end{gather*}
show that $N^2 + 1$ satisfies (4) with respect to $\varepsilon$.

\

	
	\textbf{(4) implies (1)} Assume that (1) fails, but (4) holds. Let $N \in \mathbb{N}$ be arbitrary. For every $d \in \mathbb{N}$, let $H_d = \left(E_d; X_d, Y_d, Z_d \right)$ with $X_d = Y_d := [d]$ and $Z_d := [1]$ be the $3$-partite $3$-hypergraph defined by $(x,y,1) \in E_d \iff x < y$. As  (1) fails, using that $\mathcal{H}$ is hereditarily closed (and possibly permuting the coordinates $X,Y,Z$), we have $H_d \in \mathcal{H}$ for all $d \in \mathbb{N}$ (more precisely, an isomorphic copy of it). Let $\varepsilon >0$ be arbitrary, and let $N$ be as given by (4). Then we have some $B^j \subseteq X_d \times Z_d, D^k \subseteq Y_d \times Z_d$ for $ 0 \leq j, k < N' \leq N$ satisfying (a) and (b) in (4). 
	 As $Z_d = [1]$ is a singleton, we have $ B^j = B^{j,X} \times [1]$ for some $B^{j,X} \subseteq X_d$  and $ D^k = D^{k,Y} \times [1]$ for some $D^{k,Y} \subseteq Y_d$ so that:
	 \begin{itemize}
	 	\item $\mu_{X_d} \left(  \bigsqcup_{1 \leq j < N'} B^{j,X} \right) \geq 1 - \varepsilon $,
	 	\item $\mu_{Y_d} \left(  \bigsqcup_{1 \leq k < N'} D^{k,Y} \right) \geq 1 - \varepsilon$,
	 	\item for any $1 \leq j,k < N'$,
	 	\begin{gather*}
	 		\frac{\mu_{X_d \times Y_d}\left ( (E_d)_1 \cap \left( B^{j,X} \times D^{k,Y} \right) \right )}{\mu_{X_d \times Y_d}\left (  B^{j,X} \times D^{k,Y} \right ) } \in \left[0, \varepsilon \right) \cup \left(1 - \varepsilon ;1 \right],
	 	\end{gather*}
	 \end{itemize}
	 (where $\left( E_d \right)_1 \subseteq X_d \times Y_d$ is the fiber of $E_d$ at $1 \in Z_d$).
	 
 But as $\left(E_d \right)_1$ is the half-graph with both sides of size $d$, and $\varepsilon >0$ was arbitrary, using Remark  \ref{rem: relaxed strong stable reg implies stable reg} we have demonstrated that the family of half-graphs satisfies stable regularity -- contradicting 
 Lemma \ref{lem: half-graphs fail stab reg}.
\end{proof}

\begin{cor}\label{cor: slice-wise stab binary vdisc3}
	In particular, if $\mathcal{H}$ is a family of (non-partite) 3-hypergraphs, i.e.~every $H \in \mathcal{H}$ is of the form $H = (E, V)$ with symmetric and irreflexive $E \subseteq V^3$, applying Theorem \ref{thm: finitary slice-wise reg}(3) viewing $E$ as a 3-partite 3-hypergraph on $V \times V \times V$ and taking $\mathcal{P}$ to be the partition of $V$ obtained by intersecting the partitions $A^i, B^j, C^k$ of $V$ and $\Sigma$ to be the union of $\Sigma_{X,Y}, \Sigma_{X,Z}, \Sigma_{Y,Z}$, it follows that $\mathcal{H}$ satisfies binary $\vdisc_3$ error, in the terminology of \cite[Definition 2.10(2)]{terry2021irregular}, with a stronger condition on the regularity of good triples allowing the decay function $f$. On the other hand, as the proof of (3) implies (4) of Theorem \ref{thm: finitary slice-wise reg} demonstrates, binary $\vdisc_3$ error implies that $\mathcal{H}$ satisfies the condition in Theorem \ref{thm: finitary slice-wise reg}. Hence a hereditary family of 3-hypergraphs satisfies binary $\vdisc_3$ error if and only if there is some $d \in \mathbb{N}$ so that all $H \in \mathcal{H}$ are slice-wise $d$-stable.  This answers \cite[Problem 2.14(2)]{terry2021irregular}.
\end{cor}

\subsection{A slice-wise stable $3$-hypergraph not satisfying the stable regularity lemma}\label{sec: counterexample}

We might hope to extract more from having all three directions of slice-wise stability---for example, we might hope this is enough to give us at least stable regularity.

In this section we provide an example of a (hereditarily closed) family of finite $3$-hypergraphs $\mathcal{H}$ so that  all $H \in \mathcal{H}$ are slice-wise $7$-stable, but $\mathcal{H}$ does not satisfy  stable regularity (see Definition \ref{def: strong stab reg}(3)). In particular, this refutes \cite[Conjecture 2.42]{terry2021irregular}.

\begin{definition}
	We let $X := \{0,1,2 \}^{\omega}$. For $x = \left(x(0), x(1), \ldots \right)$, we let $v(x) := \min \{n \in \omega : x(n) \neq 0 \}$. Define $E \subseteq X^3$ via: $(x,y,z) \in E$ if for the first $n \in \omega$ so that $\left \lvert \left\{ x(n), y(n), z(n) \right\} \right\rvert > 1$, we have $\left \lvert \left\{ x(n), y(n), z(n) \right\} \right\rvert  = 3$ (i.e.~at the first coordinate where not all three of the sequences have the same entry, all three have different entries). 
\end{definition}

\begin{remark}\label{rem: symmetries of the counterexample}
	\begin{enumerate}
		\item $E$ is symmetric under any permutation of its variables.
		\item For each $n \in \omega$, let $f_n$ be a permutation of $\{0,1,2\}$. Then $f: X \to X$ defined by $\left(x(0), x(1), \ldots \right) \mapsto \left( f_0 ( x(0)), f_1 (x(1)), \ldots \right)$ is a permutation of $X$, and for any $x,y,z \in X$ we have $(x,y,z) \in E \iff \left( f(x), f(y), f(z) \right) \in E$. In particular, for any two $x,y \in X$, there exists an automorphism of $(X,E)$ sending $x$ to $y$.
	\end{enumerate}
\end{remark}

\begin{lemma}\label{lem: the counterex is slice-wise stable}
	$E$ is slice-wise stable.
\end{lemma}
\begin{proof}
	Fix $x \in X$.
By Remark \ref{rem: symmetries of the counterexample}, we may assume $x(n) = 0$ for all $n \in \omega$.
	Assume that $y_i, z_j \in X$ for $0 \leq i,j < 3$ is a ladder of height $7$  for $E_x$ (see Definition \ref{def: ladders}).
	
	Observe that for any $0 \leq i<j < 7$, since $(x,y_i,z_j)\in E$, we must have $v(y_i) = v(z_j)$. Since this holds for all $0 \leq i < j < 3$, there exists some $v \in \{0,1,2\}$ so that $v = v(y_i) = v(z_j)$ for all $0 \leq i < 5, 1 \leq j < 6$.
	
	But then for any such pair $(i,j)$, $(x,y_i, z_j) \in E \iff y_i(v) \neq  z_j(v)$. By the pigeonhole principle we can choose $ 0 \leq i_1 < i_2 < 5$ with $y_{i_1}(v) = y_{i_2}(v)$ and $i_2 - i_1 \geq 2$.
	But then for any $j$ with $i_1 < j < i_2 $ we have $E(x,y_{i_1},z_j) \iff  y_{i_1}(v) = z(v) \iff y_{i_2} = z(v) \iff E(x,y_{i_2},z_j)$ --- a contradiction.
		\end{proof}
	
	\begin{remark}
		$E$ is not partition-wise stable.
		
	\noindent 	Indeed, let $x_i := 0^i 1^{\infty}$, $y_j := 0^j 1^{\infty}$ and $z_j := 0^j 2^{\infty}$. Now if $i \leq j$, then $(x_i; y_j, z_j) \notin E$ because  all positions before $i+1$ are $(0,0,0)$ and the $(i+1)$st position is $(1,0,0)$ or $(1,1,2)$.
	And if $j<i$ then $(x_i,y_j,z_j)\in E$ because all positions before $j+1$ are $(0,0,0)$ and the $(j+1)$st position is $(0,1,2)$.
	\end{remark}
	
\begin{theorem}\label{thm: slice-wise counterexample}
		 $E$ does not satisfy the stable hypergraph regularity lemma.
	\end{theorem}

\begin{proof}
For $m \in \mathbb{N}$, let $X_{m} := 3^{m}$ and $E_{m} := E \restriction \left( X_m \right)^3$. Let $\mu$ be the uniform measure on $X = 3^m$. We will show that for any $T \in \mathbb{N}$, if $m$ is sufficiently large with respect to $T$, then there is no partition $\mathcal{P} = \{P_i : i \in [T] \}$ of $X_m$ so that \emph{every} box $P_i \times P_j \times P_k$ is $\left( \frac{1}{10^7} \right)$-regular with respect to $E_{m}$ (and $\mu$). As $(X_m, E_m)$ is an induced subgraph of $(X, E)$ for every $m$, this implies in particular that there is no finite partition of $X$ with all boxes  $\left( \frac{1}{10^7} \right)$-regular with respect to $E$ (and the uniform measure on $X$).

Fix $T \in \mathbb{N}$ and $\varepsilon \in \mathbb{R}_{>0}$ with $0 < \varepsilon < \frac{1}{9}$, e.g.~we could take $\varepsilon := \frac{1}{10}$ (but we prefer to keep $\varepsilon$ in order to avoid too many numerics in the computations below). Let $m = m(T, \varepsilon) \in \mathbb{N}$ be sufficiently large, to be determined later, and let $\mathcal{P} = \{P_i : i \in [T] \}$ be a partition of $X_m$. For $s \in 3^{\leq m}$, we let $[s] := \left \{ t \in 3^{m} : s \trianglelefteq t \right\}$. It is easy to verify (see below) that no partition of $X_m$ using sets of the form $[s]$ can satisfy the required regularity, and we aim to approximate an arbitrary partition by one of this form.

First we show that, in order to prove the theorem, it suffices to show that there is no partition $\mathcal{P}$  of $X_m$ satisfying 
\begin{gather}
	(1- \varepsilon) \mu(P_i) \geq \frac{1}{3^{m-2}} \textrm{ for all } i \in [T] \label{eq: counterexample -1}
\end{gather}
with all boxes $\left(\frac{1}{10^6}\right)$-regular, assuming $m = m(T,\varepsilon)$ is sufficiently large (such that $\frac{T}{3^{m-2} (1 - \varepsilon)} < \frac{1}{10^6 T}$).

For suppose we are given a partition $\mathcal{P}$ with all boxes $\frac{1}{10^7}$-regular.  Let $U := \left\{ i \in [T] : \mu(P_i) < \frac{1}{3^{m-2} (1 - \varepsilon)} \right\}$. We have $\sum_{i \in U} \mu(P_i) \leq \frac{T}{3^{m-2} (1 - \varepsilon)}$. Taking $m = m(\varepsilon, T, \frac{1}{10^7})$ sufficiently large, we can assume $\frac{T}{3^{m-2} (1 - \varepsilon)} < \frac{1/10^7}{T}$.  We must have $\mu(P_i) \geq \frac{1}{T}$ for at least one $i \in [T]$, without loss of generality for $i = 0$. Then let $P'_0 := P_0 \cup \bigcup_{i \in U} P_i$ and $P'_i := P_i$ for $i \in [T] \setminus \{0\}$, and consider the new partition $\mathcal{P}' := \left\{ P'_i : i \in [T] \setminus  U \right\}$ of $X_m$. Note that 
\begin{gather}
	\mu \left( P'_i \setminus P_i \right) \leq \frac{1}{10^7} \mu (P_i) \textrm{ for every } i \in [T] \setminus U. \label{eq: counterexample -3}
\end{gather}
Then $\mathcal{P}'$ satisfies \eqref{eq: counterexample -1}, and all boxes in $\mathcal{P}'$ are $4\frac{1}{10^7}$-regular, and therefore $\frac{1}{10^6}$-regular. Indeed, for any $i, j, k \in [T] \setminus  U $ so that $\mu \left(E \cap \left(P_i \times P_j \times P_k \right) \right) < \frac{1}{10^7} \mu(P_i) \mu(P_j) \mu(P_k)$ we have
\begin{gather*}
	\mu \left( E \cap \left( P'_i \times P'_j \times P'_k \right) \right) 
	\leq \frac{1}{10^7} \mu(P_i) \mu(P_j) \mu(P_k) + \\ 
	\mu \left( P'_i \setminus P_i \right) \mu(P'_j) \mu(P'_k) + \mu \left( P'_i \right) \mu(P'_j \setminus P_j) \mu(P'_k) +  \mu \left( P'_i \right) \mu(P'_j) \mu(P'_k  \setminus P_k)\\
	\stackrel{\eqref{eq: counterexample -3}}{\leq} \frac{1}{10^7} \mu(P_i) \mu(P_j) \mu(P_k) + 3 \frac{1}{10^7} \mu(P_i) \mu(P_j) \mu(P_k)\\
	\leq 4 \frac{1}{10^7} \mu(P'_i) \mu(P'_j) \mu(P'_k),
\end{gather*}
and similarly for any $i, j, k \in [T] \setminus  U $ so that $\mu \left(\left(P_i \times P_j \times P_k \right)  \setminus E\right) < \delta \mu(P_i) \mu(P_j) \mu(P_k)$. Therefore it suffices to show that $\mathcal{P}'$ cannot exist.

So consider $\mathcal{P}$ satisfying \eqref{eq: counterexample -1}.	For each $P_i \in \mathcal{P}$, there is some $s_i \in 3^{\leq m}$ at which it ``splits'', i.e.~so that 
	\begin{gather}
	\frac{\mu(P_i\cap [s_i]) } {\mu(P_i)} \geq 1 - \varepsilon,	\textrm{but } \label{eq: counterexample 0} \\	\frac{\mu(P_i\cap [s_i^\frown b])} {\mu(P_i)} < 1 - \varepsilon \textrm{ for each } b \in \{ 0, 1, 2\}.\label{eq: counterexample 00}
	\end{gather}
Indeed, for any $s \in 3^{\leq m }$, $\mu([s]) = \frac{1}{3^{|s|}}$, so using \eqref{eq: counterexample -1} we can find such $s_i$ satisfying
\begin{gather}
	s_i \in 3^{\leq (m-2)} \textrm{ for all }i \in [T]. \label{eq: counterexample -2}
\end{gather}	
	  As  $\mu \left( P_i \cap [s_i] \right) \geq (1-\varepsilon) \mu(P_i)$ for every $i \in [T]$ and the $P_i$'s cover $X$, we have $\mu \left( \bigcup_{i \in [T]} \left( P_i \cap [s_i] \right)  \right) \geq 1 - \varepsilon$, so in particular the $[s_i]$'s also cover $X$, up to $\varepsilon$.
		  
 However, a priori $\mathcal{S} := ([s_i] : i \in [T])$ need not be (close to) a partition since  we might have inclusions between the sets $[s_i]$ for different $i$.  We will remedy this as follows. Let $S_1 := \left\{ s_i : i \in [T] \right\}$ and let $s'_1 \in S_1$  be arbitrary so that $|s'_{1}|$ is minimal. Assuming that $S_j$ and $s'_{j}$ are already defined, we let $S_{j+1} := S_j \setminus \left\{ s \in S_j : s'_j \trianglelefteq s  \right\}  $ and, assuming $S_{j+1} \neq \emptyset$ let $s'_{j+1} \in S_{j+1}$ be arbitrary so that $|s'_{j+1}|$ is minimal among all elements of $S_{j+1}$. We must get some $T' \in [T]$ so that $S_{T' + 1}$ is empty. Then by construction:
 \begin{enumerate}
 	\item $s'_{1}, \ldots, s'_{T'} \in 3^{\leq m}$ are pairwise $\trianglelefteq$-incomparable --- hence the sets $[s'_{1}], \ldots, [s'_{T'}]$ are pairwise  disjoint;
 	\item for every $i \in [T]$ there is a unique $i' \in [T']$ so that $s'_{i'} \trianglelefteq s_i$ --- hence $[s_i] \subseteq [s'_{i'}]$.
 \end{enumerate} 
 
 By (2) we then have: for every $i \in [T]$ there is some $i' \in [T']$ so that $\mu \left( P_i \cap [s'_{i'}] \right) \geq (1-\varepsilon) \mu(P_i)$, and as $\varepsilon < \frac{1}{2}$ and the $[s'_{j}]$'s are pairwise disjoint, such an $i'$ is unique. Then we have
 \begin{gather}
 	\sum_{i \in [T]} \mu \left( P_i \cap [s'_{i'}] \right) \geq (1 - \varepsilon) \sum_{i \in [T]}\mu(P_i) = 1 - \varepsilon. \label{eq: counterex 1}
 \end{gather}
 
This tells us that $\mathcal{S}' := \left\{ [s'_j] : j \in [T'] \right\}$ is close to being a partition of $X$, from which we will now deduce that some set in $\mathcal{S}'$ is (almost) covered by those $P_i$'s that are (almost) contained in it, as follows.
For $j \in [T']$, let 
\begin{gather*}
	Q_j := \left\{ i \in [T] : i'=j \right\} =  \left\{ i \in [T] : s'_j \trianglelefteq s_i \right\} \\
	= \left\{ i \in [T] :  \mu \left( P_i \cap [s'_{j}] \right) \geq (1-\varepsilon) \mu(P_i) \right\}.
\end{gather*}
By the above, $\left( Q_j : j \in [T'] \right)$ is a partition of $[T]$. 
Assume towards contradiction that for every $j \in [T']$, 
\begin{gather}
	\mu \left( \bigsqcup_{i \in Q_j} \left( [s'_{j}] \cap P_i \right)\right) < (1 - \varepsilon) \mu \left( [s'_j] \right). \label{eq: counterex 2}
\end{gather}
Then we have
\begin{gather*}
	1 - \varepsilon  \stackrel{\eqref{eq: counterex 1}}{\leq} \sum_{i \in [T]} \mu \left( P_i \cap [s'_{i'}] \right)  \stackrel{(\textrm{as } i \in Q_{i'})}{\leq} \sum_{j \in [T']} \sum_{i \in Q_j} \mu \left( [s'_{j}] \cap P_i \right) \\\stackrel{\eqref{eq: counterex 2}}{<} \sum_{j \in [T']}(1 - \varepsilon) \mu \left( [s'_j] \right)
	 = (1 - \varepsilon) \sum_{j \in [T']} \mu \left( [s'_j] \right) \leq 1 - \varepsilon,
\end{gather*}
a contradiction. 

So from now on we fix $j \in [T']$ so that  \begin{gather}
	\sum_{i \in Q_j} \mu \left( [s'_{j}] \cap P_i \right) \geq (1 - \varepsilon) \mu \left( [s'_j] \right). \label{eq: counterex 3}
\end{gather}
 This implies that any sufficiently large subset of $[s'_j]$ must have large intersection with some $P_i \in Q_j$:

\begin{claim}\label{cla: counterexample 1}
	For every $d \in \mathbb{N}$ so that  $ \frac{1}{3^d} > \varepsilon$ and $|s'_j|+d \leq m$ and every $t \in 3^{d}$ there exists some $i \in Q_j$ so that 
	\begin{gather*}
		\mu \left( \left[\left(s'_j\right) ^\frown  t \right] \cap P_i \right) \geq \frac{1}{3^d} (1 - 3^d \varepsilon) ( 1 - \varepsilon) \mu( P_i).
	\end{gather*}
\end{claim}
\begin{claimproof}
Note that $\mu \left(\left[\left(s'_j\right) ^\frown  t \right] \right) = \frac{1}{3^d} \mu \left(\left[s'_j \right] \right)$. We have 
\begin{gather*}
	(1 - \varepsilon) \mu ([s'_j]) \stackrel{\eqref{eq: counterex 3}}{\leq} \sum_{i \in Q_j} \mu \left( [s'_{j}] \cap P_i \right) \\
	= \sum_{i \in Q_j} \mu \left( \left[ (s'_j) ^\frown t \right] \cap P_i \right) + \sum_{i \in Q_j} \mu \left( \left( \left[s'_j \right] \setminus \left[ (s'_j) ^\frown t \right] \right)  \cap P_i \right)\\
	\leq \sum_{i \in Q_j} \mu \left( \left[ (s'_j) ^\frown t \right] \cap P_i \right) +  \mu \left(  \left[s'_j \right] \setminus \left[ (s'_j) ^\frown t \right]   \right)\\
	\leq \sum_{i \in Q_j} \mu \left( \left[ (s'_j) ^\frown t \right] \cap P_i \right) + \left( 1 - \frac{1}{3^d} \right) \mu \left(\left[s'_j \right] \right),
\end{gather*}
so
\begin{gather}
	\sum_{i \in Q_j} \mu \left( \left[ (s'_j) ^\frown t \right] \cap P_i \right)  \geq \left( \frac{1}{3^d} - \varepsilon \right) \mu \left(\left[s'_j \right] \right). \label{eq: counterex 4}
\end{gather}

Assume towards contradiction that for every $i \in Q_j$, 
	\begin{gather}
		\mu \left( P_i \cap [(s'_j) ^\frown t] \right) < \frac{1}{3^d} (1 - 3^d \varepsilon) \mu \left(\left[s'_j \right]  \cap P_i \right). \label{eq: counterex 5}
	\end{gather}
Then
\begin{gather*}
	 \left( \frac{1}{3^d} - \varepsilon \right) \mu \left(\left[s'_j \right] \right) \stackrel{\eqref{eq: counterex 4}}{\leq} \sum_{i \in Q_j} \mu \left( \left[ (s'_j) ^\frown t \right] \cap P_i \right) \\
	 \stackrel{\eqref{eq: counterex 5}}{<} \sum_{i \in Q_j} \frac{1}{3^d} (1 - 3^d \varepsilon) \mu \left(\left[s'_j \right]  \cap P_i \right) \\
	 \leq   \frac{1}{3^d} (1 - 3^d \varepsilon) \mu \left(\left[s'_j \right] \right) = \left( \frac{1}{3^d} - \varepsilon \right) \mu \left(\left[s'_j \right] \right),
\end{gather*}
a contradiction.
Hence for some $i \in Q_j$ we have
\begin{gather*}
	\mu \left( P_i \cap [(s'_j) ^\frown t] \right) \geq \frac{1}{3^d} (1 - 3^d \varepsilon) \mu \left(\left[s'_j \right]  \cap P_i \right) \stackrel{(i \in Q_j)}{\geq}\frac{1}{3^d} (1 - 3^d \varepsilon) (1 - \varepsilon) \mu( P_i). 
\end{gather*}
\end{claimproof}

Let $i^* \in [T]$ be such that $s'_j = s_{i^*}$. By \eqref{eq: counterexample 0}, there exist some $b,b' \in \{0,1,2\}$ so that 
\begin{gather}
	\mu \left( [s_{i^*} ^\frown b ^\frown  b'] \cap P_{i^*} \right)  \geq \frac{1}{9} (1 - \varepsilon)  {\mu(P_{i^*})}. \label{eq: counterex 6}
\end{gather} 
 Without loss of generality let $b = b' = 0$. By Claim \ref{cla: counterexample 1} with $d=2$ (as $\varepsilon < \frac{1}{9}$ and $s_{i^*} \in 3^{\leq (m-2)}$ by \eqref{eq: counterexample -2}) we can find some $j^*,k^* \in [T]$ so that 
\begin{gather}
	\mu \left( \left[s_{i^*} ^\frown  01 \right] \cap P_{j^*} \right) \geq \frac{1}{9} (1 - 9 \varepsilon) ( 1 - \varepsilon) \mu( P_{j^*}) \textrm{ and } \label{eq: counterex 7}\\
	\mu \left( \left[s_{i^*} ^\frown  02 \right] \cap P_{k^*} \right) \geq \frac{1}{9} (1 - 9 \varepsilon) ( 1 - \varepsilon) \mu( P_{k^*}). \label{eq: counterex 8}
\end{gather}

Finally consider the box $P_{i^*} \times P_{j^*} \times P_{k^*}$. On the one hand, by definition of $E$ we have $ \left[s_{i^*} ^\frown  00 \right] \times \left[s_{i^*} ^\frown  01 \right] \times \left[s_{i^*} ^\frown  02 \right] \subseteq E$, hence  
\begin{gather}
	\mu \left(E \cap  \left( P_{i^*} \times P_{j^*} \times P_{k^*} \right) \right) \nonumber\\
	\geq \mu \left( \left( [s_{i^*} ^\frown 00] \cap P_{i^*} \right) \times \left( \left[s_{i^*} ^\frown  01 \right] \cap P_{j^*} \right)  \times \left( \left[s_{i^*} ^\frown  02 \right] \cap P_{k^*} \right) \right) \nonumber \\\stackrel{\eqref{eq: counterex 6}, \eqref{eq: counterex 7}, \eqref{eq: counterex 8}}{\geq}
 \left(  \frac{1}{9} (1 - 9 \varepsilon) (1 - \varepsilon) \right)^{3}   \mu(P_{i^*})\mu(P_{j^*}) \mu(P_{k^*}). \label{eq: counterex 10}
  \end{gather}
  
  On the other hand, let $P_{i^*}^{-} := P_{i^*} \setminus [s_{i^*} ^\frown 0]$. By \eqref{eq: counterexample 00} we have 
  \begin{gather}
  	 \mu \left( P_{i^*}^{-} \right) \geq \varepsilon \mu( P_{i^*}). \label{eq: counterex 9}
  \end{gather}
  And by definition of $E$ we have $\left( P^-_{i^*} \times [s_{i^*} ^\frown 01] \times [s_{i^*} ^\frown 02] \right) \cap E = \emptyset$. So 
\begin{gather}
	\mu \left(\left( P_{i^*} \times P_{j^*} \times P_{k^*} \right)  \setminus E \right) \nonumber  \\
	\geq \mu \left(P_{i^*}^{-} \times \left( \left[s_{i^*} ^\frown  01 \right] \cap P_{j^*} \right) \times \left( \left[s_{i^*} ^\frown  02 \right] \cap P_{k^*} \right) \right) \nonumber \\
	\stackrel{\eqref{eq: counterex 9}, \eqref{eq: counterex 7}, \eqref{eq: counterex 8}}{\geq} \varepsilon  \left(  \frac{1}{9} (1 - 9 \varepsilon) (1 - \varepsilon) \right)^{2}   \mu(P_{i^*})\mu(P_{j^*}) \mu(P_{k^*}). \label{eq: counterex 11}
\end{gather}

So if we take $\varepsilon := \frac{1}{10}$, by \eqref{eq: counterex 10} and \eqref{eq: counterex 11} we get that the box $P_{i^*} \times P_{j^*} \times P_{k^*}$ is not $\frac{1}{10^6}$-regular with respect to $E$.
%
\end{proof}

\section{One direction of partition-wise stability and the opposite of slice-wise stability}\label{sec: one dir partition-wise one slice-wise}

\begin{theorem}\label{thm: one dir stab one slice-wise stab}
	Assume that $E \in \mathcal{B}_{X \times Y \times Z}$ is $\mu$-stable  viewed as a binary relation between $X \times Y$ and $Z$, and the slices $E_z \in \mathcal{B}_{X \times Y}$ are $\mu$-stable for almost all $z \in Z$. Then for every $\varepsilon > 0$ there exist finite collections of pairwise-disjoint sets $X_i \in \mathcal{B}^{E}_{X}$ for $i \in I$, pairwise-disjoint sets $Y_j \in \mathcal{B}^{E}_{Y}$ for $j \in J$ and pairwise-disjoint sets $Z_k \in \mathcal{B}^{E}_{Z}$ for $k \in K$ so that:
	\begin{enumerate}
	\item $\mu \left(  \bigsqcup_{i \in I} X_i \right) \geq 1 - \varepsilon \mu(X_{i^*})$ for some $i^* \in I$, $\mu \left(\bigsqcup_{j \in J} Y_j \right) \geq 1 - \varepsilon \mu(Y_{j^*})$ for some $j^* \in J$, and $\mu \left( \bigsqcup_{k \in K} Z_k \right) \geq 1 - \varepsilon \mu(Z_{k^*})$ for some $k^* \in K$;
		\item $\frac{\mu \left( E \cap \left( X_i \times Y_j \times Z_k \right)\right)}{\mu \left( X_i \times Y_j \times Z_k \right) } \in \{0,1\}$ for all $(i,j,k) \in I \times J \times K$.
	\end{enumerate}
\end{theorem}
\begin{proof}
	As the binary relation $E \subseteq \left(X \times Y \right) \times Z$ is $\mu$-stable, by Lemma \ref{lem: perf part stable} there exists a partition $Z = \bigsqcup_{k \in \omega} Z_k$ with $Z_k \in \mathcal{B}^{E}_{Z}$ perfect. 
	
\begin{claim}\label{cla: slice plus stab proof 1}
For every $k \in \omega$, for almost every $z \in Z_k$ we have:  $E_{z} =^0 E_{z'}$ for almost every $z' \in Z_k$.	
\end{claim}
\begin{claimproof}
Fix $k \in \omega$, and	consider the set $S^k \in \mathcal{B}^{E}_{X \times Y \times Z \times Z}$ defined by
$$S^k := \left\{ \left( x,y,z,z' \right) \in  X \times Y \times Z_k \times Z_k : (x,y,z) \in E \iff (x,y,z') \in E \right\}.$$
As $Z_k$ is perfect, for almost every $(x,y) \in X \times Y$ we have either $Z_k \subseteq^0 E_{(x,y)}$ or $Z_k \cap E_{(x,y)} =^0 \emptyset$, hence in either case $S^k_{(x,y)} =^0 Z_k \times Z_k$. It follows by Fubini that $S^k =^0 X \times Y \times Z_k \times Z_k$, hence by Fubini again we have that for almost every $z \in Z_k $, 
$$Z_k =^0 \left\{z' \in Z_k : S^k_{(z,z')} =^0 X \times Y \right\} = \left\{ z' \in Z^k : E_z =^0 E_{z'}\right\}.$$
\end{claimproof}
Now let $\varepsilon >0$ be fixed. By countable additivity we can choose $K \in \omega$ large enough so that $\mu \left( \bigcup_{k < K} Z_k\right) \geq 1 - \varepsilon \mu \left(Z_0 \right)$. Given $k < K$, by assumption $E_z$ is $\mu$-stable for almost all $z \in Z_k$. Using Claim \ref{cla: slice plus stab proof 1}, we can thus choose $z_k \in Z_k$ so that $E_{z_k} \in \mathcal{B}_{X \times Y}$ is $\mu$-stable and $E_{z_k} =^0 E_{z}$ for almost every $z \in Z_k$. By Proposition \ref{prop: stable graph reg} there exist partitions $X = \bigsqcup_{i \in \omega} X^k_{i}$ with $X^k_i \in \mathcal{B}^{E}_{X}$ and $Y = \bigsqcup_{j \in \omega} Y^{k}_{j}$ with $Y^k_j \in \mathcal{B}^{E}_{Y}$ into perfect sets for $E_{z_k}$. But as $E_{z_k} =^0 E_z$ for almost all $z \in Z_k$, it follows by Fubini that for every $(i,j) \in \omega^2$, either $X^k_i \times Y^k_j \times Z_k \subseteq^0 E$ or  $\left( X^k_i \times Y^k_j \times Z_k \right) \cap  E =^0 \emptyset$. 


For every $\bar{i} = \left(i_0, \ldots, i_{K-1} \right) \in \omega^K$, let $X_{\bar{i}} := \bigcap_{k < K} X^k_{i_k} \in \mathcal{B}^{E}_{X}$ and $Y_{\bar{i}} := \bigcap_{k < K} Y^k_{i_k} \in \mathcal{B}^{E}_{Y}$. Then $\left(X_{\bar{i}} : \bar{i} \in \omega^K \right)$ is a countable partition of $X$, $\left(Y_{\bar{j}} : \bar{j} \in \omega^K \right)$  is a countable partition of $Y$, and for every $\bar{i}, \bar{j} \in \omega^k, k \in K$ we have $X_{\bar{i}} \times Y_{\bar{j}} \subseteq X_{i_k} \times Y_{j_k}$, hence either $X_{\bar{i}} \times Y_{\bar{j}} \times Z_k \subseteq^0 E$ or $\left( X_{\bar{i}} \times Y_{\bar{j}} \times Z_k \right) \cap E =^0 \emptyset$. Let $I' := \left\{\bar{i} \in \omega^K : \mu \left(X_{\bar{i}} \right) > 0 \right\}$, $J' := \left\{\bar{j} \in \omega^K : \mu \left(Y_{\bar{j}} \right) > 0 \right\}$. Fix any $\bar{i}^* \in I'$, by countable additivity we can choose finite $I \subseteq I'$ with $\bar{i}^* \in I$ and $\mu \left( \bigcup_{\bar{i} \in I} X_{\bar{i}} \right) > 1 - \varepsilon \mu \left(X_{\bar{i}^*} \right)$. Similarly, we chose $\bar{j}^* \in J'$ and finite $J \subseteq J'$ with $\bar{j}^* \in J$ and $\mu \left( \bigcup_{\bar{j} \in J} Y_{\bar{j}} \right) > 1 - \varepsilon \mu \left(Y_{\bar{j}^*} \right)$.
\end{proof}

\begin{cor}\label{cor: finitary one dir stable other slice-wise}
Given $d \in \mathbb{N}$, let $\mathcal{H}$ be the family of all finite $3$-partite $3$-hypergraphs such that for every $H = (E; X,Y,Z) \in \mathcal{H}$, $E$ viewed as a binary relation on $X\times Y$ and $Z$ is $d$-stable, and for every $z \in Z$, the binary slice $E_z \subseteq X \times Y$ is $d$-stable.
	Then $\mathcal{H}$ satisfies strong stable regularity (Definition \ref{def: strong stab reg}(2)).
\end{cor}
\begin{proof}
As $d$ is fixed, by \L os theorem and assumption, in any ultraproduct $\widetilde{H} = \left(\widetilde{E}; \widetilde{X}, \widetilde{Y}, \widetilde{Z} \right)$ of hypergraphs from $\mathcal{H}$, $\widetilde{E}$ will be a $d$-stable viewed as a binary relation on $\widetilde{X}$ and $\widetilde{Y} \times \widetilde{Z}$, and for every $z \in \widetilde{Z}$, the binary slice $\widetilde{E}_z \subseteq \widetilde{X} \times \widetilde{Y}$ will be $d$-stable.
The result follows by Theorem \ref{thm: one dir stab one slice-wise stab} applied to the ultraproduct and Proposition \ref{prop: one way transfer for strong stable reg}.
\end{proof}

\begin{remark}
	This applies to Example \ref{ex: x=y<z} in particular.
\end{remark}

\bibliographystyle{alpha}
\bibliography{refs}
\end{document}